\documentclass[12pt]{article}
\usepackage{amsmath,amsfonts,amssymb,amsthm,mathrsfs}
\usepackage[a4paper,vmargin={3.5cm,3.5cm},hmargin={2.5cm,2.5cm}]{geometry}
\usepackage[margin=1cm]{caption}
\usepackage{graphicx,graphics}
\usepackage{epsfig}
\usepackage{latexsym}
\usepackage[utf8]{inputenc}
\usepackage{ae,aecompl}
\usepackage[english]{babel}
\usepackage[colorlinks=true]{hyperref}
\usepackage{enumerate}
\usepackage{bbm}
\usepackage{pxfonts}
\usepackage{dsfont}
\usepackage{tikz}

\usepackage[auto]{contour}
\usetikzlibrary{shapes, angles, decorations.text, patterns, fadings,decorations.pathreplacing}
\pgfdeclarelayer{edgelayer}
\pgfdeclarelayer{nodelayer}
\pgfsetlayers{edgelayer,nodelayer,main}
\tikzstyle{none}=[inner sep=0pt]
\tikzstyle{real}=[circle,fill=black,draw=black, inner sep=1pt]
\tikzstyle{phantom}=[circle,fill=white,draw=black, inner sep=1pt]
\tikzstyle{geoone}=[very thick, orange]
\tikzstyle{geotwo}=[very thick, blue]
\tikzstyle{geothree}=[ultra thick, green!50!black]
\tikzstyle{map}=[red, densely dashed]
\tikzstyle{root}=[red, very thick, ->]
\tikzstyle{georight}=[very thick, green!50!black]
\tikzstyle{geoleft}=[very thick, blue]
\tikzstyle{bla}=[]
\tikzset{
        hatch distance/.store in=\hatchdistance,
        hatch distance=10pt,
        hatch thickness/.store in=\hatchthickness,
        hatch thickness=2pt
    }

    \makeatletter
    \pgfdeclarepatternformonly[\hatchdistance,\hatchthickness]{flexible hatch}
    {\pgfqpoint{0pt}{0pt}}
    {\pgfqpoint{\hatchdistance}{\hatchdistance}}
    {\pgfpoint{\hatchdistance}{\hatchdistance}}
    {
        \pgfsetcolor{\tikz@pattern@color}
        \pgfsetlinewidth{\hatchthickness}
        \pgfpathmoveto{\pgfqpoint{0pt}{0pt}}
        \pgfpathlineto{\pgfqpoint{\hatchdistance}{\hatchdistance}}
        \pgfusepath{stroke}
    }

\date{}

\newtheorem{remark}{Remark}[]

\newtheorem{thm}{Theorem}[section]
\newtheorem{theorem}[thm]{Theorem}
\newtheorem{rem}{Remark}[section]
\newtheorem{df}{Definition}[section]
\newtheorem{prop}[thm]{Proposition}
\newtheorem{proposition}[thm]{Proposition}
\newtheorem{lem}[thm]{Lemma}
\newtheorem{lemma}[thm]{Lemma}
\newtheorem{cor}[thm]{Corollary}

\makeatletter
\newcommand*\dashline{\rotatebox[origin=c]{90}{$\dabar@\dabar@\dabar@$}}
\makeatother

\newcommand{\B}{\mathcal{B}} 
\renewcommand{\P}{\mathbb{P}} 
\newcommand{\E}{\mathbb{E}} 
\newcommand{\DS}{\mathsf{DS}} 
\newcommand{\Q}{\mathcal{Q}} 
\newcommand{\QS}{\widetilde{\mathcal{Q}}} 
\newcommand{\BQ}{\mathcal{Q}^B} 
\newcommand{\BPQ}{\mathcal{Q}^{B, \bullet}} 
\newcommand{\UIHPQ}{\mathcal{H}_\infty} 
\newcommand{\UIHPQS}{\widetilde{\mathcal{H}}_\infty} 
\newcommand{\sQ}{\mathsf{Q}} 
\newcommand{\sQS}{\widetilde{\mathsf{Q}}} 
\newcommand{\TB}{\mathsf{TB}} 
\newcommand{\LT}{\mathsf{LT}} 
\newcommand{\T}{\mathsf{T}} 
\newcommand{\flip}{\leftrightarrow} 
\newcommand{\q}{\mathbf{q}} 
\renewcommand{\S}{\mathbf{p}} 
\renewcommand{\epsilon}{\varepsilon}

\DeclareSymbolFont{extraup}{U}{zavm}{m}{n}
\DeclareMathSymbol{\varheart}{\mathalpha}{extraup}{86}
\DeclareMathSymbol{\vardiamond}{\mathalpha}{extraup}{87}

\makeatletter
\renewcommand*{\@fnsymbol}[1]{\ensuremath{\ifcase#1\or  \vardiamond \or \clubsuit\or \spadesuit\or
   \mathsection\or \mathparagraph\or \|\or **\or \dagger\dagger
   \or \ddagger\ddagger \else\@ctrerr\fi}}
\makeatother

\title{\bf \textsc{Geometry of the Uniform Infinite\\Half-Planar Quadrangulation}}

\author{Alessandra Caraceni\thanks{Scuola Normale Superiore and Université Paris Sud.\hfill  \texttt{alessandra.caraceni@sns.it}} \hspace{12pt} \& \hspace{4pt} Nicolas Curien\thanks{Université Paris Sud.\hfill  \texttt{nicolas.curien@gmail.com}}}
\begin{document}
\maketitle

\abstract{We give a new construction of the uniform infinite half-planar quadrangulation with a general boundary (or UIHPQ), analogous to the construction of the UIPQ presented by Chassaing and Durhuus ~\cite{CD06}, which allows us to perform a detailed study of its geometry. We show that the process of distances to the root vertex read along the boundary contour of the UIHPQ evolves as a particularly simple Markov chain and converges to a pair of independent Bessel processes of dimension $5$ in the scaling limit. We study the ``pencil'' of infinite geodesics issued from the root vertex as in~\cite{CMMinfini}, and prove that it induces a decomposition of the UIHPQ into three independent submaps. We are also able to prove that balls of large radius around the root are on average $7/9$ times as large as those in the UIPQ, both in the UIHPQ and in the UIHPQ with a simple boundary; this fact we use in a companion paper to study self-avoiding walks on large quadrangulations.}

\section{Introduction}

The aim of this paper is to investigate the geometry of large random quadrangulations with a boundary through their infinite local limit, an object named \emph{the Uniform Infinite Quadrangulation of the Half-Plane} (UIHPQ for short). The framework which we draw from is thus the broader probabilistic theory developed around random planar maps, a field which has been very active over the last decade (see~\cite{LeGallICM,MieStFlour}).

The core of this paper consists in an adaptation of methods which were originally developed in~\cite{CS04,CMMinfini,CMboundary, Kri05,LGM10, Men08} in order to study the Uniform Infinite Planar Quadrangulation (UIPQ): we hence briefly summarise progress made so far on the study of the geometry of the UIPQ, so that we may then present corresponding results and conjectures in the case of the UIHPQ.

\paragraph{A brief history of the UIPQ.}
Following pioneering work of Angel \& Schramm~\cite{AS03} on local limits of random triangulations, Krikun~\cite{Kri05} studies the local limit of uniform random quadrangulations of the sphere. Taking $\Q_n$ to be a uniform random rooted quadrangulation of the sphere with $n$ faces, he uses exact enumeration formulas to prove that
$$\Q_n \xrightarrow[n \to \infty]{(d)} \Q_\infty$$ for the local topology (see Section~\ref{quadrangulations}); $\Q_\infty$, commonly referred to as the UIPQ, is a random infinite quadrangulation of the plane with a distinguished oriented edge.

Meanwhile, Chassaing and Durhuus~\cite{CD06} give a ``Schaeffer-type'' construction for an infinite random quadrangulation of the plane; this construction, which relies on a bijection between quadrangulations and certain trees whose vertices bear positive labels (later referred to as \emph{positive labelled trees}), is proved to be equivalent to that of Krikun by M\'enard~\cite{Men08}, and lays the basis for much further work on the geometry of the UIPQ. Le Gall and M\'enard~\cite{LGM10} compute scaling limits for the contour functions coding the infinite positive labelled tree that Chassaing and Durhuus used to build the UIPQ; in particular, they prove~\cite[Theorem 6]{LGM10}: if $ \#[\Q_{\infty}]_{n}$ is the number of vertices within distance $n$ from the root vertex in the UIPQ, then we have
\begin{eqnarray} \label{eq:cvballuipq} \frac{1}{n^4} \# [\Q_{\infty}]_{n} \xrightarrow[n\to\infty]{(d)} \mathcal{V}_{p}  \end{eqnarray} 
for a certain explicit limiting law $ \mathcal{V}_{p}$ (where ``$p$'' stands for ``plane'').

An alternative construction of the UIPQ is given in~\cite{CMMinfini}: it is this time an ``unconstrained'' construction, in the sense that it is based on a Schaeffer-type correspondence which relates pointed, rooted quadrangulations to labelled plane trees with no positivity condition on their labels. While the positive construction of~\cite{CD06} encodes the UIPQ as a labelled infinite tree carrying precise information about distances between the root vertex and other vertices of the map, labels in the ``unconstrained'' tree carry a different set of geometric information, and can in some sense be interpreted as distances to infinity in the corresponding UIPQ, see~\cite[Theorem 2.8]{CMMinfini}. Such a construction is thus well-suited to the study of coalescence properties of geodesic rays to infinity, which is treated in detail in~\cite{CMMinfini}; furthermore, it underlies the proof of the fact that the UIPQ admits a scaling limit in the local Gromov--Hausdorff sense: a scale invariant locally compact random metric space that is homeomorphic to the plane and has Hausdorff dimension $4$, known as \emph{the Brownian plane}, see~\cite{CLGHull,CLGplane}.

\paragraph{Previous results on the UIHPQ.}
A \emph{quadrangulation with a boundary} is a rooted planar map whose faces are all quadrangles, except for the face adjacent to the root edge and lying to its right (called the external face, or outerface), which can be of arbitrary (even) degree (and, like other faces, is not necessarily simple); all of the relevant definitions and enumeration results are given in Section~\ref{quadrangulations}. 

If we denote by $\Q_{n,p}$ a uniformly random quadrangulation with $n$ inner quadrangular faces and a boundary of perimeter $2p$, then -- as shown in~\cite{CMboundary} -- we have the following convergences in distribution for the local metric: 
\begin{equation}\label{intro:UIHPQ as double limit} \Q_{n,p} \xrightarrow[n\to\infty]{(d)} \Q_{\infty,p} \xrightarrow[p\to\infty]{(d)}  \UIHPQ.\end{equation}
The random map $\Q_{\infty,p}$ is the Uniform Infinite Planar Quadrangulation with a boundary of perimeter $2p$, while $\UIHPQ$, our object of interest, is the aforementioned UIHPQ. The proof of~\cite{CMboundary} relies on a Schaeffer construction of the ``unconstrained'' type, roughly analogous to that of~\cite{CMMinfini} for the UIPQ.

Our first goal is to give an equivalent construction of the UIHPQ, modelled on that of~\cite{CD06}, which employs a treed bridge bearing positive labels (see Section~\ref{treed bridges}), thus rendering information about the profile of distances from the root vertex readily available for further study.

\paragraph{The positive construction of $\UIHPQ$.} 
\begin{figure}
\centering\begin{tikzpicture}

\tikzstyle{real}=[inner sep=1.5pt, draw=white,fill=black, thick, circle]
\tikzstyle{phantom}=[inner sep=1.5pt, draw=black,fill=white, thick, circle]
\tikzstyle{map}=[red,thin]
\tikzstyle{boundary}=[very thick]

	\begin{pgfonlayer}{nodelayer}
		\node [style=phantom, label=below:\contour{white}{0}] (0) at (0, 0) {};
		\node [style=real, fill=red, label=right:\contour{white}{\footnotesize{$\delta$}}] (delta) at (0, 1) {};
		\node [style=phantom, label=below:1] (1) at (1, 0) {};
		\node [style=real, label=below:2] (2) at (2, 0) {};
		\node [style=phantom, label=below:1] (3) at (3, 0) {};
		\node [style=phantom, label=below:2] (4) at (4, 0) {};
		\node [style=real, label=below:3] (5) at (5, 0) {};
		\node [style=real, label=below:1] (-1) at (-1, 0) {};
		\node [style=real, label=below:2] (-2) at (-2, 0) {};
		\node [style=real, label=below:3] (-3) at (-3, 0) {};
		\node [style=phantom, label=below:2] (-4) at (-4, 0) {};
		\node [style=real, label=below:3] (-5) at (-5, 0) {};
		\node (left) at (-6,0) {$\ldots$};
		\node (right) at (6,0) {$\ldots$};
		\contourlength{1px}
		\node [style=real, label=right:\contour{white}{\footnotesize 2}] (2-1) at (2, 0.75) {};
		\node [style=real, label=left:\contour{white}{\footnotesize 3}] (2-11) at (1.25, 1.5) {};
		\node [style=real, label=below:\contour{white}{\footnotesize 1}] (2-13) at (2.75, 1.5) {};
		\node [style=real, label=above:\contour{white}{\footnotesize 2}] (2-111) at (1, 2) {};
		\node [style=real, label=above:\contour{white}{\footnotesize 4}] (2-112) at (1.5, 2) {};
		\node [style=real, label=above:\contour{white}{\footnotesize 2}] (2-12) at (2, 1.5) {};
		\node [style=real, label=left:\contour{white}{\footnotesize 2}] (5-1) at (4.75, 0.75) {};
		\node [style=real, label=right:\contour{white}{\footnotesize 4}] (5-2) at (5.25, 0.75) {};
		\node [style=real, label=left:\contour{white}{\footnotesize 2}] (-2-1) at (-2, 0.75) {};
		\node [style=real, label=above:\contour{white}{\footnotesize 2}] (-2-11) at (-2.5, 1.5) {};
		\node [style=real, label=left:\contour{white}{\footnotesize 3}] (-2-12) at (-1.5, 1.5) {};
		\node [style=real, label=right:\contour{white}{\footnotesize 4}] (-5-1) at (-5, 0.75) {};
		\node [style=real,, label=right:\contour{white}{\footnotesize 4}] (-5-11) at (-5, 1.5) {};
		\node [style=real,, label=right:\contour{white}{\footnotesize 5}] (-5-111) at (-5, 2.25) {};
	\end{pgfonlayer}
	\begin{pgfonlayer}{edgelayer}
	

\fill[gray!10, out=170, in=210, boundary] (2) to (1.1,2.8) [out=30, in=90] to (2-13);
\fill[gray!10, out=60, in=0, looseness=1] (2-13) to (1.5,3.1) [in=80, out=180, looseness=0.8] to (delta.center) [bend left=25] to (-1.center)--(0,2)--(1,3.3)--(3.2,2.8);
\fill[gray!10, out=40, in=10, looseness=2, boundary] (2) to (1,3.4) [in=60, out=190, looseness=1] to (-1)--(0,2)--(1,3.3)--(3.2,2.8)--(2-13.center)--(2-13.west)--(2.west);

\fill[gray!10,out=100, in=10, looseness=1] (5-1.center) to (1,3.6) [in=70, out=190, looseness=1] to (-1.center)--(-1,4)--(5,4);
\fill[gray!10, bend right=30] (-1.center) to (-2.center)[bend right=25] to (-5.center)--(-5,4)--(-1,4);
\fill[gray!10, bend right=20] (5-1.center) to (5.center)[bend left=20] to (6.2,0.5)--(6.2,4)--(5,4);
\fill[gray!10, bend right=20] (-5.center) to (-6.2,0.5)--(-6.2,4)--(-5,4);
\fill[white, path fading=south] (-6.2,4) rectangle (6.2,2);
	\draw[map, boundary, bend right=20] (5-1) to (5);
	\draw[map, bend left=20] (5-2) to (5);
	\draw[map, boundary, bend left=20, dashed] (5) to (6.2,0.5);
	\draw[map, boundary, bend left=20] (5) to (6.2,0.5);
	\draw[map, out=90, in=180] (5) to (6.2,1.5);
	\draw[map, bend right=20, boundary] (-5) to (-6.2,0.5);
	
	\fill[white, path fading=east] (-6.2,4) rectangle (-5,0);
\fill[white, path fading=west] (6.2,4) rectangle (5,0);
	\draw[dashed] (0)--(delta);
		\draw   (-2) to (-2-1);
		\draw   (-2-1) to (-2-11);
		\draw   (-2-1) to (-2-12);
		\draw   (2) to (2-1);
		\draw   (2-1) to (2-13);
		\draw   (2-1) to (2-12);
		\draw   (2-11) to (2-1);
		\draw   (2-111) to (2-11);
		\draw   (2-11) to (2-112);
		\draw   (-5) to (-5-1) to (-5-11) to (-5-111);
		\draw   (5-1) to (5) to (5-2);
		\draw[thick,->]   (0) to (1);
		\draw   (1) to (2);
		\draw   (2) to (3);
		\draw   (3) to (4);
		\draw   (4) to (5) to (right);
		\draw   (0) to (-1) to (-3) to (-4) to (-5) to (left);
	\draw[map, bend left=20,rounded corners=1cm] (-5) to ([xshift=-20pt, yshift=30pt]-5-111.center) [out=20, in=180] to (-2); 
	\draw[map, bend left=20,rounded corners=15pt] (-5-1) to ([xshift=-10pt, yshift=20pt]-5-111.center) [bend left=70] to (-5); 
	\draw[map, bend left=20,rounded corners=10pt] (-5-11) to ([xshift=0pt, yshift=15pt]-5-111.center) [bend left=40] to (-5); 
	\draw[map, bend left=20] (-5-111) to (-5-11);
	\draw[map, bend left=30] (-5-11) to (-5);
	\draw[map, bend left=30] (-5-1) to (-5);
	\draw[map, bend left=25, boundary] (-5) to (-2);
	\draw[map, bend left=25, boundary] (-3) to (-2);
	\draw[map, bend left=30,rounded corners=20pt, bend left=40, looseness=1.2] (-2) to (-2.4,2.7) [bend left=50] to (-1);
	\draw[map, bend left=30,rounded corners=15pt, bend left=60, looseness=1.6] (-2-1) to (-2,2.5) [bend left=40, looseness=1] to (-1);
	\draw[map, bend left=30,rounded corners=5pt, bend left=40, looseness=1] (-2-11) to (-1.8,2.2) [bend left=30, looseness=1] to (-1);
	\draw[map, bend left=30,rounded corners=15pt, bend left=0, looseness=1] (-2-1) to (-2,2.05) [bend left=30, looseness=1] to (-1);
	\draw[map, bend left=25] (-2-12) to (-2-1);
	\draw[map, bend left=25] (-2-1) to (-1);
	\draw[map, bend left=25, boundary] (-2) to (-1);
	
	\draw[map, out=170, in=210, boundary] (2) to (1.1,2.8) [out=30, in=90] to (2-13);
	\draw[map, out=180, looseness=1, rounded corners=15pt, looseness=1.3, in=-110] (2-1) to (1,2.7) [out=0, in=110] to (2-13);
	\draw[map,bend left=30] (2-11) to (2-111);
	\draw[map,bend left=20] (2-112) to (2-11);
	\draw[map,bend left=10] (2-11) to (2-1);
	
	\draw[map, out=90, looseness=1, in=180] (2-11) to ([yshift=5pt]2-112.center) [out=0, in=125, looseness=0.8] to (2-1);
	\draw[map, out=115, looseness=2] (2-1) to (2-13);
	\draw[map,bend left=20] (2-12) to (2-13);
	\draw[map,bend left=20] (2-1) to (2-13);
	\draw[map, out=60, in=0, looseness=1,thick,<-,black] (2-13) to (1.5,3.1) [in=80, out=180, looseness=0.8] to (delta);
	\draw[map, boundary, bend right=20] (-1) to (delta);
	\draw[map, out=70, in=120, looseness=1] (2-111) to (2-13);
	\draw[map, out=20, in=20, looseness=2] (2-1) to (1,3.2) [in=30, out=200, looseness=1] to (-1);
	\draw[map, out=40, in=10, looseness=2, boundary] (2) to (1,3.4) [in=60, out=190, looseness=1] to (-1);
	\draw[map, out=100, in=10, looseness=1, boundary] (5-1) to (1,3.6) [in=70, out=190, looseness=1] to (-1);

	\end{pgfonlayer}
\end{tikzpicture}
\caption{\label{Phi(B_infty)}\small{The construction of the UIHPQ as $\Phi(\B_\infty)$.}}
\end{figure}
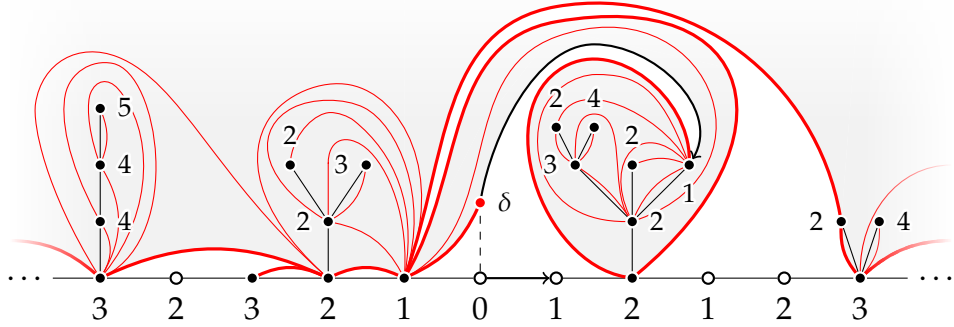
Let us describe this new construction in some detail. We first consider  a random process $(X_{i})_{i \in \mathbb{Z}}$ indexed by $ \mathbb{Z}$ such that $(X_{i})_{i \geq 0}$ and $(X_{-i})_{i \geq 0}$ are two independent, identically distributed, nearest-neighbour random walks on $  \mathbb{N}= \{0,1,2, \ldots\}$, issued from $X_0=0$, with transition probabilities given by
 \begin{eqnarray} \mathbf{p}( n, n-1) = \frac{n}{2(n+2)}, \qquad \mathbf{p}( n, n+1) = \frac{n+4}{2(n+2)}\qquad\mbox{for $n\geq0$}. 
\label{eq:defp}  \end{eqnarray}
The process $(X_{i})_{i \in \mathbb{Z}}$ can be seen as a non-negative labelling of the vertices of a doubly infinite path, indexed by $\mathbb{Z}$. 

This path acts as a baseline on which we graft a sequence of random positive labelled plane trees: each tree is finite, its vertices labelled with (strictly) positive integers so that labels assigned to neighbouring vertices differ by 0,1 or $-1$. For each $i\geq 0$, we denote by $\rho_i^+$ the probability measure that gives weight proportional to $12^{-n}$ to any positive labelled tree with $n$ edges such that the label of the root vertex is $i$ ($\rho_i^+$ is the \emph{Boltzmann distribution on positive labelled trees} with root labelled $i$, see Section~\ref{trees} and~\cite{CD06,LGM10}).

We construct a random ``infinite treed bridge'' $\B_\infty$ as follows: conditionally on $(X_{i})_{i \in \mathbb{Z}}$, on each vertex $j$ of the baseline path such that $X_{j+1}<X_j$ (such indices $j$ will be dubbed \emph{down-steps}) we graft an independent random plane tree of law $\rho_{X_j}^+$ (see Section~\ref{the new construction}). Through a variant of the Schaeffer construction, $\B_\infty$ corresponds to a (random) infinite quadrangulation with an infinite boundary, denoted $\Phi(\mathcal{B}_{\infty})$ (see Figure~\ref{Phi(B_infty)}).  Our first result (Theorem~\ref{new construction}) consists in showing that such a construction yields the UIHPQ with general boundary, that is to say
\begin{eqnarray}\Phi(\mathcal{B}_{\infty}) = \UIHPQ \quad \mbox{ in distribution}.  \label{eq:intronewconstruction}\end{eqnarray}

This new construction provides precious insight into the geometry of the UIHPQ, since vertices of $\Phi( \mathcal{B}_{\infty})$ (with the exception of the root vertex) correspond to vertices of the trees in $\mathcal{B}_{\infty}$, and their respective distances from the root vertex in $\Phi( \mathcal{B}_{\infty})$ are encoded by their labels in $\B_\infty$. The process $(X_{i})_{i \in \mathbb{Z}}$ has an especially simple geometric interpretation: it corresponds to the process of distances to the root vertex read along the left-to-right contour of the boundary of the UIHPQ. An immediate consequence is the fact (nontrivial at first glance) that the two processes of distances from the root vertex read along the boundary of the UIHPQ, starting with the root vertex and proceeding to the right or to the left, are independent.

Although the construction of the UIHPQ given via the positive treed bridge $\B_\infty$ is more involved than that of~\cite{CMboundary}, we shall see that it still leads to simple computations, delivering a substantial amount of geometric information not readily accessible from~\cite{CMboundary}.

\paragraph{Study of the geodesic pencil.} 
A geodesic ray in $\UIHPQ$ is a one-ended infinite geodesic issued from the root vertex of $\UIHPQ$. Thanks to our new construction \eqref{eq:intronewconstruction}, we are able to characterise geodesic rays as sequences of corners in $ \mathcal{B}_{\infty}$ and prove (Theorem~\ref{infinite cut points}) that the ``geodesic pencil'' consisting of all geodesic rays has infinitely many cut-points: in order words, there almost surely exists an infinite set of vertices that all geodesic rays must pass through in order to go to infinity. This result is the exact analogue of the confluence of discrete geodesics established in~\cite{CMMinfini} for the case of the UIPQ. 

We also prove that the geodesic pencil splits the UIHPQ into two independent submaps (whose distributions mirror one another) giving an additional explanation of the fact that the processes of distances to the root vertex read along the left and right halves of the boundary are independent.  This study of geodesics also yields an answer to an open question of~\cite{CMboundary} about the geometric interpretation of labels in the construction of~\cite{CMboundary}, which proves similar to~\cite[Theorem 2.8]{CMMinfini}. 

\begin{figure}[!h, width=1.5\textwidth]
\centering
\begin{tikzpicture}[scale=0.80]
\tikzstyle{real}=[inner sep=1.5pt, draw=white,fill=black, thick, circle]
\tikzstyle{bla}=[thick]
\tikzstyle{geoleft}=[very thick, blue]
\tikzstyle{boundary}=[very thick]
\tikzstyle{root}=[->, very thick]
	\begin{pgfonlayer}{nodelayer}
	\contourlength{1pt}
		\node [style=real] (0) at (0, 0) {};
		\node [style=real] (1) at (2, 0) {};
		\node [style=real] (2) at (4, 0) {};
		\node [style=real] (3) at (-2, 0) {};
		\node [style=real] (4) at (-4, 0) {};
		\node [style=real] (5) at (-6, 0) {};
		\node [style=real] (6) at (-8, 0) {};
		\node (left) at (-9,0) {$\ldots$};
		\node (right) at (9,0) {$\ldots$};
		\node [style=real] (7) at (6, 0) {};
		\node [style=real] (8) at (8, 0) {};
		\node [style=real] (9) at (0.75, -0.75) {};
		\node [fill=blue,style=real] (10) at (-0.25, -1) {};
		\node [fill=blue,style=real] (11) at (0.75, -1.75) {};
		\node [fill=blue,style=real] (12) at (-0.75, -2.5) {};
		\node [fill=blue,style=real, label=below right:$\rho$] (13) at (0.25, -2.75) {};
		\node [style=real] (14) at (-1, -1) {};
		\node [style=real] (15) at (-0.5, -3) {};
		\node [style=real] (16) at (3.25, -0.75) {};
		\node [style=real] (17) at (4.5, -0.75) {};
		\node [style=real] (18) at (4, -0.5) {};
		\node [style=real] (19) at (4, -1) {};
		\node [style=real] (20) at (4, -1.5) {};
		\node [style=real] (21) at (4, -2.25) {};
		\node [style=real] (22) at (3, -2.75) {};
		\node [style=real] (23) at (6, -1) {};
		\node [style=real] (24) at (6.75, -1) {};
		\node [style=real] (25) at (6.75, -1.75) {};
		\node [style=real] (26) at (6, -1) {};
		\node [style=real] (27) at (-4, -1) {};
		\node [style=real] (28) at (-7, -0.75) {};
		\node [style=real] (29) at (-6, -1.25) {};
		\node [style=real] (30) at (-5.5, -0.75) {};
		\node [style=real] (31) at (-5.5, -1.5) {};
		\node [style=real] (32) at (-5, -1) {};
		\node [style=real] (33) at (-5, -2) {};
		\node [style=real] (34) at (-7, -2) {};
		\node [style=real] (35) at (-8, -0.5) {};
		\node [style=real] (36) at (-8.25, -1.25) {};
		\node [style=real] (37) at (-7.5, -1.25) {};
		\node [style=real] (38) at (-0.75, 2) {};
		\node [style=real] (39) at (1.75, 5) {};
		\node [style=real] (41) at (-1, 0.75) {};
		\node [style=real] (42) at (-1, 4) {};
		\node [style=real] (43) at (0.25, 5) {};
		\node [style=real] (44) at (1.5, 5.8) {};
		\node [style=real] (45) at (0.25, 1) {};
		\node [style=real] (46) at (0.25, 1) {};
		\node [style=real] (47) at (-1.5, 3) {};
		\node [style=real] (48) at (0.75, 2.75) {};
		\node [style=real] (49) at (1.5, 3.5) {};
		\node [style=real] (50) at (0, 2.25) {};
		\node [style=real] (51) at (1.5, 4.25) {};
		\node [style=real] (52) at (0.5, 4) {};
		\node [style=real] (53) at (-0.25, 3.75) {};
		\node [style=real] (54) at (2.75, 5.75) {};
	\end{pgfonlayer}
	\begin{pgfonlayer}{edgelayer}
	\draw[line width=12pt, red, rounded corners=1pt] (28.center) to (35.center) to (36.center) to (37.center) to (28.center) to (34.center) 
	to (31.center) to (29.center) to (30.center) to (5.center) to (4.center) to (27.center) to (4.center) to 
	(2.center) to (16.center) to (21.center) to (17.center);
	\draw[line width=12pt, red, rounded corners=1pt, cap=round] (13.center) to (11.center) to (9.center) to (0.center) to (10.center) to (12.center) [bend right=90, looseness=2.6] to (13.center);
	\draw[line width=12pt, red, rounded corners=1pt, cap=round] (33.center)--(31.center)--(32.center);
	\draw[line width=12pt, red, rounded corners=1pt, cap=round] (4.center)--(27.center);
	\draw[line width=12pt, red, rounded corners=1pt, cap=round] (14.center)--(10.center);
	\draw[line width=12pt, red, rounded corners=1pt, cap=round] (21.center)--(22.center);
	
	\draw[line width=10pt, white, rounded corners=1pt, cap=round] (5.center) to (28.center) to (35.center) to (36.center) to (37.center) to (28.center) to (34.center) 
	
	to (31.center) to (29.center) to (30.center) to (5.center) to (4.center) to (27.center) to (4.center) to 
	(2.center) to (17.center) to (21.center) to (16.center) to (2.center);
	\draw[line width=10pt, white, rounded corners=1pt, cap=round] (13.center) to (11.center) to (9.center) to (0.center) to (10.center) to (12.center) [bend right=90, looseness=2.6] to (13.center);
	\draw[line width=10pt, white, rounded corners=1pt, cap=round] (33.center)--(31.center)--(32.center);
	\draw[line width=10pt, white, rounded corners=1pt, cap=round] (4.center)--(27.center);
	\draw[line width=10pt, white, rounded corners=1pt, cap=round] (14.center)--(10.center);
	\draw[line width=10pt, white, rounded corners=1pt, cap=round] (21.center)--(22.center);
	\draw[line width=10pt, white, rounded corners=1pt, cap=round] ([yshift=3pt]5.center) to ([yshift=3pt]2.center);
	\draw[red, ->, line width=1pt] ([xshift=7.2pt]21.center) to ([xshift=7.2pt]17.center);
	
		\fill[blue!20] (13.center) to (11.center) to (9.center) to (0.center) to (10.center) to (12.center) [bend right=90, looseness=2.6] to (13.center);
		\fill[green!10] (2.center) to (17.center) to (21.center) to (16.center) to (2.center);
		\fill[blue!10] (5.center) to (28.center) to (35.center) to (36.center) to (37.center) to (28.center) to (34.center) to (31.center) to (29.center) to (30.center) to (5.center);
		\fill[red!10] (0.center) to (41.center) to (38.center) to (47.center) to (42.center) to (43.center) to (39.center) to (44.center) to (54.center) to (39.center) [bend left] to (49.center) [in=0, out=0] to (48.center) to (50.center) to (38.center) to (45.center) to (0.center) to (9.center) to (11.center) to (13.center) to (12.center) to (10.center) to (0.center);
		\contourlength{2pt}
		
		\fill[blue!20, path fading=west] (0.center) to (41.center) to (38.center) to (47.center) to (42.center) to (43.center) to (39.center) to (44.center) to (-8,6) to (-9,0) to (0.center);
		\fill[green!20, path fading=east] (9,6) to (54.center) to (39.center) [bend left] to (49.center) [in=0, out=0] to (48.center) to (50.center) to (38.center) to (45.center) to (0.center) to (9,0) to (8,6);	
		
		\fill[white, path fading=west] (6,6) rectangle (9,0);
		\fill[white, path fading=east] (-6,6) rectangle (-9,0);
		\draw [style=geoleft] (0) to (41);
		\draw [style=geoleft] (38) to (41);
		\draw [style=geoleft] (42) to (43);
		\draw [style=geoleft] (43) to (39);
		\draw [style=geoleft] (39) to (44);
		\draw [style=geoleft] (38) to (47);
		\draw [style=geoleft] (47) to (42);
		\draw [style=bla, bend left=15] (50) to (42);
		\draw [style=bla, bend left, looseness=1.25] (49) to (39);
		\draw [style=bla] (51) to (49);
		\draw [style=bla, bend left=15, looseness=0.50] (42) to (52);
		\draw [style=bla] (52) to (39);
		\draw [style=bla, bend right=60, looseness=1.25] (42) to (52);
		\draw [style=bla] (53) to (52);
		\draw [style=bla] (52) to (48);
		\draw [style=georight] (0) to (45);
		\draw [style=georight] (45) to (38);
		\draw [style=georight] (50) to (38);
		\draw [style=georight] (48) to (50);
		\draw [style=georight] (48) to (49);
		\draw [style=georight, bend right] (49) to (39);
		\draw [style=georight] (54) to (39);
		
		\draw [style=geoleft] (10) to (0);
		\draw [style=bla] (10) to (11);
		\draw [style=georight] (11) to (9);
		\draw [style=georight] (9) to (0);
		\draw [style=geoleft] (12) to (10);
		\draw [style=bla, bend right=90, looseness=2.25] (12) to (13);
		\draw [style=georight,->] (13) to (11);
		\draw [style=bla] (14) to (10);
		\draw [style=bla] (13) to (15);
		\draw [style=geoleft] (12) to (13);
		\draw [style=bla] (16) to (2);
		\draw [style=bla] (2) to (17);
		\draw [style=bla] (18) to (16);
		\draw [style=bla] (18) to (17);
		\draw [style=bla] (17) to (19);
		\draw [style=bla] (19) to (16);
		\draw [style=bla] (16) to (20);
		\draw [style=bla] (20) to (17);
		\draw [style=bla] (21) to (17);
		\draw [style=bla] (21) to (16);
		\draw [style=bla] (22) to (21);
		\draw [style=bla] (23) to (7);
		\draw [style=bla] (23) to (24);
		\draw [style=bla] (23) to (25);
		\draw [style=bla] (27) to (4);
		\draw [style=bla] (5) to (28);
		\draw [style=bla] (28) to (29);
		\draw [style=bla] (29) to (30);
		\draw [style=bla] (30) to (5);
		\draw [style=bla] (29) to (31);
		\draw [style=bla] (31) to (32);
		\draw [style=bla] (31) to (33);
		\draw [style=bla] (31) to (34);
		\draw [style=bla] (34) to (28);
		\draw [style=bla] (36) to (35);
		\draw [style=bla] (37) to (36);
		\draw [style=bla] (37) to (28);
		\draw [style=bla] (28) to (35);
		\draw [style=boundary] (6) to (5);
		\draw [style=boundary] (5) to (4);
		\draw [style=boundary] (4) to (3);
		\draw [style=boundary] (3) to (0);
		\draw [style=boundary] (0) to (1);
		\draw [style=boundary] (1) to (2);
		\draw [style=boundary] (2) to (7);
		\draw [style=boundary] (7) to (8);
		\draw [style=boundary] (6) to (left);
		\draw [style=boundary] (8) to (right);
	\end{pgfonlayer}
	\fill[white, path fading=south] (-9,6) rectangle (9,0);
	\contourlength{1px}
	\draw [red, thick, densely dashed] (2) .. controls (1,3) and (2,1) .. (0) .. controls (0.5,-0.5) and (0.5,-2) .. (13);
	\draw node at (2,2) {\contour{white}{$X_i$}};
	\pgfresetboundingbox
	\path [use as bounding box] (-9,-4) rectangle (9,6.5);
\end{tikzpicture}
 \caption{\label{intro}\small{The drawing illustrates the pencil decomposition of the UIHPQ; the blue and green paths are the leftmost and rightmost geodesic rays , while $\rho$ is the root vertex and the red arrow represents the boundary contour.}}
 \end{figure}
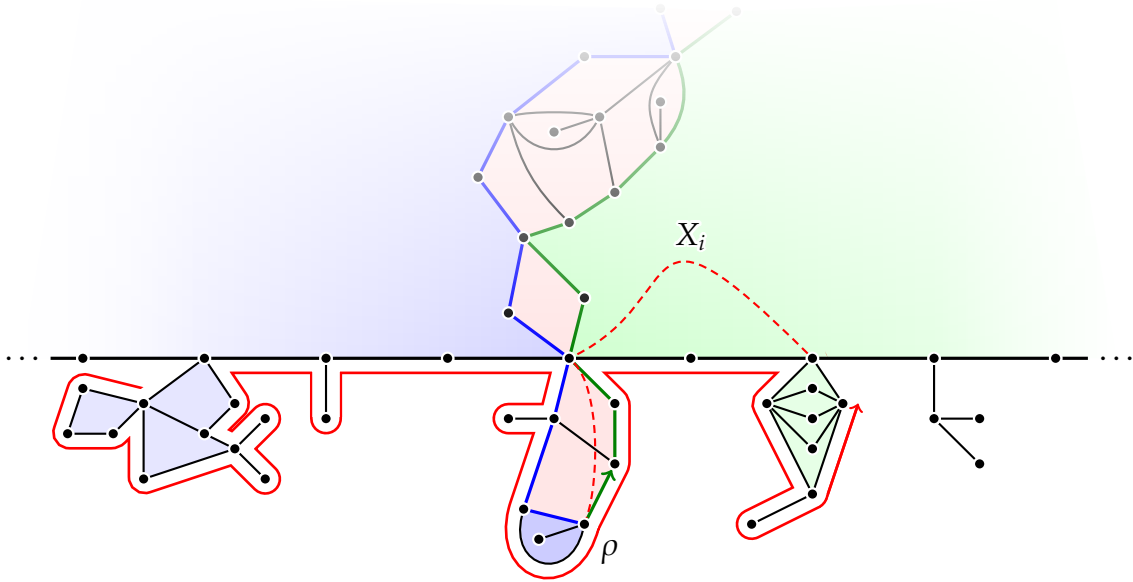
 
\paragraph{Scaling limits for $ \mathcal{B}_{\infty}$.}  
Our description of $\UIHPQ$ as $\Phi(\B_\infty)$ also yields scaling limit results for the UIHPQ: combining the explicit transition probabilities \eqref{eq:defp} with a well known result of Lamperti, we obtain
$$\left(  \frac{X_{[nt]}}{ \sqrt{n}}\right)_{t \in \mathbb{R}} \xrightarrow[n\to\infty]{(d)} (Z_{t})_{t \in \mathbb{R}},$$ where $ (Z_{t})_{t \geq 0}$ and $(Z_{-t})_{t \geq 0}$ are two independent Bessel processes of dimension $5$ started from $0$ (see Proposition~\ref{prop:scalingX}).
It is also possible to describe the scaling limit of the contour functions coding $ \mathcal{B}_{\infty}$, in the spirit of the work of Le Gall and M\'enard~\cite{LGM10}, see Section~\ref{scaling limits} for details. 

Such results yield scaling limits for various geometric quantities associated with the UIHPQ: if $ \#[\UIHPQ]_{n}$ is the number of vertices within distance $n$ from the root vertex in $\UIHPQ$, then (Proposition~\ref{prop:limitlawball}) we have
\begin{eqnarray} \frac{1}{n^4}\# [\UIHPQ]_n \xrightarrow[n\to\infty]{(d)} \mathcal{V}_{h}  \label{introcvballuihpq} \end{eqnarray} 
for a certain explicit limiting law $ \mathcal{V}_{h}$ (where ``$h$'' stands for ``half-plane''). We can compare the expectations of $ \mathcal{V}_{h}$ and of the random variable $ \mathcal{V}_{p}$ from \eqref{eq:cvballuipq}, thus obtaining 
\begin{eqnarray} \frac{\mathbb{E}[ \mathcal{V}_{h}]}{ \mathbb{E}[ \mathcal{V}_{p}]} = \frac{7}{9};  \label{intro:comparisonmean}\end{eqnarray} 
that is, large balls in the UIHPQ are on average $7/9$ times as large as balls of the same radius in the UIPQ. The factor 7/9 should be universal (i.e.~independent of the specific combinatorics of the class of maps being considered) as it can be interpreted directly in terms of the scaling limit.

Finally, by mimicking the Schaeffer construction on the continuous processes describing the scaling limit of  $\mathcal{B}_{\infty}$ (in a way similar to the construction of the Brownian map~\cite{LG07} or the Brownian plane~\cite{CLGHull,CLGplane} from labelled continuous trees) we construct a random locally compact metric space that we name the \emph{Brownian half-plane} and conjecture it to be the scaling limit of the UIHPQ in the local Gromov--Hausdorff sense.

\paragraph{The UIHPQ with a simple boundary.}

All of the scaling limit results obtained for the UIHPQ can be adapted to the case of the \emph{UIHPQ with a simple boundary}. Consider the set ${\sQS}_{n,p}\subseteq \sQ_{n,p}$ of (rooted) quadrangulations having area $n$, perimeter $2p$ and a \emph{simple} boundary (i.e.~the quadrangulations of the $2p$-gon with $n$ internal faces). If $\QS_{n,p}$ is a quadrangulation chosen uniformly at random within ${\sQS}_{n,p}$, then we have the following convergence in distribution for the local distance, shown by Angel~\cite{Ang05}:
$${\QS}_{n,p} \xrightarrow[n\to\infty]{(d)} \QS_{\infty,p} \xrightarrow[p\to\infty]{(d)}  \UIHPQS,$$ 
 where $\QS_{\infty,p}$ is the uniform infinite quadrangulation of the $2p$-gon and $\UIHPQS$ is the Uniform Infinite Half-Planar Quadrangulation with a simple boundary (abbreviated here by UIHPQ$^{(s)}$). The second author and Miermont~\cite{CMboundary} have given a construction of the UIHPQ$^{(s)}$ via a pruning procedure applied to the UIHPQ, which consists in erasing the finite quadrangulations hanging from the (simple) boundary of its infinite core (see Figure~\ref{fig:pruning}). We use this construction to extend all of our scaling limit results for the UIHPQ to the simple boundary case, mainly thanks to the fact that the extra finite quadrangulations do not contribute to the mass of balls of large radius, and simply dilate the contour process by a constant factor when considered at large scales.

Part of the interest of the UIHPQ$^{(s)}$ lies in the fact that one can perform surgery operations by glueing its boundary in various ways: glueing the left half of the boundary to the right half yields a random infinite quadrangulation of the plane endowed with a one-ended self-avoiding walk, while glueing the boundaries of two independent copies of the UIHPQ$^{(s)}$ to one another gives a random infinite quadrangulation of the plane with a two-ended SAW. Thanks to the results on the UIHPQ$^{(s)}$ obtained here, we shall investigate such models of self-avoiding walks on random infinite quadrangulations in a companion paper~\cite{companion}.
\bigskip

This paper is structured as follows. Section~\ref{notation and bijections} introduces notation, touching on useful enumeration results for various combinatorial objects, and describes the Schaeffer-type constructions used throughout the paper. Section~\ref{the new construction} computes local limits for classes of random treed bridges, culminating in the proof of \eqref{eq:intronewconstruction}. Section~\ref{study of geodesic rays} deals with geodesic rays in the UIHPQ and with the so-called ``pencil decomposition''. Finally, Section~\ref{scaling} gives scaling limits for the processes associated with the local limit of random treed bridges and computes the limiting law for the volume of large balls in the UIHPQ, while Section~\ref{simple boundary} extends these results to the case of the UIHPQ with a simple boundary.

\paragraph{Acknowledgements.} The authors wish to thank Jérémie Bouttier for providing them with an alternative derivation of \eqref{intro:comparisonmean} based on the work~\cite{BG09}.

 \section{Schaeffer-type constructions}\label{notation and bijections}
This section will recall two constructions of quadrangulations with a general boundary via Schaeffer-type bijections. To begin with, we review the formalism of plane trees together with a few enumeration results; for additional details, see~\cite{thesis}.

\subsection{Trees, bridges and treed bridges}

We start by introducing the combinatorial objects we shall employ and some necessary notation.

\subsubsection{Trees}\label{trees}
A \emph{plane tree} is a locally finite acyclic connected graph, properly embedded in the plane (up to orientation-preserving homeomorphisms of the plane itself) and endowed with a distinguished corner; the vertex determined by the distinguished corner is called the \emph{root} of the tree, and from it a genealogical structure can be inferred (so that each vertex apart from the root has a parent); moreover, an ordering may be deduced for the vertices in each generation (using the distinguished corner as a starting point). We refer the reader to~\cite{LG05} for Neveu's formalism of plane trees.

Given a plane tree $\tau$, we write $|\tau|$ for the number of edges in $\tau$, also called the \emph{size} of $\tau$. Note that, as opposed to~\cite{CMboundary}, we shall not need to deal with any infinite trees throughout the paper; thus all trees mentioned will be implicitly considered finite. 

We may introduce a metric on the set $\T$ of all finite plane trees, usually called the \emph{local metric}. Given a plane tree $\tau$ and a non-negative integer $h$ denote by $[\tau]_h$ the tree one obtains by erasing from $\tau$ all vertices having graph distance strictly greater than $h$ from the root (and any edges involving such vertices) so that only the first $h$ generations (the 0-th being the root) remain. We then define, for each pair of plane trees $\tau$, $\tau'$ in $\T$,
\begin{eqnarray*} \mathrm{d_{tree}}(\tau,\tau') &=& \big( 1+ \sup\{ h \geq 0 : [\tau]_h= [\tau']_h \}\big) ^{-1}. \end{eqnarray*}

Our trees will often be endowed with labellings; a (finite) \emph{labelled plane tree} is given by 
\begin{itemize}
\item a plane tree $\tau$;
\item a function $l$ from the vertex set of $\tau$ to the integers such that, if $u$ and $v$ are neighbours in $\tau$, then $|l(u)-l(v)|\leq1$.
\end{itemize}

For each integer $k$, we call $\LT_k$ the set of all finite labelled plane trees whose root has label $k$, and $\LT$ the set $\cup_{k\in\mathbb{Z}}\LT_k$ of all finite labelled plane trees. The distance $\mathrm{d_{tree}}$ can still be defined on $\LT$ by declaring equality of labelled trees to imply equality of labels as well as equality of underlying plane trees.

Notice that there is a bijection between the sets $\LT_k$ (for any integer $k$) and $\LT_0$ which simply consists in subtracting $k$ to all labels; as a consequence, one has
$$ \sum_{ \tau \in \mathsf{\LT}_{k}} 12^{-|\tau|} =\sum_{ \tau \in \mathsf{\LT}_{0}} 12^{-|\tau|}=2.$$

The last identity is a simple consequence of the fact that there are $3^n$ labelled trees in $\LT_0$ for each plane tree of size $n$ (since labels vary by 1, 0 or $-1$ along each edge), and that plane trees of given size are counted by Catalan numbers (that is, there are ${2n \choose n}/(n+1)$ plane trees of size $n$). 

One may then introduce the Boltzmann measure  $\rho_k$ on $\LT_k$, defined so that $\rho_k(\{\tau\})=12^{-|\tau|}/2$. The probability measure $\rho_k$ is the law of a critical geometric Galton--Watson tree for which (conditionally on its shape) a uniform random integer $i_e \in\{1,0,-1\}$ is selected independently for each edge $e$; its labels are computed as $l(u)=k+\sum_{e\in P_u}i_e$, where $P_u$ is the unique non-backtracking path leading from the root to the vertex~$u$.

We will also work with \emph{positive} versions of labelled plane trees: for each $k>0$, we define sets $\LT^+_k \subset \LT_k$ and $\LT^+\subset \LT$ by requiring each label to be a (strictly) positive integer.

This time the identity cited above takes the form
\begin{eqnarray}\sum_{ \tau \in \mathsf{\LT}^+_{k}} 12^{-|\tau|}=\frac{2k(k+3)}{(k+1)(k+2)}=:w_k\label{w_k}\end{eqnarray}
as shown in~\cite{CD06,BDFG04}, and we define the Boltzmann measure $\rho_k^+$ on $\LT_k^+$ by $\rho_k^+(\{\tau\})=12^{-|\tau|}/w_k$. Again, this is the law of a multi-type Galton--Watson tree (see~\cite[Theorem 4.6]{CD06}); the offspring of a vertex $u$ labelled $l$ is generated by repeatedly selecting one of the following outcomes at random until the offspring of $u$ is declared complete:
\begin{itemize}
\item with probability $w_{l-1}/12$, a new child labelled $l-1$ is added to the right of the last child of $u$;
\item with probability $w_{l+1}/12$, a new child labelled $l+1$ is added to the right of the last child of $u$;
\item with probability $w_{l}/12$, a new child labelled $l$ is added to the right of the last child of $u$;
\item with probability $1/w_{l}$, the offspring of $u$ is declared complete.
\end{itemize}

\subsubsection{Quadrangulations}\label{quadrangulations}
A quadrangulation with a boundary is a locally finite planar map whose faces are all quadrangles, except for one finite or infinite face which we call the outerface, and whose boundary (which is a path, not necessarily simple) we see as the boundary of the map; the map is rooted by choosing an edge of the boundary, oriented so that the outerface lies to its right.

We say a quadrangulation with a boundary has area $n$ if it has $n+1$ faces (outerface included); it has perimeter $2p$ if such is the length of its boundary. We call $\sQ_{n,p}$ the set of all quadrangulations of area $n$ and perimeter $2p$, and $\sQ$ the set of all rooted quadrangulations with a boundary, which may have finite or infinite area and perimeter (though they are always, as per our definition, locally finite). Notice that the set $\sQ_{n,1}$ can be identified with the set of all ``standard'' quadrangulations with $n$ faces (where this time we mean rooted maps all of whose faces are quadrangles) by collapsing the two edges of the boundary of each of its elements. The word ``quadrangulation'' within this paper will, from this moment onwards, always stand for ``rooted quadrangulation with a boundary'' unless otherwise stated.

We can define a local distance on the set $\sQ$: given a quadrangulation $q\in\sQ$, let $[q]_r$ (for $r\geq1$) be the (rooted) map obtained from $q$ by erasing all vertices at (graph) distance strictly greater than $r$ from the root vertex of $q$ (i.e.~the tail of the root edge) and all edges involving such vertices; for any pair of quadrangulations $q, q'$ in $\sQ$ we define
$$ \mathrm{d_{loc}}(q,q') = \big( 1+ \max\left(0,\sup\{r \geq 1 : [q]_r= [q']_r \}\right)\big) ^{-1}.$$

Furthermore, for $g,z \geq 0$ let $W(g,z)$ be the bi-variate generating function of $\sQ_{n,p}$ with weight $g$ per internal face and $\sqrt{z}$ per edge on the boundary,  that is
 \begin{eqnarray*} W(g,z) := \sum_{n,p \geq 0}\# {\sQ}_{n,p}g^n z^p.  \end{eqnarray*}
A closed form for $W$ can be found in~\cite{BG09}; the radius of convergence of $W$ in $g$ can be seen to be $1/12$, and in particular we have 
\begin{eqnarray}W_{c}(z) = W(1/12,z) = \frac{(1-8z)^{3/2}-1+12z}{24z^2}.\label{W}\end{eqnarray}
Via singularity analysis, one can deduce that
 \begin{eqnarray} [z^p] W_{c}(z)  \underset{p\to\infty}{\sim}  \frac{2}{ \sqrt{\pi}} 8^p p^{-5/2}. \label{eq:asympZp} \end{eqnarray}
 
As before we define the Boltzmann measure $\nu_p$ on the set $\sQ_p$ of all quadrangulations with perimeter $2p$: we set $\nu_p(\{q\})= \frac{1}{[z^p]W_{c}(z)} 12^{-n}$ for each quadrangulation $q$ having perimeter $2p$ and area $n$; notice that $\nu_p|_{\sQ_{n,p}}$ is the uniform probability measure. Throughout the paper we shall write $\BQ_p$ for a random quadrangulation of perimeter $2p$ distributed according to~$\nu_p$.

\subsubsection{Treed bridges}\label{treed bridges}
 A \emph{bridge of length $2p$} is a sequence of integers $ b=(x_0,\ldots,x_{2p-1})$ such that $x_0=0$ and $|x_{i+1}-x_i|=1$ for $i=0,\ldots, 2p-1$, where indices are considered modulo $2p$ (so that $x_{2p}=0$). Notice that, given a bridge $b=(x_0,\ldots,x_{2p-1})$, there are $p$ indices $d_1\leq \cdots \leq d_p$ such that $x_{d_i+1}=x_{d_i}-1$ (again, consider indices modulo $2p$); we call these indices \emph{down-steps} for $b$ and denote their set by $ \DS(b)$. 

Analogously, we define an \emph{infinite bridge} (and declare its length to be $\infty$) to be a doubly infinite sequence $b = (x_{i})_{i \in \mathbb{Z}}$ of integers such that $x_{0}=0$ and that  $|x_{i+1}-x_i|=1$ for each $i \in \mathbb{Z}$; again, we denote by $ \DS(b)$ the set of its down-steps, which may be finite or infinite.

A bridge of length $2p$ will be seen as a simple $2p$-cycle embedded in the plane, with a distinguished edge oriented so that the infinite face lies to its right; labels are assigned to its vertices so that the tail of the root edge has label $x_0=0$, and labels $x_1,\ldots, x_{2p-1}$ are assigned to subsequent vertices in the cycle according to the orientation given by the direction of the root edge. The same interpretation can be given to an infinite bridge, using a doubly infinite path (giving rise to two infinite faces in its planar embedding, which we see as the upper and lower half-planes) instead of a cycle; we order labels left-to-right (so that the tail of the root edge lies left of its head).

\begin{remark} Notice that, while labels in a labelled tree can vary by $\{-1,0,1\}$ between neighbouring vertices, they can only vary by $\{-1,+1\}$ along a bridge.\end{remark}

We introduced the notion of \emph{bridge} in order to discuss that of a \emph{treed bridge}, which we define as follows:
\begin{df}\label{unconstrained treed bridge}For $ p \in \{1,2,\ldots \} \cup \{ \infty\}$, a treed bridge of length $2p$ (or of infinite length in the case where $p=\infty$) is a pair $(b; T)$ such that $b=(x_0,\ldots,x_{2p-1})$ is a bridge of length $2p$ (or $b=(x_i)_{i\in\mathbb{Z}}$ is an infinite bridge if $p=\infty$) and $T$ is a function from $\DS(b)$ to $\LT$, such that $T(i)$ is in $\LT_{x_i}$. The size of a treed bridge $(b;T)$ is computed as $\sum_{i\in \DS(b)}|T(i)|$, i.e.~as the sum of the sizes of its trees.
\end{df}

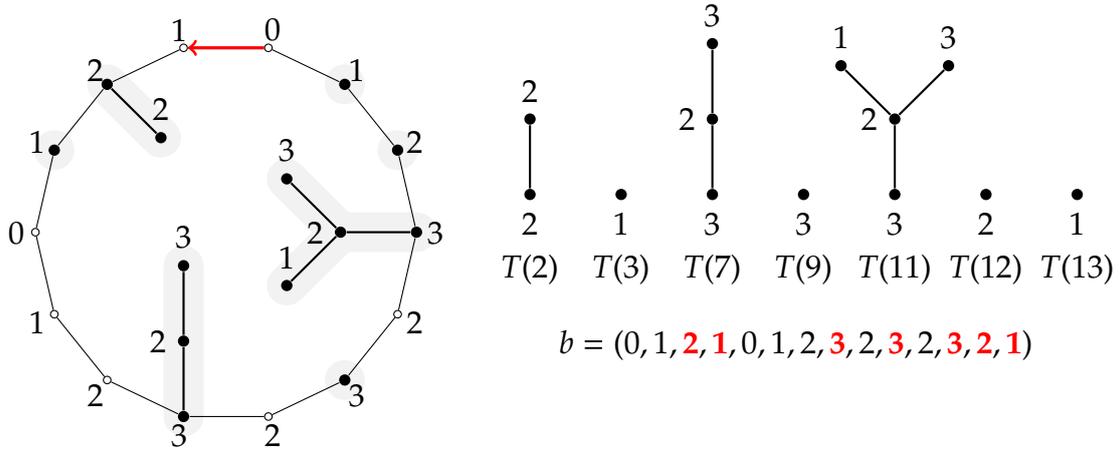
\begin{figure}\centering

\begin{tikzpicture}
\tikzstyle{real}=[inner sep=1.5pt, fill=black, circle]
\tikzstyle{map}=[red]
\def\bridge{{0,1,2,1,0,1,2,3,2,3,2,3,2,1,0}}

\node[draw=none ,minimum size=5cm,regular polygon,regular polygon sides=14] (a) {};
\node[draw=none ,minimum size=5.5cm,regular polygon,regular polygon sides=14] (b) {};
\foreach[evaluate={\lsx=int(\bridge[\x-1])}, evaluate={\ldx=int(\bridge[\x])}, evaluate={\y=int(\x)}] \x in {14,13,...,1} {
\pgfmathifthenelse{\ldx==\lsx-1}{"\noexpand\draw (a.corner \y) node[real] (\y) {};"}{"\noexpand\draw (a.corner \y) node[phantom] (\y) {};"};\pgfmathresult;
  \draw (b.corner \y) node {\lsx};}
  \foreach[evaluate={\target=int(\x-1)}] \x in {2,...,14} {
  \draw (\x) -- (\target); }
  \draw (14)--(1);
  \draw[root] (1)--(2);
 
 \node[real, below right of=3, label=2] (3-1) {};
 \draw[thick] (3)--(3-1);
 
 \node[real, above of=8, label=left:2] (8-1) {};
 \node[real, above of=8-1, label=3] (8-11) {};
 \draw[thick] (8)--(8-1)--(8-11);
 
  \node[real, left of=12, label=left:2] (12-1) {};
 \node[real, below left of=12-1, label=1] (12-11) {};
  \node[real, above left of=12-1, label=3] (12-12) {};
  \draw[thick] (12)--(12-1)--(12-11);
  \draw[thick] (12-1)--(12-12);

\begin{pgfonlayer}{edgelayer}
\draw[gray!10, line width=15pt, line cap=round] (3.center)--(3-1.center);
\draw[gray!10, line width=15pt, line cap=round] (8.center)--(8-11.center);
\draw[gray!10, line width=15pt, line cap=round] (12.center)--(12-1.center)--(12-11.center);
\draw[gray!10, line width=15pt, line cap=round] (12-1.center)--(12-12.center);
\fill[gray!10]  (4.center) circle (7.5pt);
\fill[gray!10]  (10.center) circle (7.5pt);
\fill[gray!10]  (13.center) circle (7.5pt);
\fill[gray!10]  (14.center) circle (7.5pt);
\end{pgfonlayer}

\begin{scope}
\node[real, label=below:2] (3) at (4,0.5) {};
 \node[real, above of=3, label=2] (3-1) {};
 \draw[thick] (3)--(3-1);

\node[real, label=below:1] (4) at (5.2,0.5) {};

\node[real, label=below:3] (8) at (6.4,0.5) {};
 \node[real, above of=8, label=left:2] (8-1) {};
 \node[real, above of=8-1, label=3] (8-11) {};
 \draw[thick] (8)--(8-1)--(8-11);

\node[real, label=below:3] (10) at (7.6,0.5) {};

\node[real, label=below:3] (12) at (8.8,0.5) {};
  \node[real, above of=12, label=left:2] (12-1) {};
 \node[real, above left of=12-1, label=1] (12-11) {};
  \node[real, above right of=12-1, label=3] (12-12) {};
  \draw[thick] (12)--(12-1)--(12-11);
  \draw[thick] (12-1)--(12-12);

\node[real, label=below:2] (13) at (10,0.5) {};

\node[real, label=below:1] (14) at (11.2,0.5) {};

\node at (4,-0.5) {$T(2)$};
\node at (5.2,-0.5) {$T(3)$};
\node at (6.4,-0.5) {$T(7)$};
\node at (7.6,-0.5) {$T(9)$};
\node at (8.8,-0.5) {$T(11)$};
\node at (10,-0.5) {$T(12)$};
\node at (11.2,-0.5) {$T(13)$};

\node at (7.5,-1.5) {$b=(0,1,\mathbf{\color{red} 2},\mathbf{\color{red} 1},0,1,2,\mathbf{\color{red} 3},2,\mathbf{\color{red} 3},2,\mathbf{\color{red} 3},\mathbf{\color{red} 2},\mathbf{\color{red} 1})$};
\end{scope}
\end{tikzpicture}
\caption{\label{treed bridge}A positive treed bridge $(b;T)$ of length 14 seen as a map.}

\end{figure}

Treed bridges -- like bridges -- have a clear geometric interpretation (see Figure~\ref{treed bridge}): each tree $T(i)$, for $i\in\DS(b)$, may be seen as being grafted on the $i$-th vertex of the cycle (or infinite path) which represents $b$; we embed each tree $T(i)$ \emph{inside} the cycle (or within the upper half-plane) so as to preserve its planar structure.

We call $\TB_{p}$ the set of all treed bridges of length $2p$ ($\TB_\infty$ is the set of all treed bridges of infinite length); $\TB_{n,p}$ is the set of all treed bridges of length $2p$ and size $n$. As we did for trees and labelled trees, we may define a local distance on the set $\cup_{p\in\mathbb{Z}^+\cup\{\infty\}}\TB_{p}$. Given a pair of treed bridges $B=(b;T)$ and $B'=(b';T')$ such that $b=(x_i)_{i\in I}$ and $b'=(x'_i)_{i\in I'}$, where $I$ and $I'$ are either $\mathbb{Z}$ or sets of the form $\{0,1,\ldots, 2p-1\}$, we declare that $ \mathrm{d_{loc}}(B,B') \leq (1+r)^{-1}$ if and only if $x'_i=x_i$ when $-r\leq i\leq r$ (as usual indices cycle modulo $2p$) and $[T(j)]_r=[T'(j)]_r$ for $j \in \DS(b)\cap[-r,r]$. 

As we did with labelled trees, we also define subsets $\TB_p^+$ for each $\TB_p$ (where $p\in\mathbb{Z}^+$): a \emph{positive treed bridge} is a treed bridge $(b; T)$, where $b=(x_0,\ldots,x_{2p-1})$, such that for each $i$ in $\DS(b)$ we have $x_i>0$ (that is, the bridge has no negative labels, though it may have more than one null label) and $T(i)\in\LT^+$.

  Our interest in these objects lies in the following constructions, variants of the classical Schaeffer bijection and based on~\cite{BDFG04, BG09}.

\subsection{Finite construction with unconstrained labels}\label{pmconstruction}

\begin{figure}\centering
\begin{tikzpicture}
\tikzstyle{real}=[inner sep=1.5pt, fill=black, circle, draw=white, thick]
\tikzstyle{map}=[red]
\def\bridge{{0,1,2,1,0,-1,0,-1,-2,-1,-2,-1,0,-1,0}}

\node[draw=none ,minimum size=5cm,regular polygon,regular polygon sides=14] (a) {};
\node[draw=none ,minimum size=5.5cm,regular polygon,regular polygon sides=14] (b) {};
\foreach[evaluate={\lsx=int(\bridge[\x-1])}, evaluate={\ldx=int(\bridge[\x])}, evaluate={\y=int(\x)}] \x in {14,13,...,1} {
\pgfmathifthenelse{\ldx==\lsx-1}{"\noexpand\draw (a.corner \y) node[real] (\y) {};"}{"\noexpand\draw (a.corner \y) node[phantom] (\y) {};"};\pgfmathresult;
  \draw (b.corner \y) node {\lsx};}
  \foreach[evaluate={\target=int(\x-1)}] \x in {2,...,14} {
  \draw (\x) -- (\target); }
  \draw (14)--(1);
  \draw[root] (1)--(2);
 
 \contourlength{1pt}
 \node[real, below right of=3, label=\contour{white}{2}] (3-1) {};
 \draw[thick] (3)--(3-1);

 \node[real, right of=5, label=above:\contour{white}{-1}] (5-1) {};
\node[real, right of=5-1, label=above:\contour{white}{-2}] (5-11) {};
 \draw[thick] (5)--(5-1)--(5-11); 
 
\node[real, above of=7, label=above:\contour{white}{1}] (7-1) {};
\node[real, above right of=7, label=right:\contour{white}{-1},xshift=4pt,yshift=-4pt] (7-2) {};
 \draw[thick] (7-1)--(7)--(7-2); 
 
  \node[real, below left of=13, label=\contour{white}{-1},xshift=-4pt,yshift=8pt] (13-1) {};
\draw[thick] (13)--(13-1);

\begin{pgfonlayer}{edgelayer}
\pic [angle radius=12pt, fill=red!10] {angle = 3-1--3--2};
\pic [angle radius=12pt, fill=red!10] {angle = 4--3--3-1};
\pic [angle radius=12pt, fill=red!10] {angle = 5--4--3};
\pic [angle radius=12pt, fill=red!10] {angle = 6--5--4};
\pic [angle radius=12pt, fill=red!10] {angle = 8--7--6};
\pic [angle radius=12pt, fill=red!10] {angle = 9--8--7};
\pic [angle radius=12pt, fill=red!10] {angle = 11--10--9};
\fill [fill=red!10] (3-1) circle (11pt);
\fill [fill=red!10] (5-1) circle (11pt);
\fill [fill=red!10] (5-11) circle (11pt);
\fill [fill=red!10] (7-1) circle (11pt);
\fill [fill=red!10] (7-2) circle (11pt);
\fill [fill=red!10] (13-1) circle (11pt);
\pic [angle radius=12pt, fill=red!10] {angle = 14--13--12};

\draw[white, line width=5pt, line cap=round] (3.center)--(3-1.center);
\draw[white, line width=5pt, line cap=round] (5.center)--(5-1.center)--(5-11.center);
\draw[white, line width=5pt, line cap=round] (7-1.center)--(7.center)--(7-2.center);
\draw[white, line width=5pt, line cap=round] (13.center)--(13-1.center);

  \foreach[evaluate={\target=int(\x-1)}] \x in {2,...,14} {
  \draw[white, line width=5pt, line cap=round] (\x) -- (\target); }
  \draw[white, line width=5pt, line cap=round] (14)--(1);
\end{pgfonlayer}

\node[circle, label=\contour{white}{-3}, label=below:\contour{white}{$\delta$}, draw=black, fill=red, inner sep=2pt] (delta) at (0.5,0) {};
\end{tikzpicture}
\begin{tikzpicture}
\tikzstyle{real}=[inner sep=1.5pt, fill=black, circle, draw=white, thick]
\tikzstyle{map}=[red]
\def\bridge{{0,1,2,1,0,-1,0,-1,-2,-1,-2,-1,0,-1,0}}
\node[draw=none ,minimum size=5cm,regular polygon,regular polygon sides=14] (a) {};
\node[draw=none ,minimum size=5.5cm,regular polygon,regular polygon sides=14] (b) {};
\foreach[evaluate={\lsx=int(\bridge[\x-1])}, evaluate={\ldx=int(\bridge[\x])}, evaluate={\y=int(\x)}] \x in {14,13,...,1} {
\pgfmathifthenelse{\ldx==\lsx-1}{"\noexpand\draw (a.corner \y) node[real] (\y) {};"}{"\noexpand\draw (a.corner \y) node[phantom] (\y) {};"};\pgfmathresult;
  \draw (b.corner \y) node {\lsx};}
  \foreach[evaluate={\target=int(\x-1)}] \x in {2,...,14} {
  \draw (\x) -- (\target); }
  \draw (14)--(1);
  \draw[root] (1)--(2);
 
 \contourlength{1pt}
 \node[real, below right of=3, label=\contour{white}{2}] (3-1) {};
 \draw[thick] (3)--(3-1);

 \node[real, right of=5, label=above:\contour{white}{-1}] (5-1) {};
\node[real, right of=5-1, label=above:\contour{white}{-2}] (5-11) {};
 \draw[thick] (5)--(5-1)--(5-11); 
 
\node[real, above of=7, label=above:\contour{white}{1}] (7-1) {};
\node[real, above right of=7, label=below:\contour{white}{-1},xshift=4pt,yshift=-4pt] (7-2) {};
 \draw[thick] (7-1)--(7)--(7-2); 
 
  \node[real, below left of=13, label=\contour{white}{-1},xshift=-4pt,yshift=8pt] (13-1) {};
\draw[thick] (13)--(13-1);

\node[circle, label={below:\contour{white}{-3}}, label=left:\contour{white}{$\delta$}, draw=black, fill=red, inner sep=2pt] (delta) at (-0.5,-0.3) {};

\begin{pgfonlayer}{edgelayer}
\fill[gray!10] [out=5, in=-30, looseness=6] (3.center) to (4.center) [out=0, in=0, bend right=25, looseness=1] to (3.center);
\fill[gray!10] [bend left=25] (5.center) to (7-2.center) [out=-100, in=-90, looseness=3.3] to (7.center) [in=0,out=0,looseness=1, bend left=25] to (8.center) [out=-60, in=-90, looseness=1.62] to (5-11.center) [bend left=10, looseness=1] to (13-1.center) [bend right=25, looseness=1] to (13.center) [bend right=20] to (5-1.center) [bend right=25] to (5.center);
\draw[thick, map, out=5, in=-30, looseness=5.5] (3) to (4);
\draw[map] (3-1) to (4);
\draw[thick, map, bend left=25] (3) to (4);
\draw[very thick, map, bend left=25, <-] (4) to (5);
\draw[thick, map, bend left=25] (5) to (5-1);
\draw[map, bend left=25] (5-1) to (5-11);
\draw[map] (5-11) to (delta);
\draw[map, out=-60, in=-40, looseness=2.5] (5-1) to (5-11);
\draw[thick, map, bend left=25] (5) to (7-2);
\draw[thick, map, out=120, in=120, looseness=3] (7) to (7-2);
\draw[map, bend left=25] (7-1) to (7);
\draw[map, bend left=25] (7) to (7-2);
\draw[map, out=0, in=-20, looseness=2] (7-2) to (5-11);
\draw[thick, map, bend left=25] (7) to (8);
\draw[thick, map, out=30, in=0, looseness=1.5] (8) to (5-11);
\draw[thick, map, in=10, out=110, looseness=1.2] (10) to (5-11);
\draw[thick, map, bend left=25] (13) to (13-1);
\draw[thick, map, bend right=10] (13-1) to (5-11);
\draw[thick, map, bend right=20] (13) to (5-1);
\end{pgfonlayer}

\end{tikzpicture}

\caption{\small\label{unconstrained bijection}A treed bridge with 19 corners and the construction of its corresponding quadrangulation via $\Phi^\bullet$.}
\end{figure}
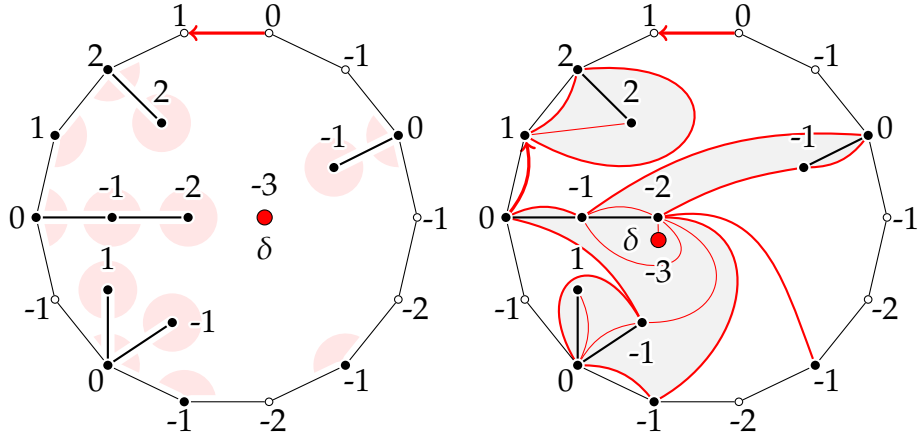

\begin{prop}\label{pmbijection}There is an explicit bijection $\Phi^\bullet$ between the set $\TB_{n,p}$ of all treed bridges of length $2p$ and size $n$, and the set $\sQ_{n,p}^\bullet$ of all pointed quadrangulations with a boundary having perimeter $2p$ and area $n$, that is
$$\sQ_{n,p}^{\bullet}=\left\{(q,\delta) \mid q \in \sQ_{n,p}, \delta\mbox{ is a vertex of }q\right\}.$$
\end{prop}

The construction is described in~\cite{BG09,CMboundary}; we briefly go through it here, referring the reader to~\cite{CMboundary} for a description of its inverse. 

Consider the geometric interpretation of a treed bridge $(b;T)\in \TB_{n,p}$ as described in the previous section: $(b;T)$ is seen as a rooted map with two faces, one of which is unbounded; the root is oriented counterclockwise. Consider the \emph{counterclockwise contour} of the bounded face, which contains the trees: it defines a (cyclic) sequence of corners, a corner being the angular region corresponding to a pair of edges sharing a vertex (see Figure~\ref{unconstrained bijection}); notice that each vertex may be adjacent to more than one corner. Since they will be preserved by the construction we are about to describe, we call vertices belonging to the trees in $T(\DS(b))$ \emph{real vertices}; other vertices, which belong to the $2p$-cycle but do not correspond to down-steps, are called \emph{phantom vertices}. The classification naturally extends to corners; we also consider corners as being labelled, by having them inherit the label of their vertex.

 Call $k$ the minimum label appearing on \emph{real} vertices of $(b;T)$, and add a vertex labelled $k-1$, which we call $\delta$, within the bounded face of the treed bridge. Now apply the standard Schaeffer construction to the labelled map at hand: link each real corner $c$ bearing label $h>k$ to the next real corner labelled $h-1$ in the counterclockwise contour of the bounded face; each real corner labelled $k$ is then linked to the added vertex $\delta$; notice that all new edges can be drawn in such a way that they do not cross each other. Finally, all phantom vertices, all labels and all original edges of the treed bridge are erased.
 
What one obtains (see~\cite{CMboundary}) is a quadrangulation with a boundary pointed in the vertex $\delta$. It is not hard to see that the boundary has length $2p$ and that each edge of the $2p$-cycle can be made to correspond to an edge of the boundary; we root the map in the edge corresponding to the original root of the treed bridge (preserving its original orientation) to obtain $\Phi^\bullet((b;T))$: see Figure~\ref{unconstrained bijection}.
 
\begin{rem}\label{pmremark}
The construction $\Phi^\bullet$ is such that:
\begin{itemize}
    \item one may identify the real vertices of $(b;T)$ with the vertices of $\Phi^\bullet((b;T))$, if one excludes the vertex $\delta$ in which the map is pointed;
    \item for each real vertex $x$ of $(b;T)$, if we call $l(\delta)$ the label the construction gives to the additional vertex $\delta$, we have $\mathrm{d_{gr}}(x,\delta)= l(x)-l(\delta)$, where the graph distance $\mathrm{d_{gr}}$ is taken in the map $\Phi^\bullet((b;T))$;
    \item the edges of the embedded cycle representing $b$ correspond to the edges on the boundary of $ \Phi^\bullet((b;T))$.
\end{itemize}
\end{rem}

The map $\Phi^\bullet$ is naturally defined on the whole set $\cup_{p>0} \TB_p$ of finite treed bridges, and takes values in the set $\sQ^\bullet$ of (rooted) pointed quadrangulations with a boundary. Composing with the forgetful map from $\sQ^\bullet$ to $\sQ$ (which simply takes elements of $\sQ^\bullet$ and forgets their pointing) we get a mapping $\Phi$ which sends treed bridges in $\cup_{p>0} \TB_p$ to the set $\sQ$ of rooted quadrangulations with a boundary.

\subsection{Finite construction with positive labels}\label{pconstruction}

\begin{figure}\centering
\begin{tikzpicture}
\tikzstyle{real}=[inner sep=1.5pt, fill=black, circle, draw=white, thick]
\tikzstyle{map}=[red]
\def\bridge{{0,1,2,1,0,1,2,3,2,3,2,3,2,1,0}}

\node[draw=none ,minimum size=5cm,regular polygon,regular polygon sides=14] (a) {};
\node[draw=none ,minimum size=5.5cm,regular polygon,regular polygon sides=14] (b) {};
\foreach[evaluate={\lsx=int(\bridge[\x-1])}, evaluate={\ldx=int(\bridge[\x])}, evaluate={\y=int(\x)}] \x in {14,13,...,1} {
\pgfmathifthenelse{\ldx==\lsx-1}{"\noexpand\draw (a.corner \y) node[real] (\y) {};"}{"\noexpand\draw (a.corner \y) node[phantom] (\y) {};"};\pgfmathresult;
  \draw (b.corner \y) node {\lsx};}
  \foreach[evaluate={\target=int(\x-1)}] \x in {2,...,14} {
  \draw (\x) -- (\target); }
  \draw (14)--(1);
  \draw[root] (1)--(2);
 
 \contourlength{1pt}
 \node[real, below right of=3, label=\contour{white}{2}] (3-1) {};
 \draw[thick] (3)--(3-1);
 
 \node[real, above of=8, label=left:\contour{white}{2}] (8-1) {};
 \node[real, above of=8-1, label=\contour{white}{3}] (8-11) {};
 \draw[thick] (8)--(8-1)--(8-11);
  
  \node[real, left of=12, label=left:\contour{white}{2}] (12-1) {};
 \node[real, below left of=12-1, label=\contour{white}{1}] (12-11) {};
  \node[real, above left of=12-1, label=\contour{white}{3}] (12-12) {};
  \draw[thick] (12)--(12-1)--(12-11);
  \draw[thick] (12-1)--(12-12);
  
\begin{pgfonlayer}{edgelayer}
\pic [angle radius=12pt, fill=red!10] {angle = 3-1--3--2};
\fill [fill=red!10] (3-1) circle (11pt);
\pic [angle radius=12pt, fill=red!10] {angle = 4--3--3-1};
\pic [angle radius=12pt, fill=red!10] {angle = 5--4--3};
\pic [angle radius=12pt, fill=red!10] {angle = 8-1--8--7};
\pic [angle radius=12pt, fill=red!10] {angle = 8-11--8-1--8};
\fill [fill=red!10] (8-11) circle (11pt);
\pic [angle radius=12pt, fill=red!10] {angle = 8--8-1--8-11};
\pic [angle radius=12pt, fill=red!10] {angle = 9--8--8-1};
\pic [angle radius=12pt, fill=red!10] {angle = 11--10--9};
\pic [angle radius=12pt, fill=red!10] {angle = 12-1--12--11};
\pic [angle radius=12pt, fill=red!10] {angle = 12-11--12-1--12};
\fill [fill=red!10] (12-11) circle (11pt);
\pic [angle radius=12pt, fill=red!10] {angle = 12-12--12-1--12-11};
\fill [fill=red!10] (12-12) circle (11pt);
\pic [angle radius=12pt, fill=red!10] {angle = 12--12-1--12-12};
\pic [angle radius=12pt, fill=red!10] {angle = 13--12--12-1};
\pic [angle radius=12pt, fill=red!10] {angle = 14--13--12};
\pic [angle radius=12pt, fill=red!10] {angle = 1--14--13};

\draw[white, line width=5pt, line cap=round] (3.center)--(3-1.center);
\draw[white, line width=5pt, line cap=round] (8.center)--(8-11.center);
\draw[white, line width=5pt, line cap=round] (12.center)--(12-1.center)--(12-11.center);
\draw[white, line width=5pt, line cap=round] (12-1.center)--(12-12.center);

  \foreach[evaluate={\target=int(\x-1)}] \x in {2,...,14} {
  \draw[white, line width=5pt, line cap=round] (\x) -- (\target); }
  \draw[white, line width=5pt, line cap=round] (14)--(1);
\end{pgfonlayer}

\node[circle, label=0, label=below:$\delta$, draw=black, fill=red, inner sep=2pt] (delta) at (0,0.5) {};
\end{tikzpicture}
\begin{tikzpicture}
\tikzstyle{real}=[inner sep=1.5pt, fill=black, circle, draw=white, thick]
\tikzstyle{map}=[red]
\def\bridge{{0,1,2,1,0,1,2,3,2,3,2,3,2,1,0}}
\node[draw=none ,minimum size=5cm,regular polygon,regular polygon sides=14] (a) {};
\node[draw=none ,minimum size=5.5cm,regular polygon,regular polygon sides=14] (b) {};
\foreach[evaluate={\lsx=int(\bridge[\x-1])}, evaluate={\ldx=int(\bridge[\x])}, evaluate={\y=int(\x)}] \x in {14,13,...,1} {
\pgfmathifthenelse{\ldx==\lsx-1}{"\noexpand\draw (a.corner \y) node[real] (\y) {};"}{"\noexpand\draw (a.corner \y) node[phantom] (\y) {};"};\pgfmathresult;
  \draw (b.corner \y) node {\lsx};}
  \foreach[evaluate={\target=int(\x-1)}] \x in {2,...,14} {
  \draw (\x) -- (\target); }
  \draw (14)--(1);
  \draw[root] (1)--(2);
 
 \contourlength{1pt}
 \node[real, below right of=3, label=\contour{white}{2}] (3-1) {};
 \draw[thick] (3)--(3-1);
 
 \node[real, above of=8, label=left:\contour{white}{2}] (8-1) {};
 \node[real, above of=8-1, label=\contour{white}{3}] (8-11) {};
 \draw[thick] (8)--(8-1)--(8-11);
  
  \node[real, left of=12, label=left:\contour{white}{2}] (12-1) {};
 \node[real, below left of=12-1, label=\contour{white}{1}] (12-11) {};
  \node[real, above left of=12-1, label={[xshift=3pt, yshift=-2pt]\contour{white}{3}}] (12-12) {};
  \draw[thick] (12)--(12-1)--(12-11);
  \draw[thick] (12-1)--(12-12);
  
  \node[label=0, real] (delta) at (0,0.5) {};

\begin{pgfonlayer}{edgelayer}
\fill[gray!10] [out=5, in=-30, looseness=6] (3.center) to (4.center) [out=0, in=0, bend right=10] to (3.center);
\fill[gray!10] [bend left=25] (8) to (8-1) to (8-1.center) [bend left=0, in=120, out=100, looseness=2.65] to (12-11) to (12-11.center) [in=140, out=0, looseness=0.9] to (delta.center) [bend left=30] to (14) to (14.center) [bend right=25] to (13) to (13.center) to (12) to (12.center) [bend left=25] to (12-1) to (12-1.center) [bend left=28] to (8);
  \draw[thick, red, out=5, in=-30, looseness=5.5] (3) to (4);
    \draw[map] (3-1) to (4);
\draw[thick, red, bend left=15] (3) to (4);
\draw[very thick, red, out=-60, in=190, looseness=1.5,<-] (4) to (delta);
\draw[thick, red, bend left=25] (8) to (8-1);
\draw[thick, red, out=130, in=150, looseness=2.5] (8-1) to (12-11);
\draw[map, bend left=25] (8-11) to (8-1);
\draw[map, bend right=25] (8-1) to (12-11);
\draw[thick, red, bend right=25] (8) to (12-1);
\draw[thick, red] (10) to (12-1);
\draw[thick, red, bend left=25] (12) to (12-1);
\draw[map, bend left=25] (12-1) to (12-11);
\draw[thick, red, out=120, in=-90] (12-11) to (delta);
\draw[map, out=180, in=200, looseness=1.5] (12-1) to (14);
\draw[map, bend left=25] (12-12) to (12-1);
\draw[map] (12-1) to (14);
\draw[thick, red, bend left=25] (12) to (13);
\draw[thick, red, bend left=25] (13) to (14);
\draw[thick, red, bend right=25] (14) to (delta);
\end{pgfonlayer}

\end{tikzpicture}

\caption{\small\label{positive bijection}The positive treed bridge $(b;T)$ from Figure~\ref{treed bridge} is depicted with its 19 corners, which become the 19 edges of $\Phi((b;T))$; notice that the quadrangulation is the same as that of Figure~\ref{unconstrained bijection}.}
\end{figure}
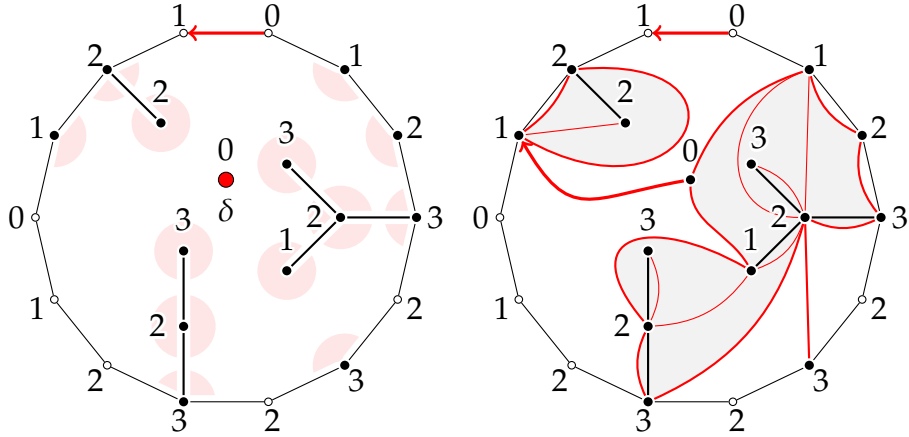

Consider the restriction of $\Phi$ (as defined above) to the set $\cup_{p>0}\TB_p^+$ of finite positive treed bridges. 

\medskip
Notice that, if we apply $\Phi^\bullet$ to a positive treed bridge $(b;T)$, the label of the added vertex $\delta$ is~0: the trees in $T(\DS(b))$ are in $\LT^+$, so their vertices bear only strictly positive labels; furthermore, $2p-1$ is always a down-step for any positive bridge of length $2p$; the root of the tree $T(2p-1)$ is thus always labelled $1$, hence the minimum label that does not appear on real vertices is $0$. Since the root vertex of the quadrangulation $\Phi^\bullet((b;T))$ corresponds to the root vertex of the treed bridge $(b;T)$ and thus has the same (null) label, the added vertex $\delta$ in which $\Phi^\bullet((b;T))$ is pointed must also be its root vertex; hence no information is lost by forgetting the pointing and taking $\Phi((b;T))$ instead of $\Phi^\bullet((b;T))$.

As a consequence we have the following:

\begin{prop} The restriction of $\Phi$ to the set $\TB^{+}_{n,p}$ of all positive treed bridges of length $2p$ and size $n$ is a bijection with the set $\sQ_{n,p}$ of rooted quadrangulations with a boundary of length $2p$ and area $n$.
\end{prop}

Remark~\ref{pmremark} is updated as follows: 

\begin{rem}\label{premark} When applying the construction $\Phi$ to a (finite) positive treed bridge,
\begin{itemize}
    \item there is a correspondence between real vertices of the positive treed bridge $(b;T)$ and vertices of its image $\Phi((b;T))$, if one excludes its root vertex;
    \item the label of each real vertex of $(b;T)$ represents its distance from the root vertex in the quadrangulation $\Phi((b;T))$;
    \item the edges of the embedded cycle representing $b$ correspond to the edges on the boundary of $ \Phi((b;T))$; as a consequence, the bridge labels $(x_0,\ldots,x_{2p-1})$ correspond to the distances from the root vertex of vertices on the boundary of $\Phi((b;T))$, read counterclockwise along its contour.
\end{itemize}
\end{rem}

\subsection{Infinite construction}\label{infinite construction}

Notice that the construction $\Phi$ can easily be extended to certain classes of treed bridges with an infinite boundary in such a way that it still yields elements of $\sQ$.

Let $\TB^{-\infty}$ be the set of all infinite treed bridges $(b;T)$, with $b=(x_i)_{i\in\mathbb{Z}}$, such that $\liminf_{i\to\pm\infty}x_i=-\infty$. Given such a bridge $(b;T)$, we may order its real corners according to the left-to-right contour of the upper face (similarly to how we took the counterclockwise contour of the inner face in the finite case). We can list real corners of $(b;T)$, in this order, as a sequence $(c_i)_{i\in I}$, where $I$ is either $\mathbb{Z}$ (the sequence is two-ended if there is an infinite number of real corners on each side of the root vertex, as will essentially be the case for all ``unconstrained'' treed bridges within this paper) or $\mathbb{N}$ (which covers the case where only a finite number of real corners lie left of the root vertex). Notice that $\liminf_{i\to\infty}x_i=-\infty$ implies that there is always an infinite number of real corners lying right of the root vertex, since $\DS(b)\cap \mathbb{N}$ must be infinite.

We now define $\Phi((b;T))$ by analogy with the construction performed on finite treed bridges: for each $i$, set $j=\min\{h>i \mid l(c_h)=l(c_i)-1\}$ (such a minimum always exists, since the requirement on bridge labels implies $\liminf_{i\to+\infty}l(c_i)=-\infty$); link the real corner $c_i$ to the real corner $c_j$, which we shall call the \emph{successor} of $c_i$ and denote by $s(c_i)$. The map obtained after erasing all phantom vertices and all edges of the original treed bridge is then rooted as in Section~\ref{pmconstruction} to obtain $\Phi((b;T))$, which is a locally finite (rooted) quadrangulation with an infinite boundary.

\begin{rem}\label{continuity} Notice that the mapping $\Phi:\cup_{p>0}\TB_p\cup\TB^{-\infty}\to\sQ$ is continuous. This is stated in~\cite{CMboundary} and exploited in order to construct the uniform infinite quadrangulation with an infinite general boundary from an infinite labelled treed bridge (see Section~\ref{Boltzmann} and~\cite[Section 6]{CMboundary}). A full proof can be found in~\cite{thesis}.\end{rem}

The bijection presented in Section~\ref{pconstruction} induces us to also consider an extension of $\Phi$ to the set of all positive treed bridges of infinite length where each label appears a finite number of times, which we call $\TB_\infty^+$. Given a treed bridge $(b;T)$ in $\TB_\infty^+$, we list its real corners as a sequence $(c_i)_{i\in I}$ ordered according to the left-to-right contour of the upper face, where this time $I$ is either $\mathbb{Z}$ or $\mathbb{Z}_{\leq 0}$ (the fact that each label appears a finite number of times implies divergence to $+\infty$ of labels along the boundary, hence $\DS(b)\cap \mathbb{Z}_{\leq 0}$ is infinite). Note that, while we shall mostly consider positive treed bridges for which we can take $I=\mathbb{Z}$, we will see some examples of positive treed bridges whose sequence of real corners is one-ended in Sections~\ref{decomposition} and~\ref{leftmost} (see Corollary~\ref{laws of two sides of gammaright} and Proposition~\ref{law of middle section}).

We build $\Phi((b;T))$ exactly as described above (joining each corner to its successor), except that this time the set $\{h>i \mid l(c_h)=l(c_i)-1\}$ might very well be empty; when this is the case and $l(c_i)>1$, we consider $j'=\min\{h \mid l(c_h)=l(c_i)-1\}$ and set $s(c_i)=c_{j'}$; finally, we add an extra vertex labelled 0 within the upper face of the treed bridge, setting its one corner to be the successor of every real corner labelled 1. Joining each real corner to their successor, erasing phantom vertices and original edges of the treed bridge and rooting appropriately yields a quadrangulation with an infinite boundary $\Phi((b;T))$; notice that all properties from Remark~\ref{premark} still hold: in particular, labels of real vertices in $(b;T)$ represent distances from the root vertex, so that the condition on the number of appearances of each label entails local finiteness of $\Phi((b;T))$, which thus belongs to the set $\sQ$.

Let $\TB^+=\bigcup_{p\in\mathbb{Z}^+\cup\{\infty\}}\TB^+_p$ be the set of all positive treed bridges with a finite or infinite perimeter, such that each label appears a finite number of times. The question of whether $\Phi$ is continuous (for the local metric) as a function from $\TB^+$ to $\sQ$ will be very relevant (continuity would allow us to proceed as in~\cite{CMboundary}, whose proofs of convergence rely on Remark~\ref{continuity}); we have the following:

\begin{rem}\label{non continuity}
The mapping $\Phi$ is \emph{not} continuous (for the local metric) on $\TB^+$; in fact, it is discontinuous at every point of $\TB_\infty^+$.

The main problem comes from the possible occurrence of small labels far from the root of a treed bridge: two treed bridges $B$ and $B'$ may be very near for the local metric, in the sense that their bridges and trees coincide up to a great distance from the root (say $[B]_r=[B']_r$ for some large $r$); if, however, a (real) vertex labelled 1 appears in $B\setminus[B]_r$ and does not appear in $B'\setminus [B']_r$, then $[\Phi(B)]_1\neq [\Phi(B')]_1$, so the two images are far apart.

This problem is exactly the same as that arising for the definition of the UIPQ via the positive bijection from~\cite{CD06}, and needs to be addressed as Ménard did in~\cite{Men08} by giving bounds for the probability that small labels appear very far from the root, a statement which will be made precise in Section~\ref{control on small labels}.
\end{rem}

\section{The new construction}\label{the new construction}
Before introducing our new construction based on the bijection from Section~\ref{pconstruction}, we shall briefly review the construction of the UIHPQ given in~\cite{CMboundary}. We will then prove that the UIHPQ can alternatively be obtained as the local limit of Boltzmann quadrangulations with a boundary whose perimeter is sent to infinity, without the need for a double limit such as that of \eqref{intro:UIHPQ as double limit}, which will render our new construction of the UIHPQ considerably simpler.
\subsection{Convergences to the UIHPQ}\label{Boltzmann}
The construction of the UIHPQ given in~\cite{CMboundary} is illustrated by the following commutative diagram (see~\cite{CMboundary} for the proof), where the second line is obtained by applying the (unpointed) Schaeffer mapping $\Phi$ to the first one:
\begin{equation}\label{diagram} \begin{array}{ccccc}
  \B_{n,p}^\pm &\xrightarrow[n\to\infty]{(d)}& \B_{\infty,p}^\pm &\xrightarrow[p\to\infty]{(d)} & \B_{\infty}^\pm\\ 
  \Phi\downarrow\phantom{\Phi} && \Phi\downarrow\phantom{\Phi} && \Phi\downarrow\phantom{\Phi} \\
  \Q_{n,p} &\xrightarrow[n\to\infty]{(d)}& \Q_{\infty,p} &\xrightarrow[p\to\infty]{(d)} & \UIHPQ\\
 \end{array}
\end{equation}
In this diagram, $ \B_{n,p}^\pm$ is a uniform treed bridge of length $2p$ and size $n$ (with no positivity constraint on the labels). As per Proposition~\ref{pmbijection}, its image via $\Phi^\bullet$ is $\Q^\bullet_{n,p}$, a uniformly random pointed quadrangulation with a boundary of perimeter $2p$ and area $n$ (i.e.~a uniform element of $\sQ^\bullet_{n,p}$); forgetting the pointing of $\Q^\bullet_{n,p}$ yields the random quadrangulation with a boundary $\Q_{n,p}=\Phi(\B_{n,p}^\pm)$. Notice now that, since each quadrangulation in $\sQ_{n,p}$ has exactly $n+p+1$ vertices by Euler's formula, the law of $\Q_{n,p}$ is uniform on the set $\sQ_{n,p}$.

We shall not describe the intermediate bridge $ \B_{\infty,p}^\pm$ or its image $\Q_{\infty,p}$, since they are irrelevant for our purpose; we will, however, recall the definition of $\B_{\infty}^\pm$, which will be used below and again in Section~\ref{Extension of results from CMMinfini}:

\begin{df}\label{unconstrained infinite bridge}
The random treed bridge $\B_\infty^\pm$ is the pair $((X^\pm_i)_{i\in\mathbb{Z}};T^\pm)$, where $(X^\pm_i)_{i\geq 0}$ and $(X^\pm_{-i})_{i\geq 0}$ are two independent uniform simple random walks issued from $X^\pm_0=0$ and, conditionally on the bridge $X^\pm$, the trees $T^{\pm}(i)$, for $i\in\DS(X^\pm)$, are independent and distributed according to $\rho_{X^\pm_i}$.
\end{df} 

Notice that we have $\liminf_{i\rightarrow\pm\infty}X^\pm_i=-\infty$ almost surely, hence $\B^\pm_\infty$ almost surely belongs to the set $\TB^{-\infty}$ as defined in Section~\ref{infinite construction}. Given the first line of the diagram \eqref{diagram} (that is the convergence of random treed bridges) one can thus invoke the continuity of $\Phi$ on $\bigcup_{p>0}\TB_p\cup\TB^{-\infty}$ (see Remark~\ref{continuity}) to produce the convergences from the second line, thus deducing that
  \begin{eqnarray} \Phi( \B_{\infty}^\pm) = \UIHPQ,   \label	{eq:constructionCMboundary}\end{eqnarray}
  which is how the UIHPQ $\UIHPQ$ was introduced in~\cite{CMboundary}. 
  
  We now use the same argument to prove that the UIHPQ can alternatively be seen as the local limit in $p$ of the random Boltzmann quadrangulations $\BQ_{p}$ from  Section~\ref{quadrangulations}:
\begin{proposition}\label{boltzmann convergence}  The UIHPQ is the local limit of Boltzmann quadrangulations with a boundary of perimeter $p$, when $p$ is sent to infinity:
$$ \BQ_{p} \xrightarrow[p \to \infty]{(d)} \UIHPQ.$$
\end{proposition}
\proof For each $p>0$, we shall denote by $\BPQ_p$ the random pointed map obtained by  selecting a vertex of $\BQ_p$ uniformly at random (conditionally on the quadrangulation itself). Consider now the random treed bridge obtained by taking the pre-image of $\BPQ_p$ via the pointed Schaeffer construction $\Phi^\bullet$ from Section~\ref{pconstruction}: set $\B_{p}^\pm := (\Phi^\bullet)^{-1}(\BPQ_p)$.

We claim that $\B_{p}^\pm$ converges in law to $\B_{\infty}^\pm$, when $p$ is sent to infinity, for the local distance on the set of treed bridges; as above, by Remark~\ref{continuity} this entails the convergence of $\Phi(\B_p^\pm)$ to $\Phi(\B_\infty^\pm)=\UIHPQ$; since $\Phi(\B_p^\pm)=\BQ_p$ by construction, this would establish the proposition.

Recall that quadrangulations in $\sQ_{n,p}$ have $n+p+1$ vertices, so that for any treed bridge $(b;T)$ of length $2p$ the definition of the Boltzmann law (see Section~\ref{quadrangulations}) implies
\begin{eqnarray}\label{law of Bpmp}\P(\B^{\pm}_p=(b;T))=\P\left(\BPQ_p=\Phi^\bullet((b;T))\right) = \frac{12^{-\sum_{i\in\DS(b)}|T(i)|}}{[z^p]W_c(z)(\sum_{i\in\DS(b)}|T(i)|+p+1)}.\end{eqnarray}

Notice that bridges in $\TB_{p}$ can be represented as $(p+1)$-tuples $(b, \tau_1,\ldots,\tau_p)$, where $b=(x_0,\ldots,x_{2p-1})$ is any bridge of length $2p$ and $\tau_1,\ldots,\tau_p$ are labelled trees in $\LT_0$: one recovers the treed bridge as presented in Section~\ref{treed bridges} by grafting tree $\tau_i$ on the $i$-th down-step $d\in\DS(b)$ and shifting all of its labels by $x_d$. 

It is clear from \eqref{law of Bpmp} that, if we see $\B^\pm_p=(X^{p,\pm}, T^\pm_1,\ldots, T^\pm_p)$ as being presented this way, the trees $T^\pm_i$ (for $i=1,\ldots,p$) are independent of the random bridge $X^{p,\pm}$, which is distributed uniformly over all bridges of length $2p$. Similarly, one can present $\B^\pm_\infty$ as $((X^\pm_i)_{i\in\mathbb{Z}},(\theta_i)_{i\in\mathbb{Z}})$, where $(X^\pm_i)_{i\in\mathbb{Z}}$ is as in Definition~\ref{unconstrained infinite bridge}, and $\theta_i$ is distributed according to $\rho_0$; one recovers $T^{\pm}(j)$, for $j\in\DS(X^\pm)$, by shifting labels of $\theta_j$ by $X^\pm_j$.

The fact that the local limit of $X^{p,\pm}$ is the infinite two-sided simple random walk $X^\pm$ is a classical result; all we need in order to prove local convergence of $\B_{p}^\pm$ to $\B_{\infty}^\pm$ is thus the fact that the trees $T_{i}^\pm$ are asymptotically independent and distributed according to $\rho_{0}$.  

Take any non-negative continuous function $f$ on $p$-tuples in $\LT_0^p$; let $T^\pm_1,\ldots, T^\pm_p$ be the $p$-tuple of random trees of $\B^\pm_p$ and let $\theta_1,\ldots,\theta_p$ be $p$ independent trees of law $\rho_0$; notice that by combining the fact that $\rho_0(\{\tau\})=12^{-|\tau|}/2$ with \eqref{law of Bpmp} we get
$$\E[f(T^\pm_1,\ldots, T^\pm_p)]=\frac{2^p}{[z^p]W_c(z)}\sum_{\tau_1,\ldots,\tau_p \in \LT_0}\frac{f(\tau_1,\ldots,\tau_n)}{\sum_{i=1}^p|\tau_i|+p+1}{\prod_{i=1}^p \frac{12^{-|\tau_i|}}{2}}=\frac{2^p}{[z^p]W_c(z)}\E\left[\frac{f(\theta_1,\ldots,\theta_p)}{\sum_{i=1}^p|\theta_i|+p+1}\right].$$
Since the above expected value must be 1 when $f$ is taken to be the constant 1, we have $1=\frac{2^p}{[z^p]W_c(z)}\E\left[\frac{1}{\sum_{i=1}^p|\theta_i|+p+1}\right]$, and we can write
$$\E[f(T^\pm_1,\ldots, T^\pm_p)]={\E\left[\frac{f(\theta_1,\ldots,\theta_p)}{\sum_{i=1}^p|\theta_i|+p+1}\right]}{\E\left[\frac{1}{\sum_{i=1}^p|\theta_i|+p+1}\right]^{-1}}.$$

Choose any $k>0$; by renumbering trees $T^\pm_1,\ldots, T^\pm_p$ appropriately, we may suppose that, in addition to them being grafted in counterclockwise order within the bridge, the root vertex lies between the trees $T^\pm_{[k/2]}$ and $T^\pm_{[k/2]+1}$ (possibly coinciding with the root of $T^\pm_{[k/2]}$). Suppose that $f$ is a bounded non-negative function that only depends on $T^\pm_1,\ldots, T^\pm_k$; we want to investigate $\lim_{p\to\infty}\E[f(T^\pm_1,\ldots,T^\pm_k)]$, which can be rewritten as 
\begin{eqnarray}\label{limit}\lim_{p\to\infty}\E\left[\frac{f(\theta_1,\ldots,\theta_k)}{\sum_{i=1}^p|\theta_i|+p+1}\right]\E\left[\frac{1}{\sum_{i=1}^p|\theta_i|+p+1}\right]^{-1}\end{eqnarray}
where the trees $(\theta_i)_{i>0}$ are independent and distributed according to $\rho_0$.

Using the fact that $\lim_{p\to\infty} \P\left(\frac{\sum_{i=k+1}^p|\theta_i|+p+1}{\sum_{i=1}^p|\theta_i|+p+1}<1-\epsilon\right)=0$ for all $\epsilon>0$, straightforward computations lead to the identity
$$\lim_{p\to\infty}\E\left[\frac{f(\theta_1,\ldots,\theta_k)}{\sum_{i=1}^p|\theta_i|+p+1}\right]=\lim_{p\to\infty}\E\left[\frac{f(\theta_1,\ldots,\theta_k)}{\sum_{i=k+1}^p|\theta_i|+p+1}\right];$$
since numerator and denominator are now independent, by splitting the expected value and applying the identity again to the constant function 1 two copies of $\E\left[\left(\sum_{i=k+1}^p|\theta_i|+p+1\right)^{-1}\right]$ cancel out from the expression \eqref{limit}, thus proving that
$$\lim_{p\to\infty}\E[f(T^\pm_1,\ldots, T^\pm_k)]=\E\left[{f(\theta_1,\ldots,\theta_k)}\right].$$

The above identity, together with the convergence of bridges, implies the convergence in law for the local topology of the random treed bridges $\B_p^{\pm}$ to the random treed bridge $\B^\pm_\infty$, which was our claim; this completes the proof of the proposition.
\endproof
\subsection{The Boltzmann positive treed bridges}\label{Boltzmann quadrangulations and statement of the theorem}

Our aim is to present $\UIHPQ$ as the image via $\Phi$ of the random infinite positive treed bridge $\B_\infty$ mentioned in the Introduction. In view of Proposition~\ref{boltzmann convergence}, we start by describing the law of the finite positive treed bridge that corresponds to a Boltzmann quadrangulation of perimeter $2p$. Recall the transition probabilities $\mathbf{p}$ from \eqref{eq:defp}:
   \begin{eqnarray} \label{eq:p} \forall n\geq 0,\quad \mathbf{p}(n,n-1)=\frac{n}{2(n+2)}; \quad \mathbf{p}(n,n+1)=\frac{n+4}{2(n+2)}. \end{eqnarray}
 We will occasionally call $\P_x$ the law of a nearest neighbour random walk with transition probabilities given by $\mathbf{p}$, issued from $x$. The following property of $\mathbf{p}$, relating it to the quantity $w_n$ as defined in \eqref{w_k}, shall be especially useful: we have 
\begin{eqnarray}\label{eq:w=pp}w_{n} = 8  \mathbf{p}(n,n-1) \mathbf{p}(n-1,n)\end{eqnarray}
as can be easily checked by substituting the expression $\frac{2n(n+3)}{(n+1)(n+2)}$ for $w_n$.
 
This being said, we can now prove the following:
\begin{prop}\label{Bp} Given $p>0$, consider the bijection $\Phi|_{\TB_{p}^+}:\TB_p^+\to\sQ_p$ between the set of all positive treed bridges of length $2p$ and the set of all finite rooted quadrangulations with a boundary of perimeter $2p$; write $\Phi^{-1}$ for its inverse. Let $\BQ_p$ be a Boltzmann quadrangulation with a boundary of length $2p$ (which is a random variable taking values in $\sQ_p$), and set $\B_p:=\Phi^{-1}(\BQ_p)$. Then the random positive treed bridge $\B_p=(X^p,T^p)$ can be described as follows:
\begin{itemize}
\item the random bridge $X^p=(X^p_{0}, \ldots , X^p_{2p-1})$ has the same law as the initial segment of a nearest neighbour random walk with transition probabilities given by $\mathbf{p}$, issued from 0 and conditioned on hitting 0 at time $2p$;
\item conditionally on the bridge $X^p$, the random trees $T^p(i)$, for $i \in\DS(X^p)$, are independent and distributed according to $\rho^+_{X^{p}_i}$.
\end{itemize}
\label{prop:loibridgeboltz}
\end{prop}
\proof Suppose $(b ; T)$ is a given positive treed bridge with $b=(x_0,\ldots,x_{2p-1})$; using the fact that $\Phi|_{\TB_{n,p}^+}$ is a bijection with the set $\sQ_{n,p}$ and that the size of a treed bridge is the sum of the sizes of its trees, we have
$$ \mathbb{P}(\B_p=(b;T))=\P(\BQ_p=\Phi((b;T)))=\frac{1}{[z^p] W_{c}(z)} \prod_{i\in \DS(b)} 12^{-|T(i)|} = \frac{1}{[z^p] W_{c}(z)} \prod_{i\in \DS(b)} w_{x_{i}} \prod_{i \in \DS(b)} \rho_{x_i}^+(\{T(i)\}).$$

It is clear from the rightmost expression that the random trees $T^p(i)$ (for $i\in\DS(X^p)$) are independent and distributed according to $\rho_{X^p_i}^+$ conditionally on the labels of their roots. As for the law of the bridge, it follows from \eqref{eq:w=pp} that 
$$ \prod_{i\in \DS(b)} w_{x_{i}} =  8^p \prod_{i\in\DS(b)}  \mathbf{p}(x_{i}, x_{i}-1) \mathbf{p}(x_{i}-1, x_{i});$$
notice now that, in order for the bridge to ``close'' (i.e.~to attain the value 0 at time $2p$), there must be a bijection $f$ between $\DS(b)$ and $\{0,\ldots,2p-1\}\setminus \DS(b)$ such that $x_{f(i)}=x_i-1$ (each down-step from a label $x$ must have a corresponding up-step from $x-1$), hence 
$$8^p \prod_{i\in\DS(b)}  \mathbf{p}(x_{i}, x_{i}-1) \mathbf{p}(x_{i}-1, x_{i})=8^p \prod_{i\in\DS(b)}  \mathbf{p}(x_{i}, x_{i}-1)\prod_{i\notin \DS(b)}\mathbf{p}(x_{i}, x_{i}+1)
=8^p \prod_{i=0}^{2p-1}  \mathbf{p}(x_{i}, x_{i+1}),$$ 
and this completes the proof of the proposition.\endproof

As a byproduct of the proof of the preceding proposition one may compute the probability that $X_{2p}=0$ for a nearest neighbour random walk $(X_i)_{i\geq 0}$, issued from 0 and with transition probabilities $\mathbf{p}$;  we have obtained that $\P_0(X_{2p}=0)=8^{-p} [z^p]W_{c}(z)$; hence, using \eqref{eq:asympZp}, we have

\begin{eqnarray} \label{eq:0-0} \mathbb{P}_{0}(X_{2p}=0)  \underset{p\to\infty}{\sim}  \frac{2}{ \sqrt{\pi}} p^{-5/2},  \end{eqnarray}
which will be useful later.
  
\begin{remark} Notice that, since the bridge labels may be interpreted as distances between boundary vertices and the root vertex of $\BQ_p$, read along the contour of the map, invariance of their law by time-reversal is clear (though it is not evident in the description of the process given above). In other words, $(0, X^p_1, \ldots,X^p_{2p-1})$ has the same law as $(0, X^p_{2p-1},\ldots,X^p_{1})$.
\end{remark}

As we have seen with Proposition~\ref{boltzmann convergence}, the UIHPQ is the local limit in $p$ of Boltzmann quadrangulations $\BQ_p$; we wish to show that it has the same law as the image via $\Phi$ of the local limit $\B_\infty$ of the random treed bridges $\B_p$ from Proposition~\ref{prop:loibridgeboltz}, when $p$ is sent to infinity. We give now a description of $\B_\infty$, whose validity we shall prove in the following section:

\begin{prop}\label{limitinp} The local limit (in law) of the sequence of random positive treed bridges $\B_p$, for $p\rightarrow\infty$, is the random positive treed bridge $\B_{\infty}=(X; T)$ such that:
\begin{itemize}
\item the two process $(X_{i})_{i \geq 0}$ and $(X_{-i})_{i \geq 0}$ are independent nearest neighbour random walks with transition probabilities given by $\mathbf{p}$, issued from $X_0=0$;
\item conditionally on the bridge $X$, the random trees $T(i)$, for $i\in\DS(X)$, are independent and distributed according to $\rho^+_{X_i}$.
\end{itemize}
\end{prop}

We then claim that the UIHPQ can be constructed from $\B_\infty$ via $\Phi$ as claimed in \eqref{eq:intronewconstruction}:
\begin{theorem}[The new construction of the UIHPQ]\label{new construction} We have $\UIHPQ =  \Phi(\B_{\infty})$ in distribution.
\end{theorem}

Establishing such a theorem given Proposition~\ref{limitinp} and Proposition~\ref{boltzmann convergence} would be straightforward if $\Phi$ were continuous on $\TB^+$, which -- as we have seen in Remark~\ref{non continuity} -- is not the case, even if we restrict ourselves to a set of probability 1 under the law of $\B_\infty$. As anticipated, we need the same kind of computation that Ménard used in~\cite{Men08} to show that the two constructions of the UIPQ (from~\cite{CD06} and~\cite{Kri05}) are equivalent; we shall set it up in Section~\ref{control on small labels}.

\subsection{$ \mathcal{B}_{\infty}$ is the local limit of the $ \mathcal{B}_{p}$'s}

This section is devoted to showing Proposition~\ref{limitinp}. Since the trees in $\B_p$ and $\B_\infty$ are conditionally independent given the bridges, with laws independent of $p$, all we need to show is that the local limit of the bridge $X^p=(X^p_0,\ldots,X^p_{2p-1})$, which is distributed as the initial segment of a nearest neighbour random walk with law $\mathbb{P}_0$ conditioned on hitting 0 after time $2p$, is simply an infinite bridge $(X_i)_{i\in\mathbb{Z}}$, where $(X_i)_{i\geq 0}$ and $(X_{-i})_{i\geq 0}$ are independent nearest neighbour random walks with law $\P_0$. In order to show this, we first prove a lemma giving an estimate for the probability that such a random walk transitions from $i$ to $j$ in time $p$.

\begin{lemma}  \label{lem:asympbridge} Let $i,j$ and $p$ be non-negative integers such that $i+j$ has the same parity as $p$; call $\S^{(p)}(i,j)$ the probability that a walk with transition probabilities given by $\mathbf{p}$ moves from $i$ to $j$ in time $p$, that is
$$\S^{(p)}(i,j) = \sum_{}\prod_{k=0}^{p-1}  \mathbf{p}(x_{k},x_{k+1})$$
where the sum is taken over all sequences $(x_0,\ldots, x_p)$ of non-negative integers such that $x_0=i$, $x_p=j$ and $|x_{k+1}-x_k|=1$ (for $k=0,\ldots,p-1$).
Then, when $p$ is sent to infinity (along values having the same parity as $i+j$), we have
\begin{eqnarray}\label{S(i,j,p)} \S^{(p)}(i,j) \sim \frac{1}{ 6\sqrt{\pi}} (j+1)(j+2)^2(j+3)  p^{-5/2}.\end{eqnarray}
\end{lemma}
\begin{proof}Notice that the right hand side of \eqref{S(i,j,p)} only depends on $j$, and that we have already dealt with the case $i=j=0$, see \eqref{eq:0-0}. For $\S^{(p)}(i,0)$, one has the recursive decomposition
$$\S^{(p)}(i,0) = \S^{(p-1)}(i+1,0) \mathbf{p}(i,i+1) + \S^{(p-1)}(i-1,0) \mathbf{p}(i,i-1) \mathbf{1}_{i \geq 1}$$
which inductively yields that for all $i\geq 0$ we have $\S^{(p)}(i,0)\sim \S^{(p)}(0,0) \sim \frac{2}{ \sqrt{\pi}} p^{-5/2}$ as $p\to\infty$ along values with the right parity.

An analogous expression can be written for $\S^{(p)}(i,j)$, using recurrence in $j$:
$$\S^{(p)}(i,j) =  \S^{(p-1)}(i,j+1) \mathbf{p}(j+1,j) + \S^{(p-1)}(i,j-1) \mathbf{p}(j-1,j) \mathbf{1}_{j \geq 1};$$ 
this time, one obtains by induction on $j$ that $\S^{(p)}(i,j) \sim \frac{2}{ \sqrt{\pi}}C(j)p^{-5/2}$ (along values of $p$ having the same parity as $i+j$), where $C(j)$ only depends on $j$ and satisfies $C(0)= 1$, as well as $C(j) = \mathbf{p}(j+1,j) C(j+1) + \mathbf{p}(j-1,j) C(j-1)$ for $j\geq 1$. The solution for $C(j)$ is indeed
$$C(j) = \frac{1}{ 12}(j+1)(j+2)^2(j+3) $$
as can be checked by induction.
\end{proof}

In fact, all that we shall use from the above lemma is a property of the quantity $C(j)$ that is easily deduced from the recursion
$C(j) = \mathbf{p}(j+1,j) C(j+1) + \mathbf{p}(j-1,j) C(j-1)$ shown within the proof of the lemma,
even without the explicit expression for $C(j)$: dividing both terms of the equality by $C(j)$ yields $\mathbf{p}(j+1,j) \frac{C(j+1)}{C(j)}=1-\mathbf{p}(j-1,j)\frac{C(j-1)}{C(j)}$, which by induction (using the fact that $C(0)=1$) implies 
\begin{eqnarray}\label{telescopic}\mathbf{p}(j,j-1) \frac{C(j)}{C(j-1)}=\mathbf{p}(j-1,j).\end{eqnarray}

\begin{proof}[Proof of Proposition~\ref{limitinp}]
Let $(l_{-r},\ldots,l_0,\ldots,l_r)$ be a sequence of non-negative integers such that $l_0=0$, $l_{-r}=i$, $l_r=j$ and $|l_{h+1}-l_h|=1$ for all $h$ such that $-r\leq h<r$; we wish to compute the limit in $p$ of the probability that $(X^p_i)_{-r\leq i\leq r}=(l_i)_{-r\leq i\leq r}$, where as usual indices in the sequence $X^p$ are read modulo $2p$. 

We can express such a probability, thanks to the description of $X^p$ given in Proposition~\ref{Bp}, as the probability that $X_i=l_i$ and that $X_{2p-i}=l_{-i}$ for $i=0,\ldots,r$, conditioned on the fact that $X_{2p}=0$; since the probability that $X_{2p}=0$ is $\S^{(2p)}(0,0)$, the limit of $\mathbb{P}\left((X^p_i)_{-r\leq i\leq r}=(l_i)_{-r\leq i\leq r}\right)$ as $p\to\infty$ can be expressed as
\begin{eqnarray}\label{eq}\lim_{p\rightarrow\infty}\left(\prod_{i=-r}^{r-1}\mathbf{p}(l_{i},l_{i+1})\right)\frac{\S^{(2p-2r)}(l_r,l_{-r})}{\S^{(2p)}(0,0)}=C(l_{-r})\prod_{i=-r+1}^{0}\mathbf{p}(l_{i-1},l_{i}) \prod_{i=0}^{r-1}\mathbf{p}(l_{i},l_{i+1}).\end{eqnarray}

We know that $\frac{C(l_{i-1})}{C(l_i)}\mathbf{p}(l_{i-1},l_{i})=\mathbf{p}(l_{i},l_{i-1})$ from \eqref{telescopic}; hence
$$C(l_{-r})\prod_{i=-r+1}^{0}\mathbf{p}(l_{i-1},l_{i})=\prod_{i=-r+1}^{0}\frac{C(l_{i-1})}{C(l_i)}\prod_{i=-r+1}^{0}\mathbf{p}(l_{i-1},l_{i})=\prod_{i=-r+1}^{0}\mathbf{p}(l_{i},l_{i-1})=\prod_{i =0}^{r-1}\mathbf{p}(l_{-i},l_{-(i+1)}).$$
Substituting in \eqref{eq} yields
$$\lim_{p\to\infty}\mathbb{P}\left((X^p_i)_{-r\leq i\leq r}=(l_i)_{-r\leq i\leq r}\right)=\mathbb{P}(X_0=l_0, X_1=l_1,\ldots,X_r=l_r) \mathbb{P}(X_0=l_0,X_{-1}=l_{-1},\ldots,X_{-r}=l_{-r}),$$
which proves the proposition.
\end{proof}

\subsection{Control on small labels and proof of Theorem~\ref{new construction}}\label{control on small labels}

In order to complete the proof of Theorem~\ref{new construction} and to overcome the obstructions on the continuity of $\Phi$ on $\TB^+$ that we discussed in Section~\ref{infinite construction} we must ensure that, \emph{uniformly in $p$}, the probability that there are small labels far from the root of the treed bridges $\B_p$ is small (Proposition~\ref{small labels}). Once that is done, the proof of Theorem~\ref{new construction} can proceed in a way analogous to~\cite{Men08}. 

Our first estimate bounds the rate of growth of labels along the bridge $(X_i)_{i\in\mathbb{Z}}$ of $\B_\infty$. Since the right and left half of the bridge behave as nearest neighbour random walks with transition probabilities given by $\mathbf{p}$, one can invoke general results for such random walks (see~\cite{CFR09}); in particular, the processes $(X_i)_{i\geq 0}$ and $(X_{-i})_{i\geq 0}$ are transient \cite[Theorem A]{CFR09}; more precisely, by~\cite[Theorem 6.1]{CFR09}, we have:

\begin{prop} \label{prop:transience} Let $ \mathcal{B}_{\infty}=( (X_{i})_{i \in \mathbb{Z}}; T)$ be the random infinite treed bridge defined in Proposition~\ref{limitinp}; then for every $\varepsilon>0$ we have
$$ X_{|i|} \geq |i|^{ \frac{1}{2}- \varepsilon}$$
almost surely, for all $i$ large enough.
\end{prop}

We shall need an analogous result for finite bridges, which reads as follows:

\begin{lemma}\label{small roots} We have
$$ \lim_{k \to \infty} \sup_{p \geq 1} \mathbb{P}_{0}\left( \exists i : k \leq i \leq p,\ X_{i} < i^{1/2 - \varepsilon}  \mid X_{2p}=0\right) =0.$$
\end{lemma}
\proof Introduce the stopping time $\tau_k = \inf\{ i \geq k : X_{i} < i^{1/2- \varepsilon}\}$. We have 
 \begin{eqnarray*} \mathbb{P}_{0}(\exists i : k \leq i \leq p,\ X_{i} < i^{1/2 - \varepsilon} \mid X_{2p}=0)  &=& \mathbb{P}_{0}( \{\tau_k\leq p\} \mid X_{2p}=0)\\
 & = & (\mathbb{P}_{0}(X_{2p}=0))^{-1} \mathbb{P}_{0}( \{\tau_k\leq p\}  \cap \{X_{2p} =0\})\\ 
 & \underset{ \mathrm{Markov}}{=} & (\mathbb{P}_{0}(X_{2p}=0))^{-1} \mathbb{E}_{0}[  \mathbf{1}_{\{\tau_k \leq p\}} \  \mathbb{P}_{X_{\tau_k}}( \tilde{X}_{2p-\tau_k}=0)],  \end{eqnarray*}
where $(\tilde{X}_i)_{i\geq 0}$ is an independent copy of $(X_i)_{i\geq 0}$.
 
 Now notice that by an easy coupling argument we have $ \mathbb{P}_{i}( X_{n}=0) \leq \mathbb{P}_{j}( X_{n}=0)$ if $j \leq i$ have the same parity. Hence we can bound $\mathbb{P}_{X_{\tau_k}}( \tilde{X}_{2p-\tau_k}=0)$ from above by $\mathbb{P}_{0/1}(\tilde{X}_{2p-\tau_k}=0)$, where $0/1$ depends on the parity of ${X}_{\tau_k}$ and $\tau_k$. Using \eqref{eq:0-0} and Lemma~\ref{lem:asympbridge} together with the fact that $2p - \tau_k  \geq p$ we deduce that there exists a constant $C>0$ (which does not depend on $p$) such that $\mathbb{P}_{X_{\tau_k}}(\tilde{X}_{2p-\tau_k}=0) \leq C p^{-5/2}$. Coming back to the last display and using \eqref{eq:0-0} again, we obtain the following expression, where the constant $C$ is not necessarily the same as above:
 $$\mathbb{P}_{0}( \{ \tau_k \leq p\} \mid X_{2p}=0) \leq C \mathbb{P}_{0}(\{ \tau_k \leq p\}).$$
 We have found that for all $p$ we have $\mathbb{P}_{0}\left( \exists i : k \leq i \leq p,\  X_{i} < i^{1/2 - \varepsilon}  \mid X_{2p}=0\right)\leq C\mathbb{P}_{0}(\{ \tau_k \leq p\})\leq C\mathbb{P}_{0}(\{ \tau_k <\infty\})$. Proposition~\ref{prop:transience} then clearly entails that $\mathbb{P}_0( \{ \tau_k <\infty\})\to 0$ as $k\to\infty$, which implies the desired result.
 \endproof
 
We now use the above bounds on the rate of growth of labels along the bridge to prove that labels in the random treed bridges $\B_p$, provided they do not belong to trees grafted near the root, are unlikely to be small (see~\cite[Section 4.3]{Men08} for analogous results in the case of the UIPQ).

 \begin{proposition} \label{small labels}
 Consider the random treed bridges $\B_p=(X^p;T^p)$ as defined in Proposition~\ref{Bp}. For each $i\in\DS(X^p)$, call $L^p_i$ the minimum label appearing in $T^p(i)$. For any $\epsilon>0$ and $m\geq 0$ we can find $k>0$ such that for all $p$, assuming we see $i$ modulo $2p$,\ 
 $$\P\left(\min_{i\in\DS(X^p)\setminus[-k,k]}L^p_i \leq m\right)\leq\epsilon.$$. 
 \end{proposition}
 
 \begin{proof}
 Fix $m \geq 0$ and $ \varepsilon>0$. Throughout this proof we shall write the random bridge $X^p$ as $$(X_0^p, X_1^p,\ldots, X_p^p, X_{-(p-1)}^p, X_{-(p-2)}^p,\ldots X_{-1}^p)$$ and see $\DS(X^p)$ as a subset of $\{-(p-1), -(p-2),\ldots, 0, 1,\ldots p\}$. This way Lemma~\ref{small roots} (together with a symmetry argument) ensures that
 $$\lim_{k\to\infty}\sup_{p\geq 1}\P\left(\exists{i\in\DS(X^p)\setminus[-k,k]} : X^p_i\leq|i|^{1/2-\epsilon}\right)= 0.$$
 That is, we can assume $k$ is such that $\P\left(\exists {i\in\DS(X^p)\setminus[-k,k]} : X^p_i\leq |i|^{1/2-\epsilon}\right)\leq \epsilon/2$ for all $p$.
 
 Therefore we can write
  $$\P\left(\min_{i\in\DS(X^p)\setminus[-k,k]}L^p_i \leq m\right)\leq\epsilon/2+\P\left(\min_{i\in\DS(X^p)\setminus[-k,k]}L^p_i \leq m \bigm\vert \forall {i\in\DS(X^p)\setminus[-k,k]}, X^p_i>|i|^{1/2-\epsilon}\right).$$
  
Since the trees grafted on the bridge are conditionally independent given $(X_i^p)_{i\geq 0}$, the second term on the right hand side of the inequality is at most
$$\sum_{i\in\DS(X^p)\setminus[-k,k]}\P\left(L_i^p\leq m \bigm \vert X^p_i>|i|^{1/2-\epsilon}\right).$$

Let now $\LT_x^{>m}\subset \LT_x^+$ be the set of all positive labelled plane trees such that their root is labelled $x$ and the minimum label appearing in the tree is strictly greater than  $m$. For each index $i$ in $\DS(X^p)\setminus[-k,k]$, we have 
$$\P\left(L_i^p\leq m \bigm \vert X^p_i>|i|^{1/2-\epsilon}\right)\leq \sup_{x>|i|^{1/2-\epsilon}}\P\left(L_i^p\leq m \bigm \vert X^p_i=x\right)= \sup_{x>|i|^{1/2-\epsilon}}1-\rho^+_{x}(\LT^{>m}_{x}).$$ Since $m$ is fixed, we may assume that $k>m^{1/(1/2-\epsilon)}$; we know that $|i|>k$, so we may assume that the supremum is taken over values of $x$ greater than $m$. If $x>m$, it's quite clear that there is a bijection between $\LT_x^{>m}$ and $\LT^+_{x-m}$ (given by simply subtracting $m$ from all labels), hence $\rho_x^+(\LT_x^{>m})=\frac{w_{x-m}}{w_x}$. This shows that the probability $\P\left(L_i^p\leq m \bigm \vert X^p_i=x\right) $ is $$1-\frac{w_{x-m}}{w_x}=1-\frac{(x-m)(x-m+3)(x+1)(x+2)}{x(x+3)(x-m+1)(x-m+2)}$$ which is 
less than $Cx^{-3}$ for some constant $C$ only depending on $m$.

As a consequence,
$$ \sup_{x>|i|^{1/2-\epsilon}}\P\left(L_i^p\leq m \bigm \vert X^p_i=x\right)\leq \P\left(L_i^p\leq m \bigm \vert X^p_i=\lceil |i|^{1/2-\epsilon}\rceil \right)<C|i|^{-3/2-3\epsilon}$$
for some constant $C$ (only depending on $m$). Now, summing over $i$ in $\DS(X^p)\setminus[-k,k]$, one gets
$$\sum_{i\in\DS(X^p)\setminus[-k,k]}\P\left(L_i^p\leq m \bigm \vert X^p_i>|i|^{-1/2-\epsilon}\right)<2\sum_{i=k}^pCi^{-3/2-3\epsilon}<2C\sum_{i=k}^\infty i^{-3/2-3\epsilon}.$$
Since the last sum is infinitesimal for $k\to\infty$ and does not depend on $p$, we can choose $k$ such that the original expression
$$\P\left(\min_{i\in\DS(X^p)\setminus[-k,k]}L^p_i \leq m \bigm\vert \forall {i\in\DS(X^p)\setminus[-k,k]}\;\; X^p_i>|i|^{1/2-\epsilon}\right)$$
is at most $\epsilon/2$ for all $p$, and thus establish the proposition.
 \end{proof}
 
 Finally, we shall give the last ingredients needed in order to incorporate our estimates into a proof of Theorem~\ref{new construction}. The following lemma (analogous to~\cite[Proposition 4]{Men08}) gives the property of $\Phi$ that will act as a surrogate for continuity; immediately after, we give a corollary of Proposition~\ref{small labels} that relates Lemma~\ref{kinda continuity} to the previous estimates for the growth of labels.
 
 \begin{lem}\label{kinda continuity}Let $\TB(k,r)$ the set of all (finite or infinite) positive treed bridges $B$ such that the minimum label appearing on vertices in $B\setminus [B]_k$ is at least $r+2$. Then for any pair of treed bridges $B$, $B'$ in $\TB(k,r)$, the equality $[B]_k=[B']_k$ implies $[\Phi(B)]_r=[\Phi(B')]_r$.
 \end{lem}
 
 \begin{proof}
 Suppose $B$ is a treed bridge in $\TB(k,r)$; the ball $[\Phi(B)]_r$ is the submap of $\Phi(B)$ spanned by vertices of $\Phi(B)$ having label at most $r$ in $B$ (its edges are those that the construction $\Phi$ draws between such vertices). Notice that, if $(c_i)_{i\in{I}}$ is the sequence of real corners of $B$, ordered according to the left-to-right contour of the upper face, then $[\Phi(B)]_r$ is determined by the subsequence $(c_{i_j})_{j=1}^N$ of corners bearing labels not exceeding $r$. Such a subsequence is in turn determined by $[B]_k$ if $B$ is in $\TB(k,r)$, since all real vertices bearing label $k$ or smaller and all edges of the treed bridge involving such vertices (whose endpoints, by definition of a treed bridge, have to bear labels no greater than $k+1$) are not erased in the construction of $[B]_k$; hence corners labelled $k$ or less (along with their ordering) can be recovered from $[B]_k$ alone. As a consequence, $[B]_k=[B']_k$ implies $[\Phi(B)]_r=[\Phi(B')]_r$ for any pair of treed bridges $B,B'$ in $\TB(k,r)$.
 \end{proof}

 \begin{cor}\label{likely kinda continuity}Given $r\geq 0$ and $\epsilon>0$, there is an integer $k>0$ such that, for each $p$ in $\mathbb{Z}^+\cup \{\infty\}$, we have $\P(\B_p\notin\TB(k,r))\leq \epsilon$. \end{cor}
 
 The proof of this corollary is straightforward from Proposition~\ref{small labels} and left to the reader: we refer to~\cite{Men08} and~\cite{thesis} for further details. We now proceed to show Theorem~\ref{new construction}.

\begin{proof}[Proof of Theorem~\ref{new construction}]
We shall show that, for each $r>0$ and for each $q$, where $q=[Q]_r$ for some rooted quadrangulation $Q$ with an infinite boundary,
$$\lim_{p\to\infty}\P\left([\Phi(\B_p)]_r=q\right)=\P\left([\Phi(\B_\infty)]_r=q\right);$$
this will prove that the law of the ball $[\Phi(\B_\infty)]_r$ is the limit in $p$ of the laws of $[\Phi(\B_p)]_r$ for every $r\geq 0$, hence that $\Phi(\B_\infty)$ is the local limit of the random quadrangulations $\Phi(\B_p)$, that is to say the Boltzmann quadrangulations $\BQ_p$, for $p\to\infty$. By Proposition~\ref{boltzmann convergence}, this implies that $\Phi(\B_\infty)$ is indeed distributed as the UIHPQ $\UIHPQ$.

Fix $\epsilon>0$; by Corollary~\ref{likely kinda continuity}, we can find $k$ such that for all $p$ we have $\P(\B_p\notin\TB(k,r))\leq\epsilon$, and also $\P(\B_\infty\notin\TB(k,r))\leq\epsilon$. Consider the set $\mathsf{A}$ of bridges $B'$ in $\TB(k,r)$ such that $[\Phi(B')]_r=q$, and define the set $\mathsf{A}_k$ to be the set of maps $[B']_k$ for $B'$ in $\mathsf{A}$; by Lemma~\ref{kinda continuity} a bridge $B''$ in $\TB(k,r)$ is such that $[\Phi(B'')]_r=q$ if and only if $[B'']_k\in \mathsf{A}_k$.
Now consider
$$\left|\lim_{p\to\infty}\P\left([\Phi(\B_p)]_r=q\right)-\P\left([\Phi(\B_\infty)]_r=q\right)\right|;$$
by taking intersections with the events $\B_p\in \TB(k,r)$ and $\B_\infty\in \TB(k,r)$, whose complements have probability at most $\epsilon$, the above can be bounded by
$$2\epsilon+\left|\lim_{p\to\infty}\P\left([\Phi(\B_p)]_r=q\mbox{ and } \B_p\in \LT(k,r)\right)
-\P\left([\Phi(\B_\infty)]_r=q\mbox{ and } \B_\infty\in \LT(k,r)\right)\right|$$
that is to say
$$2\epsilon+\left|\lim_{p\to\infty}\P\left([\B_p]_k\in \mathsf{A}_k\mbox{ and } \B_p\in \LT(k,r)\right)
-\P\left([\B_\infty]_k\in \mathsf{A}_k \mbox{ and } \B_\infty\in \LT(k,r)\right)\right|.$$
Up to losing another $2\epsilon$, one can now add terms of the kind $\P\left([\B_p]_k\in \mathsf{A}_k\mbox{ and } \B_p\notin \LT(k,r)\right)$ and $\P\left([\B_\infty]_k\in \mathsf{A}_k\mbox{ and } \B_p\notin \LT(k,r)\right)$ within the absolute value, thus finally bounding the quantity in question by
$$4\epsilon+\left|\lim_{p\to\infty}\P\left([\B_p]_k\in \mathsf{A}_k\right)-\P\left([\B_\infty]_k\in \mathsf{A}_k \right)\right|,$$
which is $4\epsilon$ by Proposition~\ref{limitinp} (since the law of $[\B_\infty]_k$ is the limit of the laws of $[\B_p]_k$ when $p$ is sent to infinity). Since $\epsilon$ is arbitrary, we have shown the initial desired equality.
\end{proof}

\section{A study of geodesic rays in $\UIHPQ$}\label{study of geodesic rays}

Recall that a geodesic $\gamma$ in a planar map is a path, finite or infinite, that visits a sequence of vertices $(\gamma(i))_{i\in I}$, where $I$ may be $\mathbb{Z}$, $\mathbb{N}$ or the set $\{0,1,\ldots,n\}$ (in which case we say the geodesic has length~$n$), such that for each $i,j\in I$ we have $\mathrm{d_{gr}}(\gamma(i),\gamma(j))=|i-j|$. Notice that, since the maps we consider are not necessarily simple and since $\gamma$ is formally seen as a sequence of concatenated edges, it is \emph{not} determined by the sequence $(\gamma(i))_{i\in I}$ of the vertices it visits.

Given a vertex $x_0$ in an infinite map we call a geodesic $\gamma$ a \emph{geodesic ray issued from $x_0$} if it is one-ended and starts with $x_0$ (that is, if it visits a sequence of vertices $(\gamma(i))_{i\in\mathbb{N}}$, with $\gamma(0)=x_0$). In a rooted infinite map, we simply say ``geodesic ray'' for geodesic rays issued from the root vertex, as all geodesic rays in Sections~\ref{decomposition} and~\ref{leftmost} will be.

We shall investigate the (random) set of geodesic rays in the UIHPQ $\UIHPQ$; in order to do this, given a quadrangulation with a boundary $q\in\sQ$ that is the image via $\Phi$ of a positive treed bridge $(b;T)$ in $\TB^+$, we reconstruct geodesics issued from the root vertex in $q$ as coded by $(b;T)$. A path in $q$ can be expressed in $(b;T)$ as a sequence of corners, since real corners of $(b;T)$ correspond to edges of $q=\Phi((b;T))$. Given a real corner $c$ in $(b;T)$, we shall write $s(c)$ for its successor (as defined in Section~\ref{infinite construction}) and $v(c)$ for its corresponding vertex (which, as remarked in Section \ref{pconstruction}, can be seen both as a vertex of $(b;T)$ and as a vertex of $q$). Thanks to the properties that $\Phi$ displays when applied to positive treed bridges, a sequence $(c_i)_{i\in I}$ of real corners in $(b;T)$ encodes a geodesic $\gamma$ issued from the root vertex in $q$ provided that:

\begin{itemize}
\item $I$ is either the set $\{1,\ldots, n\}$ (for a geodesic of length $n$) or $\mathbb{N}^+=\{1,2,3,\ldots\}$ (for a geodesic ray);
\item for each $i\in I$, we have $l(c_i)=i$ ;
\item for each $i\in I\setminus\{1\}$, we have $v(s(c_i))=v(c_{i-1})$.
\end{itemize}

The geodesic $\gamma$ is the path obtained by concatenating the edges drawn by $\Phi$ between each corner $c_i$ and its successor $s(c_i)$; in particular, the first edge of $\gamma$ joins corner $c_1$ to the corner around the added vertex $\delta$ (that is the root vertex in $\Phi((b;T))$). 

We shall make this identification implicitly in all that follows: we suppose that the UIHPQ is constructed as $\UIHPQ=\Phi(\B_\infty)$, thus reducing the problem of studying geodesic rays in the UIHPQ to that of investigating geodesic rays in the infinite random treed bridge $\B_\infty$.

\subsection{The pencil decomposition}\label{decomposition}
\begin{figure}\centering
\begin{tikzpicture}
\tikzset{every path/.style={thick}}
\tikzstyle{real}=[inner sep=1.5pt, fill=black, circle, draw=white, thick]
\tikzset{every node/.style={font=\scriptsize}}
\useasboundingbox (-6,-3.5) rectangle (6,1.5);
\contourlength{1pt}
\contournumber{100}
	\begin{pgfonlayer}{nodelayer}
		\node [style=real] (0) at (0, -2) {};
		\node [style=real] (1) at (1.5, -2) {};
		\node [style=real] (2) at (2.5, -2) {};
		\node [style=real] (3) at (3.5, -2) {};
		\node [style=real] (4) at (4, -2) {};
		\node [style=real] (5) at (4, 0) {};
		\node [style=real] (6) at (5, -2) {};
		\node [style=real] (7) at (-1.25, -2) {};
		\node [style=real] (8) at (-2, -2) {};
		\node [style=real] (9) at (-3, -2) {};
		\node [style=real] (10) at (-4, -2) {};
		\node [style=real] (11) at (-4.75, -2) {};
		\node [style=real] (12) at (0, -2.5) {};
		\node [style=real] (13) at (0, -3) {};
		\node [style=real] (14) at (-0.75, -2.5) {};
		\node [style=real] (15) at (-0.75, -3.25) {};
		\node [style=real] (16) at (1.25, -2.75) {};
		\node [style=real] (17) at (1.75, -2.75) {};
		\node [style=real] (18) at (1, -0.75) {};
		\node [style=real] (19) at (1.25, -1.25) {};
		\node [style=real] (20) at (0.25, -0.75) {};
		\node [style=real] (21) at (-0.5, -1.25) {};
		\node [style=real] (22) at (3, -1.75) {};
		\node [style=real] (23) at (1.75, -1.5) {};
		\node [style=none] (24) at (0, -1.65) {\contour{white}0};
		\node [style=none] (25) at (1, -0.4) {\contour{white}1};
		\node [style=none] (26) at (1.2, -2.3) {\contour{white}1};
		\node [style=none] (27) at (-0.5, -0.95) {\contour{white}1};
		\node [style=none] (28) at (1.2, -1.6) {\contour{white}2};
		\node [style=none] (29) at (2.5, -2.3) {\contour{white}2};
		\node [style=real] (30) at (2, 0.25) {};
		\node [style=real] (31) at (3, 0) {};
		\node [style=real] (32) at (2.5, 0.25) {};
		\node [style=real] (33) at (4.25, -1.25) {};
		\node [style=none] (34) at (4, -2.3) {\contour{white}2};
		\node [style=none] (35) at (3.9, -0.3) {\contour{white}2};
		\node [style=real] (36) at (-2.5, -1.75) {};
		\node [style=real] (37) at (-3.25, -0.5) {};
		\node [style=real] (38) at (-1.25, 0) {};
		\node [style=real] (39) at (-0.25, 0) {};
		\node [style=real] (40) at (-4, -2.75) {};
		\node [style=real] (41) at (-3, -2.75) {};
		\node [style=real] (42) at (-3, -3.5) {};
		\node [style=real] (43) at (-3, -3) {};
		\node [style=real] (44) at (-2.5, -2.5) {};
		\node [style=real] (45) at (-2.25, 0.25) {};
		\node [style=none] (46) at (-1.25, -2.3) {\contour{white}1};
		\node [style=none] (47) at (-0.9, -2.5) {\contour{white}1};
		\node [style=none] (48) at (0.15, -2.5) {\contour{white}1};
		\node [style=none] (49) at (0, -3.3) {\contour{white}1};
		\node [style=none] (50) at (0.25, -1.05) {\contour{white}2};
		\node [style=none] (51) at (-0.75, -3.6) {\contour{white}2};
		\node [style=none] (52) at (-2, -2.3) {\contour{white}2};
		\node [style=none] (53) at (1.1, -3) {\contour{white}2};
		\node [style=none] (54) at (1.9, -3) {\contour{white}2};
		\node [style=none] (55) at (2, -0.1) {\contour{white}2};
		\node [style=none] (56) at (-4.15, -2.3) {\contour{white}2};
		\node [style=none] (66) at (-1.25, 0.35) {\contour{white}2};
		\node [style=none] (66) at (1.5, -1.4) {\contour{white}2};
		\node [style=none] (57) at (-4.15, -3) {\contour{white}3};
		\node [style=none] (58) at (-3.2, -0.8) {\contour{white}3};
		\node [style=none] (59) at (-3.25, -2.3) {\contour{white}3};
		\node [style=none] (60) at (-2.75, -1.75) {\contour{white}3};
		\node [style=none] (61) at (-2.5, 0.25) {\contour{white}3};
		\node [style=none] (62) at (-0.30, -0.25) {\contour{white}3};
		\node [style=none] (63) at (3.5, -2.3) {\contour{white}3};
		\node [style=none] (64) at (5, -2.3) {\contour{white}3};
		\node [style=none] (65) at (2.9, -0.25) {\contour{white}3};
		\node [style=none] (65) at (2.8, -1.8) {\contour{white}3};
		\node [style=none] (67) at (-3.25, -2.75) {\contour{white}4};
		\node [style=none] (68) at (-2.35, -2.75) {\contour{white}4};
		\node [style=none] (69) at (4, -1.25) {\contour{white}4};
		\node [style=none] (70) at (2.45, 0.5) {\contour{white}4};
		\node [style=none] (71) at (-3, -3.25) {\contour{white}5};
		\node [style=none] (71) at (-3, -3.75) {\contour{white}5};
		\node [style=real] (72) at (-4.5, -1.25) {};
		\node [style=none] (73) at (-4.25, -1.25) {\contour{white}3};
		\contourlength{2px}
		\node [style=none] (74) at (-5.25, -0.75) {\small\contour{white}{{\color{blue}$\gamma^{\mathrm{left}}$}}};
		\node [style=none] (75) at (-4.75, -2.3) {3};
		\node [style=none, fill=white] (76) at (-5.75, -2) {\small{$\ldots$}};
		\node [style=none] (77) at (5.75, -1.25) {\small\contour{white}{\color{green!50!black}$\gamma^{\mathrm{right}}$}};
		\node [style=none, fill=white] (78) at (5.75, -2) {\small{$\ldots$}};
		\node [style=none] (79) at (-4, -0.5) {};
		\node [style=none] (80) at (-1, 0.75) {};
		\node [style=none] (81) at (-0.25, 0.5) {};
		\node [style=none] (82) at (4.5, 0.25) {};
		\node [style=real] (83) at (0.5, 0.25) {};
		\contourlength{1px}
		\node [style=none] (84) at (0.7, 0.4) {\contour{white}2};
	\end{pgfonlayer}
	\begin{pgfonlayer}{edgelayer}
	\fill[orange!10] (74.center) to (72.center) to (10.center) [in=135, out=45] to (21.center) [bend right=0] to (0.center) to (18.center) [bend left=45] to (4.center) [bend left=0] to (6.center) to (77.center) [in=90] to (82.center) [out=-60, in=180] to (80.center) [out=-80, in=-40] to (79.center) [in=-20, out=80] to (74.center);
		\fill[blue!10] (0.center) to (21.center) [out=135, in=45] to (10.center) to [bend left=0] (0.center);
		\fill[blue!10, bend left=60, looseness=1.4] (41.center) to (42.center) [bend left=60, looseness=1.4] to (41.center);
		\fill[blue!10] (76.center) [bend left=20] to (74.center) [bend right=0] to (72.center) to (10.center) to (11.center) to (76.center);
		\fill [green!50!black!10] (0.center) to (18.center) [bend left=45] to (4.center) [bend left=0] to (0.center);
		\fill [green!50!black!10, bend right=45, looseness=1.3] (0.center) to (13.center) [bend right=45, looseness=1.3] to (0.center);
		\fill [green!50!black!10] (6.center) to (77.center) to (78.center) to (6.center);
		
		\draw (0) to (14);
		\draw [style=root, bend right=45, looseness=1.25] (0) to (13);
		\draw (12) to (0);
		\draw [bend right=45, looseness=1.25] (13) to (0);
		\draw (7) to (0);
		\draw (0) to (1);
		\draw (8) to (7);
		\draw [bend left=90, looseness=2.00] (9) to (8);
		\draw (10) to (9);
		\draw (15) to (14);
		\draw (16) to (1);
		\draw (1) to (17);
		\draw (1) to (2);
		\draw [bend left=75, looseness=2.00] (2) to (3);
		\draw [style=georight] (18) to (0);
		\draw (20) to (18);
		\draw [style=geoleft] (0) to (21);
		\draw (21) to (20);
		\draw (18) to (19);
		\draw [bend left=90, looseness=2.50] (19) to (1);
		\draw [in=90, out=24, looseness=1.25] (18) to (2);
		\draw (22) to (3);
		\draw (2) to (3);
		\draw (23) to (1);
		\draw (19) to (1);
		\draw (3) to (4);
		\draw (32) to (31);
		\draw (18) to (30);
		\draw (31) to (4);
		\draw [bend right=15] (30) to (31);
		\draw [bend right=60, looseness=1.25] (31) to (30);
		\draw (33) to (31);
		\draw (33) to (6);
		\draw (5) to (6);
		\draw [bend left=75, looseness=1.50] (18) to (5);
		\draw (21) to (8);
		\draw [style=geoleft, in=135, out=45] (10) to (21);
		\draw (36) to (8);
		\draw (9) to (8);
		\draw (10) to (37);
		\draw (38) to (21);
		\draw [bend left=75, looseness=1.50] (37) to (38);
		\draw (38) to (39);
		\draw (39) to (20);
		\draw (10) to (11);
		\draw (10) to (40);
		\draw [bend right=60, looseness=1.25] (42) to (41);
		\draw (9) to (41);
		\draw [bend left=60, looseness=1.25] (42) to (41);
		\draw (43) to (41);
		\draw (44) to (9);
		\draw (45) to (38);
		\draw (37) to (38);
		\draw (31) to (5);
		\draw (11) to (76.center);
		\draw (6) to (78.center);
		\draw (37) to (79.center);
		\draw (38) to (80.center);
		\draw (5) to (82.center);
		\draw (81.center) to (39);
		\draw (39) to (83);
		\draw (83) to (18);
		\draw [style=georight] (6) to (4);
		\draw [style=geoleft] (72) to (10);
		\draw [style=geoleft] (72) to (74.center);
		\draw [style=georight] (6) to (77.center);
		\draw [style=georight, bend left=45] (18) to (4);
	\end{pgfonlayer}
\end{tikzpicture}
\caption{\label{geodesic rays} \small A neighbourhood of the root vertex in a quadrangulation with an infinite boundary, with its leftmost and rightmost geodesic rays. Any geodesic ray is contained in the orange zone between them. Notice that in this case the root vertex belongs to the infinite ``core''; for a more general situation, see Figure~\ref{intro}.}
\end{figure}
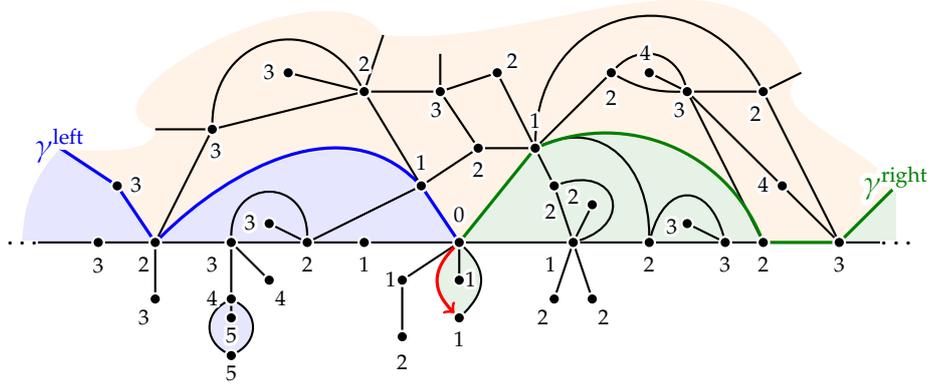

Given two geodesic rays $\gamma, \gamma'$ in a rooted (infinite) quadrangulation with a boundary $q\in\sQ$ we write $\gamma\preceq\gamma'$ if $\gamma$ lies \emph{to the left} of $\gamma'$; that is, if $\gamma$ lies within the quadrangulation $q|{\gamma'}$ obtained from $q$ by erasing all vertices that lie strictly to the right of $\gamma'$ and all edges involving such vertices; equivalently, $\gamma\preceq\gamma'$ if $\gamma'$ lies within $\gamma|q$, the quadrangulation obtained by erasing all vertices and edges of $q$ lying strictly to the left of $\gamma$. Notice it is easy to prove existence and uniqueness of maximal and minimal elements for the partial order $\succ$: we call such elements the \emph{rightmost} and \emph{leftmost} geodesic rays, denoted by $\gamma^{\mathrm{right}}$ and $\gamma^{\mathrm{left}}$ (where dependence on $q$ is implicit). Given any treed bridge $(b;T)\in\TB^+_\infty$, we shall speak of $\gamma^{\mathrm{right}}$ and $\gamma^{\mathrm{left}}$ in $(b;T)$ (as the sequences of corners corresponding to the rightmost and leftmost geodesic rays of $\Phi((b;T))$).

In this section we shall investigate the three random quadrangulations with a boundary one obtains by ``cutting up'' the UIHPQ along its leftmost and rightmost geodesic rays.

To start with, we shall give a description of the rightmost geodesic ray $\gamma^{\mathrm{right}}$ in an infinite positive treed bridge from $\TB^+_\infty$.

\begin{lem}The rightmost geodesic ray $\gamma^{\mathrm{right}}$ in a treed bridge $(b;T)\in\TB^+_\infty$ is the sequence of corners $(c^r_i)_{i\in\mathbb{N}^+}$ such that for each $i>0$  corner $c^r_i$ is the \emph{leftmost} real corner labelled $i$ to be found in $(b;T)$ (which is well defined since each label appears a finite number of times in $(b;T)$).\end{lem}
\begin{proof}
The sequence of corners $(c^r_i)_{i\in\mathbb{N}^+}$ is indeed a geodesic ray, since for $i>0$ we have $l(c^r_i)=i$ by definition, and $c^r_i$ is the successor of $c^r_{i+1}$ (hence $v(s(c^r_{i+1}))=v(c^r_i)$).

Now consider any geodesic ray $\gamma=(c_i)_{i\in \mathbb{N}^+}$ in $(b;T)$; we shall see that all corners of $\gamma$ belong to trees lying left of the root vertex.

Notice that, for any real corner $c$ of $(b;T)$ such that $v(c)$ belongs to a tree $T(j)$ with $j<0$, $v(s(c))$ belongs to a tree $T(j')$, again with $j'<0$. Now suppose by contradiction that some corner $c_k$ in $\gamma=(c_i)_{i\in\mathbb{N}^+}$ belongs to a tree $T(n_k)$, with $n_k>0$; since we know that $v(s(c_{k+1}))=v(c_k)$, the vertex $v(c_{k+1})$ must belong to a tree $T(n_{k+1})$ with $0<n_{k+1}\leq n_k$; inductively, the same result applies to all corners $c_i$ with $i>k$. This, however, is in contradiction with the fact that there is only a finite number of corners whose vertices belong to $\cup_{i\in\DS(b)\cap [0,n_k]}T(i)$.

Consider now any edge $(c, s(c))$ of $\gamma$: $c$ belongs to a tree $T(i)$ with $i<0$, hence there must be two corners $c^r_a$ and $c^r_{a-1}$ in $(c^r_i)_{i\in\mathbb{N}^+}$, such that $c$ is between $c^r_a$ and $c^r_{a-1}$ in the (left-to-right) contour (and $c$ is not $c^r_{a-1}$); since $s(c^r_a)=c^r_{a-1}$ as remarked, $c$ must have label at least $a$, and $s(c)$ must be a corner between $c$ and $c^r_{a-1}$; as a consequence, the edge $(c, s(c))$ of $\gamma$ is drawn below the edge $(c^r_a,c^r_{a-1})$, and the geodesic ray $\gamma$ is contained in the portion of $\Phi((b;T))$ lying left of the geodesic $(c^r_i)_{i\in\mathbb{N^+}}$, which is therefore the rightmost geodesic ray.
\end{proof}

\begin{figure}
\centering
\begin{tikzpicture}[scale=.9]
\tikzset{every path/.style={thick}}
\tikzset{every node/.append style = {font = \footnotesize}}
\tikzstyle{real}=[inner sep=1.5pt, fill=black, circle, draw=white, thick]
\useasboundingbox (-7.5,-4) rectangle (6,3);
\node (qleft) at (-4, -0.5) {\color{blue}$q|\gamma^{\mathrm{left}}$};
\node (qright) at (4.5, -0.5) {\color{green!50!black}$\gamma^{\mathrm{right}}|q$};
\node (mid) at (-0.5, 3.5) {\color{orange}$\gamma^{\mathrm{left}}|q|\gamma^{\mathrm{right}}$};
	\begin{pgfonlayer}{nodelayer}
	\node [style=real] (0) at (0, -2) {};

		\node [style=real] (0) at (0, -2) {};
		\node [style=real] (1) at (1.5, -2) {};
		\node [style=real] (2) at (2.5, -2) {};
		\node [style=real] (3) at (3.5, -2) {};
		\node [style=real] (4) at (4, -2) {};
		\node [style=real] (5) at (5, -2) {};
		\node [style=real] (6) at (0, -2.5) {};
		\node [style=real] (7) at (0, -3) {};
		\node [style=real] (8) at (1.25, -2.75) {};
		\node [style=real] (9) at (1.75, -2.75) {};
		\node [style=real] (10) at (1, -0.75) {};
		\node [style=real] (11) at (1.25, -1.25) {};
		\node [style=real] (12) at (3, -1.75) {};
		\node [style=real] (13) at (1.75, -1.5) {};
		\node [style=none, fill=white] (25) at (5.75, -1.25) {$\ldots$};
		\node [style=none, fill=white] (26) at (5.75, -2) {$\ldots$};
		\node [style=real] (27) at (-0.25, 2.5) {};
		\node [style=real] (33) at (3.5, 1) {};
		\node [style=real] (34) at (1.75, 2.5) {};
		\node [style=real] (35) at (-1, 2.25) {};
		\node [style=real] (36) at (-2, 2.25) {};
		\node [style=real] (37) at (-0.5, 1.5) {};
		\node [style=real] (39) at (0.25, 1.5) {};
		\node [style=real] (40) at (-3, 2.5) {};
		\node [style=real] (41) at (1.25, 2.5) {};
		\node [style=real] (43) at (3.25, 2.25) {};
		\node [style=none, fill=white] (45) at (-6, 1.5) {$\ldots$};
		\node [style=none] (49) at (-1, 2.75) {};
		\node [style=real] (51) at (3.25, 0.25) {};
		\node [style=none] (55) at (3.75, 2.5) {};
		\node [style=real] (57) at (4.25, 0.25) {};
		\node [style=real] (58) at (-0.75, 0.25) {};
		\node [style=real] (59) at (-4, 1.75) {};
		\node [style=none] (60) at (-4.75, 1.75) {};
		\node [style=real] (62) at (-5.25, 1) {};
		\node [style=none] (64) at (-1.75, 2.75) {};
		\node [style=none, fill=white] (65) at (5, 1) {$\ldots$};
		\node [style=real] (66) at (2.25, 2.25) {};
		\node [style=real] (67) at (-4.75, 0.25) {};
		\node [style=real] (68) at (-1.25, 1) {};
		\node [style=real] (70) at (-4, -2.5) {};
		\node [style=none, fill=white] (76) at (-7.25, -2) {$\ldots$};
		\node [style=real] (77) at (-4.5, -3) {};
		\node [style=real] (78) at (-2.25, -3.25) {};
		\node [style=real] (79) at (-4.5, -2.75) {};
		\node [style=real] (81) at (-2.75, -2) {};
		\node [style=real] (82) at (-6.25, -2) {};
		\node [style=real] (83) at (-3.5, -2) {};
		\node [style=real] (84) at (-4, -1.75) {};
		\node [style=none, fill=white] (85) at (-6.75, -0.75) {$\ldots$};
		\node [style=real] (87) at (-4.5, -2) {};
		\node [style=real] (88) at (-2.25, -2.5) {};
		\node [style=real] (91) at (-1.5, -2) {};
		\node [style=real] (94) at (-4.5, -3.5) {};
		\node [style=real] (97) at (-6, -1.25) {};
		\node [style=real] (99) at (-5.5, -2.75) {};
		\node [style=real] (100) at (-5.5, -2) {};
		\node [style=real] (101) at (-2, -1.25) {};
		\node [style=none] (102) at (0, -2) {};
		\node [style=none] (103) at (1, -0.75) {};
		\node [style=none] (104) at (4, -2) {};
	\end{pgfonlayer}
	\begin{pgfonlayer}{edgelayer}
	\fill[blue!2, rounded corners=5] (-8,-3.5) rectangle (-1,-0.5);
	\fill[green!50!black!2, rounded corners=5] (-0.5,-3.5) rectangle (6.5,-0.5);
	\fill[orange!2, rounded corners=5] (-7,0) rectangle (5.5,3.5);
	
		\fill [green!50!black!10] (102.center) to (103.center) [bend left=45] to (104.center) [bend left=0] to (102.center);
		\fill [green!50!black!10, bend right=45, looseness=1.25] (0.center) to (7.center) [bend right=45, looseness=1.25] to (0.center);
		\fill [green!50!black!10] (5.center) to (25.center) to (26.center) to (5.center);
		\fill (5) to (26.center);
		\draw [style=root, bend right=45, looseness=1.25] (0) to (7);
		\draw (6) to (0);
		\draw [bend right=45, looseness=1.25] (7) to (0);
		\draw (0) to (1);
		\draw (8) to (1);
		\draw (1) to (9);
		\draw (1) to (2);
		\draw [bend left=75, looseness=2.00] (2) to (3);
		\draw [style=georight] (10) to (0);
		\draw (10) to (11);
		\draw [bend left=90, looseness=2.50] (11) to (1);
		\draw [in=90, out=24, looseness=1.25] (10) to (2);
		\draw (12) to (3);
		\draw (2) to (3);
		\draw (13) to (1);
		\draw (11) to (1);
		\draw [style=georight, bend left=45] (10) to (4);
		\draw (3) to (4);
		\draw [style=georight] (5) to (4);
		\draw [style=georight] (5) to (25.center);
		\draw (5) to (26.center);
		
		\fill[orange!10] (45.center) to (62.center) to (67.center) [in=135, out=45] to (68.center) [bend right=0] to (58.center) to (39.center) [bend left=45] to (51.center) [bend left=0] to (57.center) to (65.center) [in=90] to (55.center) [out=-60, in=180] to (64.center) [out=-80, in=-40] to (60.center) [in=-20, out=80] to (45.center);
		\draw (37) to (39);
		\draw [style=geoleft] (58) to (68);
		\draw (68) to (37);
		\draw (34) to (66);
		\draw (39) to (41);
		\draw (66) to (51);
		\draw [bend right=15] (41) to (66);
		\draw [bend right=60, looseness=1.25] (66) to (41);
		\draw (33) to (66);
		\draw (33) to (57);
		\draw [style=georight] (57) to (51);
		\draw (43) to (57);
		\draw [bend left=75, looseness=1.50] (39) to (43);
		\draw [style=geoleft, in=135, out=45] (67) to (68);
		\draw (67) to (59);
		\draw (36) to (68);
		\draw [bend left=75, looseness=1.50] (59) to (36);
		\draw (36) to (35);
		\draw (35) to (37);
		\draw (40) to (36);
		\draw (59) to (36);
		\draw (66) to (43);
		\draw [style=geoleft] (62) to (67);
		\draw [style=geoleft] (62) to (45.center);
		\draw [style=georight] (57) to (65.center);
		\draw (59) to (60.center);
		\draw (36) to (64.center);
		\draw (43) to (55.center);
		\draw (49.center) to (35);
		\draw (35) to (27);
		\draw (27) to (39);
		\draw [style=georight, bend left=45] (39) to (51);
		
		\fill[blue!10] (91.center) to (101.center) [out=135, in=45] to (100.center) to [bend left=0] (91.center);
		\fill[blue!10, bend left=65, looseness=1.35] (94.center) to (79.center) [bend left=65, looseness=1.35] to (94.center);
		\fill[blue!10] (85.center) [bend right=20] to (76.center) [bend right=0] to (82.center) to (100.center) to (97.center) to (85.center);
		\draw (91) to (88);
		\draw (81) to (91);
		\draw (83) to (81);
		\draw [bend left=90, looseness=2.00] (87) to (83);
		\draw (100) to (87);
		\draw (78) to (88);
		\draw [style=root] (91) to (101);
		\draw (101) to (83);
		\draw [style=geoleft, in=135, out=45] (100) to (101);
		\draw (84) to (83);
		\draw (87) to (83);
		\draw (100) to (82);
		\draw (100) to (99);
		\draw [bend right=60, looseness=1.25] (94) to (79);
		\draw (87) to (79);
		\draw [bend left=60, looseness=1.25] (94) to (79);
		\draw (77) to (79);
		\draw (70) to (87);
		\draw [style=geoleft] (97) to (100);
		\draw [style=geoleft] (97) to (85.center);
		\draw (82) to (76.center);
		\draw [style=root] (58) to (39);
		
	\end{pgfonlayer}
\end{tikzpicture}
\caption{\label{pencil decomposition}\small The three quadrangulations in the pencil decomposition of $q$; notice that the root of $\gamma^{\mathrm{right}}|q$ is the root of $q$, while $q|\gamma^{\mathrm{left}}$ and $\gamma^{\mathrm{left}}|q|\gamma^{\mathrm{right}}$ are rooted in the first edge of the leftmost and rightmost geodesic rays respectively.}
\end{figure}
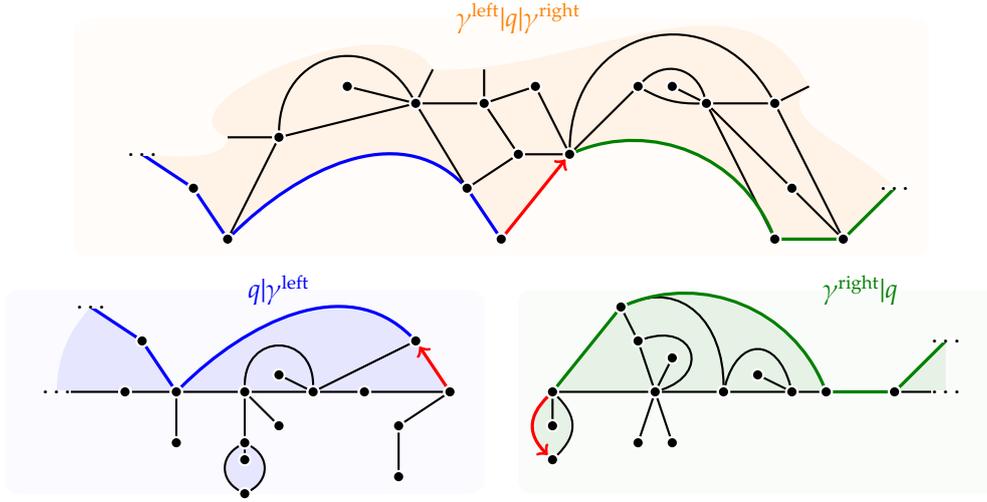

Consider the portion of the UIHPQ that lies to the right of the rightmost geodesic ray $\gamma^{\mathrm{right}}$; this is a rooted quadrangulation with a boundary (we consider any ``bubbles'' attached to the root vertex lying to the right of the root edge as being part of such a quadrangulation,  so that it also contains the root edge of the UIHPQ, see Figure~\ref{pencil decomposition}), which we call $\gamma^{\mathrm{right}}|\UIHPQ$. It is now easy to describe it in terms of a random positive treed bridge built from $\B_\infty$:

\begin{cor}\label{laws of two sides of gammaright}Consider the random infinite treed bridges $\B^r_\infty=((X^r_i)_{i\in\mathbb{Z}}; T^r)$ and $\B^l_\infty=((X^l_i)_{i\in\mathbb{Z}}; T^l)$ obtained from $\B_\infty=((X_i)_{i\in\mathbb{Z}}; T)$ as follows:
\begin{itemize}
\item for all $i<0$, $X^r_i=|i|$ and $T(i)$ is the labelled tree which consists of only a root vertex labelled $|i|$; on the other hand, $X^l_i=X_i$ and $T^l(i)=T(i)$;
\item for all $i>0$, $X^r_i=X_i$ and $T^r(i)=T(i)$; on the other hand, $X^l_i=i$ and $T^l(i)$ consists of only a root vertex labelled $i$.
\end{itemize}
Notice that $\B^r_\infty$ and $\B^l_\infty$ belong to $\TB^+_\infty$ almost surely.
If we consider $\UIHPQ$ as being $\Phi(\B_\infty)$, then we have $\gamma^{\mathrm{right}}|\UIHPQ
{=}\Phi(\B^r_\infty)$ and $\UIHPQ|\gamma^{\mathrm{right}}
{=}\Phi(\B^l_\infty)$. In particular, the two random variables $\gamma^{\mathrm{right}}|\UIHPQ$ and $\UIHPQ|\gamma^{\mathrm{right}}$ are independent.
\end{cor}

\begin{figure}\centering
\begin{tikzpicture}
\tikzset{every node/.style={font=\scriptsize}}
\useasboundingbox (-4,-3.5) rectangle (10,3);
\begin{scope}[cm={-1,0,0,1,(3,0)}]
\useasboundingbox (-6,-3.5) rectangle (6,1.5);
\tikzset{every path/.style={thick}}
\tikzstyle{real}=[inner sep=1.5pt, fill=black, circle, draw=white, thick]
\useasboundingbox (-6,-3.5) rectangle (6,1.5);
\contourlength{1pt}
\contournumber{100}
	\begin{pgfonlayer}{nodelayer}
\node [style=real] (0) at (0, -2) {};
		\node [style=real] (1) at (1.5, -2) {};
		\node [style=real] (2) at (2.5, -2) {};
		\node [style=real] (3) at (3.5, -2) {};
		\node [style=real] (4) at (4, -2) {};
		\node [style=real] (5) at (4, 0) {};
		\node [style=real] (6) at (5, -2) {};
		\node [style=real] (7) at (-1.25, -2) {};
		\node [style=real] (8) at (-2, -2) {};
		\node [style=real] (9) at (-3, -2) {};
		\node [style=real] (10) at (-4, -2) {};
		\node [style=real] (11) at (-4.75, -2) {};
		\node [style=real] (12) at (0, -2.5) {};
		\node [style=real] (13) at (0, -3) {};
		\node [style=real] (14) at (-0.75, -2.5) {};
		\node [style=real] (15) at (-0.75, -3.25) {};
		\node [style=real] (16) at (1.25, -2.75) {};
		\node [style=real] (17) at (1.75, -2.75) {};
		\node [style=real] (18) at (1, -0.75) {};
		\node [style=real] (19) at (1.25, -1.25) {};
		\node [style=real] (20) at (0.25, -0.75) {};
		\node [style=real] (21) at (-0.5, -1.25) {};
		\node [style=real] (22) at (3, -1.75) {};
		\node [style=real] (23) at (1.75, -1.5) {};
		\node [style=none] (24) at (0, -1.65) {\contour{white}0};
		\node [style=none] (25) at (1, -0.4) {\contour{white}1};
		\node [style=none] (26) at (1.2, -2.3) {\contour{white}1};
		\node [style=none] (27) at (-0.5, -0.95) {\contour{white}1};
		\node [style=none] (28) at (1.2, -1.6) {\contour{white}2};
		\node [style=none] (29) at (2.5, -2.3) {\contour{white}2};
		\node [style=real] (30) at (2, 0.25) {};
		\node [style=real] (31) at (3, 0) {};
		\node [style=real] (32) at (2.5, 0.25) {};
		\node [style=real] (33) at (4.25, -1.25) {};
		\node [style=none] (34) at (4, -2.3) {\contour{white}2};
		\node [style=none] (35) at (3.9, -0.3) {\contour{white}2};
		\node [style=real] (36) at (-2.5, -1.75) {};
		\node [style=real] (37) at (-3.25, -0.5) {};
		\node [style=real] (38) at (-1.25, 0) {};
		\node [style=real] (39) at (-0.25, 0) {};
		\node [style=real] (40) at (-4, -2.75) {};
		\node [style=real] (41) at (-3, -2.75) {};
		\node [style=real] (42) at (-3, -3.5) {};
		\node [style=real] (43) at (-3, -3) {};
		\node [style=real] (44) at (-2.5, -2.5) {};
		\node [style=real] (45) at (-2.25, 0.25) {};
		\node [style=none] (46) at (-1.25, -2.3) {\contour{white}1};
		\node [style=none] (47) at (-0.9, -2.5) {\contour{white}1};
		\node [style=none] (48) at (0.15, -2.5) {\contour{white}1};
		\node [style=none] (49) at (0, -3.3) {\contour{white}1};
		\node [style=none] (50) at (0.25, -1.05) {\contour{white}2};
		\node [style=none] (51) at (-0.75, -3.6) {\contour{white}2};
		\node [style=none] (52) at (-2, -2.3) {\contour{white}2};
		\node [style=none] (53) at (1.1, -3) {\contour{white}2};
		\node [style=none] (54) at (1.9, -3) {\contour{white}2};
		\node [style=none] (55) at (2, -0.1) {\contour{white}2};
		\node [style=none] (56) at (-4.15, -2.3) {\contour{white}2};
		\node [style=none] (66) at (-1.25, 0.35) {\contour{white}2};
		\node [style=none] (66) at (1.5, -1.4) {\contour{white}2};
		\node [style=none] (57) at (-4.15, -3) {\contour{white}3};
		\node [style=none] (58) at (-3.2, -0.8) {\contour{white}3};
		\node [style=none] (59) at (-3.25, -2.3) {\contour{white}3};
		\node [style=none] (60) at (-2.75, -1.75) {\contour{white}3};
		\node [style=none] (61) at (-2.5, 0.25) {\contour{white}3};
		\node [style=none] (62) at (-0.30, -0.25) {\contour{white}3};
		\node [style=none] (63) at (3.5, -2.3) {\contour{white}3};
		\node [style=none] (64) at (5, -2.3) {\contour{white}3};
		\node [style=none] (65) at (2.9, -0.25) {\contour{white}3};
		\node [style=none] (65) at (2.8, -1.8) {\contour{white}3};
		\node [style=none] (67) at (-3.25, -2.75) {\contour{white}4};
		\node [style=none] (68) at (-2.35, -2.75) {\contour{white}4};
		\node [style=none] (69) at (4, -1.25) {\contour{white}4};
		\node [style=none] (70) at (2.45, 0.5) {\contour{white}4};
		\node [style=none] (71) at (-3, -3.25) {\contour{white}5};
		\node [style=none] (71) at (-3, -3.75) {\contour{white}5};
		\node [style=real] (72) at (-4.5, -1.25) {};
		\node [style=none] (73) at (-4.25, -1.25) {\contour{white}3};
		\contourlength{2px}
		\node [style=none] (74) at (-5.25, -0.75) {\small\contour{white}{{\color{blue}$\gamma^{\mathrm{left}}$}}};
		\node [style=none] (75) at (-4.75, -2.3) {3};
		\node [style=none, fill=white] (76) at (-5.75, -2) {\small{$\ldots$}};
		\node [style=none] (77) at (5.75, -1.25) {\small\contour{white}{\color{green!50!black}$\gamma^{\mathrm{right}}$}};
		\node [style=none, fill=white] (78) at (5.75, -2) {\small{$\ldots$}};
		\node [style=none] (79) at (-4, -0.5) {};
		\node [style=none] (80) at (-1, 0.75) {};
		\node [style=none] (81) at (-0.25, 0.5) {};
		\node [style=none] (82) at (4.5, 0.25) {};
		\node [style=real] (83) at (0.5, 0.25) {};
		\contourlength{1px}
		\node [style=none] (84) at (0.7, 0.4) {\contour{white}2};
	\end{pgfonlayer}
	\begin{pgfonlayer}{edgelayer}
	\fill[orange!10] (74.center) to (72.center) to (10.center) [in=135, out=45] to (21.center) [bend right=0] to (0.center) to (18.center) [bend left=45] to (4.center) [bend left=0] to (6.center) to (77.center) [in=90] to (82.center) [out=-60, in=180] to (80.center) [out=-80, in=-40] to (79.center) [in=-20, out=80] to (74.center);
		\fill[blue!10] (0.center) to (21.center) [out=135, in=45] to (10.center) to [bend left=0] (0.center);
		\fill[blue!10, bend left=60, looseness=1.4] (41.center) to (42.center) [bend left=60, looseness=1.4] to (41.center);
		\fill[blue!10] (76.center) [bend left=20] to (74.center) [bend right=0] to (72.center) to (10.center) to (11.center) to (76.center);
		\fill [green!50!black!10] (0.center) to (18.center) [bend left=45] to (4.center) [bend left=0] to (0.center);
		\fill [green!50!black!10, bend right=45, looseness=1.3] (0.center) to (13.center) [bend right=45, looseness=1.3] to (0.center);
		\fill [green!50!black!10] (6.center) to (77.center) to (78.center) to (6.center);
		
		\draw [gray,->, very thick, bend right=45, looseness=1.25] (0) to (13);
		\draw [] (12) to (0);
		\draw [bend right=45, looseness=1.25] (13) to (0);
		\draw (7) to (0);
		\draw (0) to (1);
		\draw (8) to (7);
		\draw [bend left=90, looseness=2.00] (9) to (8);
				\draw[root] (0) to (14);
		\draw (10) to (9);
		\draw (15) to (14);
		\draw (16) to (1);
		\draw (1) to (17);
		\draw (1) to (2);
		\draw [bend left=75, looseness=2.00] (2) to (3);
		\draw [style=georight] (18) to (0);
		\draw (20) to (18);
		\draw [style=geoleft] (0) to (21);
		\draw (21) to (20);
		\draw (18) to (19);
		\draw [bend left=90, looseness=2.50] (19) to (1);
		\draw [in=90, out=20, looseness=1.05] (18) to (2);
		\draw (22) to (3);
		\draw (2) to (3);
		\draw (23) to (1);
		\draw (19) to (1);
		\draw (3) to (4);
		\draw (32) to (31);
		\draw (18) to (30);
		\draw (31) to (4);
		\draw [bend right=15] (30) to (31);
		\draw [bend right=60, looseness=1.25] (31) to (30);
		\draw (33) to (31);
		\draw (33) to (6);
		\draw (5) to (6);
		\draw [bend left=75, looseness=1.50] (18) to (5);
		\draw (21) to (8);
		\draw [style=geoleft, in=135, out=45] (10.center) to (21.center);
		\draw (36) to (8);
		\draw (9) to (8);
		\draw (10) to (37);
		\draw (38) to (21);
		\draw [bend left=75, looseness=1.50] (37) to (38);
		\draw (38) to (39);
		\draw (39) to (20);
		\draw (10) to (11);
		\draw (10) to (40);
		\draw [bend right=60, looseness=1.25] (42) to (41);
		\draw (9) to (41);
		\draw [bend left=60, looseness=1.25] (42) to (41);
		\draw (43) to (41);
		\draw (44) to (9);
		\draw (45) to (38);
		\draw (37) to (38);
		\draw (31) to (5);
		\draw (11) to (76.center);
		\draw (6) to (78.center);
		\draw (37) to (79.center);
		\draw (38) to (80.center);
		\draw (5) to (82.center);
		\draw (81.center) to (39);
		\draw (39) to (83);
		\draw (83) to (18);
		\draw [style=georight] (6) to (4);
		\draw [style=geoleft] (72) to (10);
		\draw [style=geoleft] (72) to (74.center);
		\draw [style=georight] (6) to (77.center);
		\draw [style=georight, bend left=45] (18) to (4);
	\end{pgfonlayer}
 \node (q) at (5,0.5) {\small$q^\flip$};
  \draw [->,very thick, label=right:\small$\flip$] (6.5,1.5) --(6.5,0)-- (5,0); 
  \end{scope}

 \begin{scope}[cm={0.5,0,0,0.5,(-1,3)}]
 
\tikzstyle{real}=[inner sep=1.5pt, fill=black, circle, draw=white, thick]
  \begin{pgfonlayer}{nodelayer}
		\node [style=real] (0) at (0, -2) {};
		\node [style=real] (1) at (1.5, -2) {};
		\node [style=real] (2) at (2.5, -2) {};
		\node [style=real] (3) at (3.5, -2) {};
		\node [style=real] (4) at (4, -2) {};
		\node [style=real] (5) at (4, 0) {};
		\node [style=real] (6) at (5, -2) {};
		\node [style=real] (7) at (-1.25, -2) {};
		\node [style=real] (8) at (-2, -2) {};
		\node [style=real] (9) at (-3, -2) {};
		\node [style=real] (10) at (-4, -2) {};
		\node [style=real] (11) at (-4.75, -2) {};
		\node [style=real] (12) at (0, -2.5) {};
		\node [style=real] (13) at (0, -3) {};
		\node [style=real] (14) at (-0.75, -2.5) {};
		\node [style=real] (15) at (-0.75, -3.25) {};
		\node [style=real] (16) at (1.25, -2.75) {};
		\node [style=real] (17) at (1.75, -2.75) {};
		\node [style=real] (18) at (1, -0.75) {};
		\node [style=real] (19) at (1.25, -1.25) {};
		\node [style=real] (20) at (0.25, -0.75) {};
		\node [style=real] (21) at (-0.5, -1.25) {};
		\node [style=real] (22) at (3, -1.75) {};
		\node [style=real] (23) at (1.75, -1.5) {};
		\node [style=none] (24) at (0, -1.75) {};
		\node [style=none] (25) at (1, -0.5) {};
		\node [style=none] (26) at (1.25, -2.25) {};
		\node [style=none] (27) at (-0.75, -1.25) {};
		\node [style=none] (28) at (1.25, -1.5) {};
		\node [style=none] (29) at (2.5, -2.25) {};
		\node [style=real] (30) at (2, 0.25) {};
		\node [style=real] (31) at (3, 0) {};
		\node [style=real] (32) at (2.5, 0.25) {};
		\node [style=real] (33) at (4.25, -1.25) {};
		\node [style=none] (34) at (4, -2.25) {};
		\node [style=none] (35) at (4, -0.25) {};
		\node [style=real] (36) at (-2.5, -1.75) {};
		\node [style=real] (37) at (-3.25, -0.5) {};
		\node [style=real] (38) at (-1.25, 0) {};
		\node [style=real] (39) at (-0.25, 0) {};
		\node [style=real] (40) at (-4, -2.75) {};
		\node [style=real] (41) at (-3, -2.75) {};
		\node [style=real] (42) at (-3, -3.5) {};
		\node [style=real] (43) at (-3, -3) {};
		\node [style=real] (44) at (-2.5, -2.5) {};
		\node [style=real] (45) at (-2.25, 0.25) {};
		\node [style=none] (46) at (-1.25, -2.25) {};
		\node [style=none] (47) at (-1, -2.5) {};
		\node [style=none] (48) at (0, -2.75) {};
		\node [style=none] (49) at (0, -3.25) {};
		\node [style=none] (50) at (0.25, -1) {};
		\node [style=none] (51) at (-0.75, -3.5) {};
		\node [style=none] (52) at (-2, -2.25) {};
		\node [style=none] (53) at (1, -3) {};
		\node [style=none] (54) at (1.75, -3) {};
		\node [style=none] (55) at (2, 0) {};
		\node [style=none] (56) at (-4.25, -2.25) {};
		\node [style=none] (57) at (-4.25, -3) {};
		\node [style=none] (58) at (-3.5, -0.5) {};
		\node [style=none] (59) at (-3.25, -2.25) {};
		\node [style=none] (60) at (-2.75, -1.75) {};
		\node [style=none] (61) at (-2.5, 0.25) {};
		\node [style=none] (62) at (-0.25, -0.25) {};
		\node [style=none] (63) at (3.5, -2.25) {};
		\node [style=none] (64) at (5, -2.25) {};
		\node [style=none] (65) at (2.75, -0.25) {};
		\node [style=none] (66) at (-1.25, 0.25) {};
		\node [style=none] (67) at (-3.25, -2.75) {};
		\node [style=none] (68) at (-2.25, -2.75) {};
		\node [style=none] (69) at (4, -1.25) {};
		\node [style=none] (70) at (2.25, 0.25) {};
		\node [style=none] (71) at (-3, -3.25) {};
		\node [style=real] (72) at (-4.5, -1.25) {};
		\node [style=none] (73) at (-4.25, -1.25) {};
		\contourlength{1px}
		\node [style=none] (74) at (-5.25, -0.75) {\contour{white}{\color{blue}$\gamma^{\mathrm{left}}$}};
		\node [style=none] (75) at (-4.75, -2.25) {};
		\node [style=none, fill=white] (76) at (-5.75, -2) {$\ldots$};
		\node [style=none] (77) at (5.75, -1.25) {\contour{white}{\color{green!50!black}$\gamma^{\mathrm{right}}$}};
		\node [style=none, fill=white] (78) at (5.75, -2) {$\ldots$};
		\node [style=none] (79) at (-4, -0.5) {};
		\node [style=none] (80) at (-1, 0.75) {};
		\node [style=none] (81) at (-0.25, 0.5) {};
		\node [style=none] (82) at (4.5, 0.25) {};
		\node [style=real] (83) at (0.5, 0.25) {};
		\node [style=none] (84) at (0.75, 0.5) {};
	\end{pgfonlayer}
	\begin{pgfonlayer}{edgelayer}
	\fill[orange!10] (74.center) to (72.center) to (10.center) [in=135, out=45] to (21) [bend right=0] to (0) to (18) [bend left=45] to (4) [bend left=0] to (6) to (77.center) [in=90] to (82.center) [out=-60, in=180] to (80.center) [out=-80, in=-40] to (79.center) [in=-20, out=80] to (74.center);
		\fill[blue!10] (0.center) to (21.center) [out=135, in=45] to (10.center) to [bend left=0] (0.center);
		\fill[blue!10, bend left=60, looseness=1.25] (41.center) to (42.center) [bend left=60, looseness=1.25] to (41.center);
		\fill[blue!10] (76.center) [bend left=20] to (74.center) [bend right=0] to (72.center) to (10.center) to (11.center) to (76.center);
		\fill [green!50!black!10] (0.center) to (18.center) [bend left=45] to (4.center) [bend left=0] to (0.center);
		\fill [green!50!black!10, bend right=45, looseness=1.25] (0) to (13) [bend right=45, looseness=1.25] to (0);
		\fill [green!50!black!10] (6.center) to (77.center) to (78.center) to (6.center);
		
		\draw (0) to (14);
		\draw [style=root, bend right=45, looseness=1.25] (0) to (13);
		\draw (12) to (0);
		\draw [bend right=45, looseness=1.25] (13) to (0);
		\draw (7) to (0);
		\draw (0) to (1);
		\draw (8) to (7);
		\draw [bend left=90, looseness=2.00] (9) to (8);
		\draw (10) to (9);
		\draw (15) to (14);
		\draw (16) to (1);
		\draw (1) to (17);
		\draw (1) to (2);
		\draw [bend left=75, looseness=2.00] (2) to (3);
		\draw [style=georight] (18) to (0);
		\draw (20) to (18);
		\draw [style=geoleft] (0) to (21);
		\draw (21) to (20);
		\draw (18) to (19);
		\draw [bend left=90, looseness=2.50] (19) to (1);
		\draw [in=90, out=24, looseness=1.25] (18) to (2);
		\draw (22) to (3);
		\draw (2) to (3);
		\draw (23) to (1);
		\draw (19) to (1);
		\draw (3) to (4);
		\draw (32) to (31);
		\draw (18) to (30);
		\draw (31) to (4);
		\draw [bend right=15] (30) to (31);
		\draw [bend right=60, looseness=1.25] (31) to (30);
		\draw (33) to (31);
		\draw (33) to (6);
		\draw [style=georight] (6) to (4);
		\draw (5) to (6);
		\draw [bend left=75, looseness=1.50] (18) to (5);
		\draw (21) to (8);
		\draw [style=geoleft, in=135, out=45] (10) to (21);
		\draw (36) to (8);
		\draw (9) to (8);
		\draw (10) to (37);
		\draw (38) to (21);
		\draw [bend left=75, looseness=1.50] (37) to (38);
		\draw (38) to (39);
		\draw (39) to (20);
		\draw (10) to (11);
		\draw (10) to (40);
		\draw [bend right=60, looseness=1.25] (42) to (41);
		\draw (9) to (41);
		\draw [bend left=60, looseness=1.25] (42) to (41);
		\draw (43) to (41);
		\draw (44) to (9);
		\draw (45) to (38);
		\draw (37) to (38);
		\draw (31) to (5);
		\draw [style=geoleft] (72) to (10);
		\draw [style=geoleft] (72) to (74.center);
		\draw [style=georight] (6) to (77.center);
		\draw (11) to (76.center);
		\draw (6) to (78.center);
		\draw (37) to (79.center);
		\draw (38) to (80.center);
		\draw (5) to (82.center);
		\draw (81.center) to (39);
		\draw (39) to (83);
		\draw (83) to (18);
		\draw [style=georight, bend left=45] (18) to (4);
		\end{pgfonlayer}
  \end{scope}
\end{tikzpicture}
\caption{\label{flip} \small The flipped version $q^\flip$ of the quadrangulation $q$ from Figure~\ref{geodesic rays}; notice how the rightmost and leftmost geodesic rays are exchanged, and the root is rotated so as to be oriented clockwise with respect to the outerface.}
\end{figure}
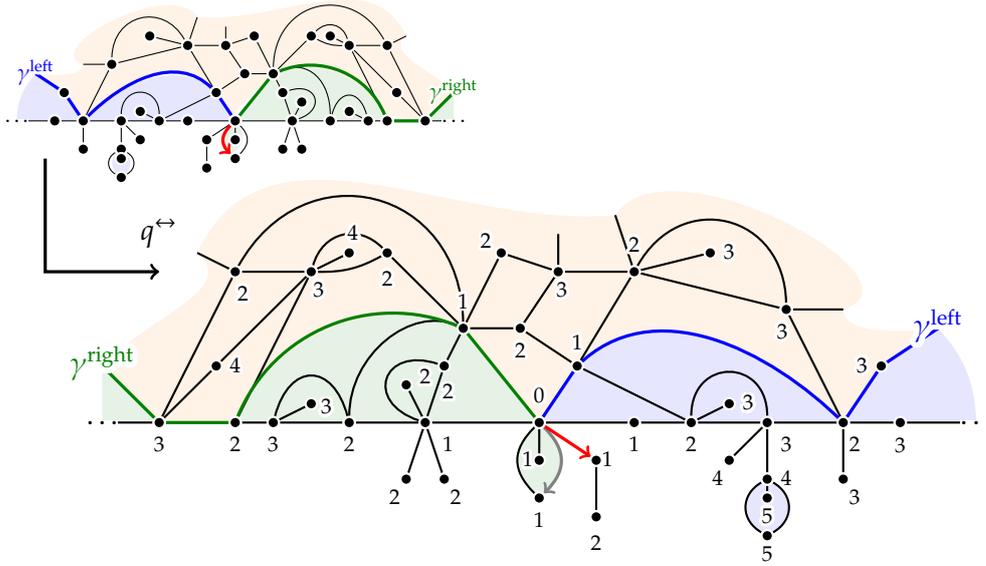
Consider now, given $q\in\sQ$, the ``flip'' (Figure~\ref{flip}) performed by applying a reflection to $q$ and then rerooting the map thus obtained in the following way: consider the image of the root edge; reroot on the next edge to be found by rotating around the (image of) the root vertex counterclockwise. We call the image of $q$ via this flip $q^\flip$.

The flip is an involution on the set $\sQ$, and is easily seen to be continuous for the local distance. Notice that the rightmost and leftmost geodesic rays become exchanged in $q^\flip$; it is easy to see that $(\gamma^{\mathrm{right}}|q)^\flip=q^\flip|\gamma^{\mathrm{left}}$, where $\gamma^{\mathrm{right}}|q$ is the quadrangulation obtained by taking what lies to the right of the rightmost geodesic ray of $q$ (including bubbles lying to the right of the root edge, as before), and $q^\flip|\gamma^{\mathrm{left}}$ the part of $q^\flip$ that lies left of the leftmost geodesic ray (rooted on the first edge of $\gamma^{\mathrm{left}}$). This, since $\UIHPQ^\flip$ has the same law as $\UIHPQ$, implies that $\UIHPQ|\gamma^{\mathrm{left}}$ has the same law as $(\gamma^{\mathrm{right}}|\UIHPQ)^\flip$.

Using the notation $\gamma^{\mathrm{left}}|\UIHPQ|\gamma^{\mathrm{right}}$ for the random quadrangulation $(\gamma^{\mathrm{left}}|\UIHPQ)|\gamma^{\mathrm{right}}$ (which is the same as $\gamma^{\mathrm{left}}|(\UIHPQ|\gamma^{\mathrm{right}})$), this yields:

\begin{prop}[Pencil decomposition]
Let $\UIHPQ$ be the UIHPQ; then the random variables $\UIHPQ|\gamma^{\mathrm{left}}$, $\gamma^{\mathrm{right}}|\UIHPQ$, $\gamma^{\mathrm{left}}|\UIHPQ|\gamma^{\mathrm{right}}$ defined as above (see Figure~\ref{pencil decomposition}) are independent.
\end{prop}

\begin{proof}
Notice that $\gamma^{\mathrm{left}}|\UIHPQ$ and $\gamma^{\mathrm{left}}|\UIHPQ|\gamma^{\mathrm{right}}$ are Borel-measurable functions of the whole $\UIHPQ|\gamma^{\mathrm{right}}$, which is independent of $\gamma^{\mathrm{right}}|\UIHPQ$ by Corollary~\ref{laws of two sides of gammaright}.

Let $f_l$, $f_c$, $f_r$ be bounded Borel-measurable functions on the space $\sQ$; then thanks to the above observation we have
$$\mathbb{E}[f_l(\UIHPQ|\gamma^{\mathrm{left}})f_c(\gamma^{\mathrm{left}}|\UIHPQ|\gamma^{\mathrm{right}})f_r(\gamma^{\mathrm{right}}|\UIHPQ)]=\mathbb{E}[f_l(\UIHPQ|\gamma^{\mathrm{left}})f_c(\gamma^{\mathrm{left}}|\UIHPQ|\gamma^{\mathrm{right}})]\mathbb{E}[f_r(\gamma^{\mathrm{right}}|\UIHPQ)].$$
On the other hand, since $\UIHPQ|\gamma^{\mathrm{left}}=(\gamma^{\mathrm{right}}|\UIHPQ^\flip)^\flip$ and $\gamma^{\mathrm{left}}|\UIHPQ|\gamma^{\mathrm{right}}=(\gamma^{\mathrm{left}}|\UIHPQ^\flip|\gamma^{\mathrm{right}})^\flip$, we have
$$\mathbb{E}[f_l(\UIHPQ|\gamma^{\mathrm{left}})f_c(\gamma^{\mathrm{left}}|\UIHPQ|\gamma^{\mathrm{right}})]=\mathbb{E}[f_l((\gamma^{\mathrm{right}}|\UIHPQ^\flip)^\flip)f_c((\gamma^{\mathrm{left}}|\UIHPQ^\flip|\gamma^{\mathrm{right}})^\flip)];$$
since $\UIHPQ$ has the same law as $\UIHPQ^\flip$, the above is the same as
$$\mathbb{E}[f_l((\gamma^{\mathrm{right}}|\UIHPQ)^\flip)f_c((\gamma^{\mathrm{left}}|\UIHPQ|\gamma^{\mathrm{right}})^\flip)]=\mathbb{E}[f_l((\gamma^{\mathrm{right}}|\UIHPQ)^\flip)]\mathbb{E}[f_c((\gamma^{\mathrm{left}}|\UIHPQ|\gamma^{\mathrm{right}})^\flip)]=$$
$$=\mathbb{E}[f_l(\UIHPQ|\gamma^{\mathrm{left}})]\mathbb{E}[f_c(\gamma^{\mathrm{left}}|\UIHPQ|\gamma^{\mathrm{right}})],$$ which proves the proposition.
\end{proof}

In this pencil decomposition, the quadrangulations $\gamma^{\mathrm{right}}|\UIHPQ$ and $\UIHPQ|\gamma^{\mathrm{left}}$ have boundaries that are ``free'' on one side of the root vertex (the right and left side respectively), while the other half of the boundary is a strict geodesic, in the sense that there is no way to join any two of its vertices by a geodesic other than following the boundary itself. This boundary condition is reminiscent of the works by Bouttier \& Guitter \cite{BG12} and Le Gall \cite{LG11} on quadrangulations with geodesic boundaries involved in the ``slice decomposition''.

 We shall now study the geometry of the quadrangulation $\gamma^{\mathrm{left}}|\UIHPQ|\gamma^{\mathrm{right}}$ and show that it does look like a ``pencil'' at large scales; in order to do this, we need a more detailed description of $\gamma^{\mathrm{left}}$.

\subsection{The leftmost geodesic ray}\label{leftmost}

\begin{prop}\label{leftmost geodesic ray as a limit}Given a treed bridge $(b;T)\in\TB_\infty^+$, define a sequence of (finite) geodesics $(\gamma^n)_{n\geq 1}$ so that $\gamma^n=(c^n_i)_{1\leq i\leq n}$, where $v(c^n_n)$ is the root of the leftmost tree in $T(\DS(b))$ having root label $n$, $c_n^n$ is its rightmost corner, and for each $i=1,\ldots, n-1$ the corner $c^n_i$ is the rightmost among those such that $v(c^n_i)=v(s(c^n_{i+1}))$. Then $\gamma^n$ converges to $\gamma$ as $n\to\infty$, where $\gamma$ is the leftmost geodesic ray in $\Phi((b;T))$, monotonically with respect to the left-to-right order $\prec$.\end{prop}

\begin{figure}
\vspace{-2.5cm}
\begin{tikzpicture}
\tikzstyle{real}=[fill=black, draw=white, very thick, circle, inner sep=2pt]
\tikzset{every path/.style = {thick}}
	\begin{pgfonlayer}{nodelayer}
	\contourlength{.8pt}
		\node [style=phantom] (0) at (6, -2) {};
		\node [style=real] (1) at (4.5, -2) {};
		\node [style=real] (2) at (3, -2) {};
		\node [style=real] (3) at (1.5, -2) {};
		\node [style=phantom] (4) at (0, -2) {};
		\node [style=real] (5) at (-1.5, -2) {};
		\node [style=real] (6) at (-3, -2) {};
		\node [style=none] (7) at (4.5, -2.25) {0};
		\node [style=none] (8) at (3, -2.25) {\contour{white}1};
		\node [style=none] (9) at (0, -2.25) {\contour{white}1};
		\node [style=none] (10) at (-1.5, -2.25) {\contour{white}2};
		\node [style=none] (11) at (1.5, -2.25) {\contour{white}2};
		\node [style=real] (12) at (3.25, -1.25) {};
		\node [style=real] (13) at (3.25, -0.5) {};
		\node [style=real] (14) at (1.5, -1.25) {};
		\node [style=real] (15) at (-1.5, -1.25) {};
		\node [style=real] (16) at (-1.5, -0.5) {};
		\node [style=phantom] (17) at (-4.5, -2) {};
		\node [style=none] (18) at (-3, -2.25) {\contour{white}3};
		\node [style=none] (19) at (1.25, -1.5) {\contour{white}1};
		\node [style=none] (20) at (3.5, -1.25) {\contour{white}1};
		\node [style=none] (21) at (6, -2.25) {\contour{white}1};
		\node [style=none] (22) at (-1.5, -0.25) {\contour{white}2};
		\node [style=none] (23) at (-1.25, -1.25) {\contour{white}3};
		\node [style=none] (24) at (3, -0.5) {\contour{white}2};
		\node [style=none] (25) at (-4.5, -2.25) {\contour{white}2};
		\node [style=real] (26) at (-6, -2) {};
		\node [style=real] (27) at (-6, -1.25) {};
		\node [style=real] (28) at (-6.5, -0.5) {};
		\node [style=real] (29) at (-5.5, -0.5) {};
		\node [style=real] (30) at (-3, -1.25) {};
		\node [style=real] (31) at (-3, -0.5) {};
		\node [style=real] (32) at (-3.5, 0.5) {};
		\node [style=none] (33) at (-2.75, -0.5) {\contour{white}2};
		\node [style=none] (34) at (-3.25, -1.25) {\contour{white}3};
		\node [style=none] (35) at (-3.75, 0.5) {\contour{white}1};
		\node [style=none] (36) at (-6, -2.25) {\contour{white}3};
		\node [style=none] (37) at (-6.25, -1.25) {\contour{white}2};
		\node [style=none] (38) at (-5.25, -0.25) {\contour{white}2};
		\node [style=none] (39) at (-6.75, -0.25) {\contour{white}3};
		\node [fill=white] (40) at (-7.5, -2) {$\ldots$};
		\node [fill=white] (41) at (7.5, -2) {$\ldots$};
		\contourlength{1pt}
		\contournumber{100}
		\node (42) at (3.75, -1.6) {\contour{white}{$\gamma^1$}};
		\node (43) at (2, 0.25) {\contour{white}{$\gamma^2$}};
		\node (44) at (-2.1, 0.8) {\contour{white}{$\gamma^3$}};
	\end{pgfonlayer}
	\begin{pgfonlayer}{edgelayer}
			\fill [blue!10,  in=72, out=152, looseness=7.3] (2.center) to (1.center) [bend right=65, looseness=1.2] to (2.center);
			\fill [green!10, bend left=75, looseness=1.85] (31.center) to (14.center) [bend right=10, looseness=1] to (5.center) [bend right=30] to (6.center) [bend left=90, looseness=1.2] to (31.center);
		\draw (2) to (12);
		\draw [style=map, bend right=60] (12) to (13);
		\draw [style=geoone, bend left=60] (2) to (1);
		\draw [style=map, in=165, out=60] (3) to (2);
		\draw (3) to (14);
		\draw (3) to (4);
		\draw (4) to (5);
		\draw [style=map, bend right=60, looseness=1.25] (5) to (15);
		\draw [style=map, bend left=60, looseness=1.25] (15) to (16);
		\draw (17) to (6);
		\draw [style=geotwo, bend left=15, looseness=0.75] (5) to (14);
		\draw [style=geotwo, bend left=105, looseness=2.50] (14) to (1);
		\draw [style=map, bend left] (16) to (14);
		\draw [style=map, bend left=60] (12) to (1);
		\draw [bend right=15, looseness=0.00] (2) to (1);
		\draw (1) to (0);
		\draw (15) to (5);
		\draw (3) to (2);
		\draw [style=map, in=105, out=119, looseness=2.75] (5) to (14);
		\draw (15) to (16);
		\draw (13) to (12);
		\draw [style=map, in=90, out=135, looseness=3.75] (12) to (1);
		\draw [style=map, in=75, out=150, looseness=6.25] (2) to (1);
		\draw (26) to (17);
		\draw (27) to (29);
		\draw [style=geothree, bend left=15, looseness=0.75] (26) to (31);
		\draw [style=geothree, bend left=75, looseness=1.75] (31) to (14);
		\draw [style=map, bend left=75] (6) to (31);
		\draw [style=map, bend left=45, looseness=1.25] (30) to (31);
		\draw [style=map, bend left=15] (30) to (5);
		\draw [style=map, bend left] (6) to (5);
		\draw [style=map, bend right=60] (32) to (31);
		\draw [style=map, in=45, out=60, looseness=1.25] (32) to (1);
		\draw (30) to (6);
		\draw (31) to (30);
		\draw (32) to (31);
		\draw [style=map, bend left=120, looseness=2.50] (27) to (32);
		\draw [style=map, bend left=60, looseness=1.25] (29) to (32);
		\draw [style=map, in=120, out=90, looseness=1.75] (27) to (32);
		\draw [style=map, bend left=45, looseness=0.75] (28) to (27);
		\draw [style=map, bend right=15, looseness=0.75] (27) to (32);
		\draw [style=map, bend left=90, looseness=1.25] (26) to (27);
		\draw (26) to (27);
		\draw (28) to (27);
		\draw (26) to (40.center);
		\draw (0) to (41.center);
		\draw (6) to (5);
	\end{pgfonlayer}
\end{tikzpicture}
\caption{\small The sequence of increasing geodesics $\gamma^1, \gamma^2, \gamma^3,\ldots$.}
\end{figure}
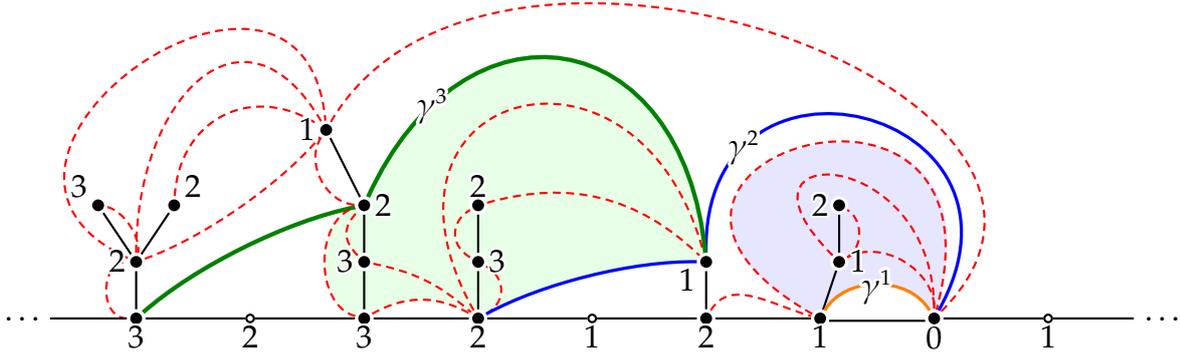

\begin{proof}
The sequence of geodesics $(\gamma^n)_{n\geq 1}$ is indeed increasing for $\prec$, in the sense that $\gamma^{n}$ lies left of $\gamma^{n+1}$ for each $n>0$. As a consequence, finite initial segments stabilise after finite time: for each $k$ there is $n_k$ such that, for $n\geq n_k$, we have $c^n_j=c^{n_k}_j$ for $j=1,\ldots,k$ (that is, the geodesics $\gamma^n$, for $n\geq n_k$, coincide up to distance $k$ from the root).

Consider now the sequence of corners $\gamma=(c_k^{n_k})_{k>0}$; such a sequence is a geodesic ray (since for any $k>1$ we must have $v(s(c_k^{n_k}))=v(c_{k-1}^{n_k})=v(c_{k-1}^{n_{k-1}})$), lies to the right of each geodesic in the sequence $(\gamma^n)_{n\geq 1}$ and is the leftmost geodesic with this property.

We now claim that $\gamma$ is the leftmost geodesic ray in $(b;T)$; all we need to show is that there cannot be a geodesic ray lying left of $\gamma$. Equivalently, we show that for each geodesic $\gamma^n$, all geodesic rays lie to the right of $\gamma^n$. Suppose $\gamma'=(c'_i)_{i\in\mathbb{N}^+}$ is a geodesic ray, and suppose part of it lies strictly to the left of $\gamma^n$ for some $n>0$. Since the region of $\Phi((b;T))$ lying left of $\gamma^n$ is finite, by following $\gamma'$ away from the root vertex, we must eventually leave it. This amounts to saying that we can take 
$$m=\max\left\{i\in\{1,\ldots,n\}\bigm\vert v(c'_i)=v(c^n_i) \mbox{ and the edge $(c'_i, s(c'_i))$ lies strictly left of $\gamma^n$}\right\},$$
since if the set above is empty, then $\gamma'$ lies to the right of $\gamma^n$. This, however, is a contradiction: $c^n_i$ is defined as the rightmost corner of $v(c^n_i)$, hence there is no edge issued from $v(c^n_i)$ lying strictly to the left of the edge $(c^n_i, s(c^n_i))$, which belongs to $\gamma^n$.
\end{proof}

\begin{prop}\label{law of middle section}
Let $\B^c_\infty=((X^c_i)_{i\in\mathbb{Z}};T^c)$ be a random treed bridge such that
\begin{itemize}
\item $X^c_i=|i|$ for all $i\in\mathbb{Z}$;
\item the trees $T^c(i)$ are independent and distributed according to $\rho^+_{|i|}$ for all $i<0$.
\end{itemize}
Notice that $\B^c_\infty$ belongs to $\TB^+_\infty$ almost surely. Then $\gamma^{\mathrm{left}}|\UIHPQ|\gamma^{\mathrm{right}}$ has the same law as $\Phi(\B^c_\infty)$.
\end{prop}

\begin{proof}
Let $(\gamma^n)_{n\geq 1}$ be the (random) sequence of geodesics in $\B_\infty$ defined within Proposition~\ref{leftmost geodesic ray as a limit}, and $\gamma^\mathrm{left}=(c_i)_{i\in\mathbb{N}^+}$ be the leftmost geodesic ray in $\B_\infty$. Call $T'(-i)$, for $i>0$, the tree of descendants of $v(c_i)$ that lie ``above'' the geodesic ray; then $\gamma^{\mathrm{left}}|\Phi(\B_\infty)|\gamma^{\mathrm{right}}$ is the image via $\Phi$ of the treed bridge $(b;T')$, where $b=(|i|)_{i\in\mathbb{Z}}$. All we need to show is, therefore, that for each $i>0$ the tree $T'(-i)$ has law $\rho^+_i$, and that all such trees are independent.

Given any fixed $n>0$, consider the trees $t^n_1,\ldots,t^n_{n}$ lying above the geodesic $\gamma^n$, rooted in $v(c^{n}_1),\ldots,v(c^{n}_{n})$; we show they are independent and distributed according to $\rho^+_1,\ldots,\rho^+_{n}$. This is essentially a consequence of the fact, remarked upon in Section~\ref{trees}, that $\rho^+_i$ is the law of a multi-type Galton--Watson tree. Consider the part of $\B_\infty$ that lies below $\gamma^{n}$ and notice that it is independent of trees $t^n_1,\ldots,t^n_{n}$ (given an instance of $\B_\infty$, substituting such trees for another $n$-tuple $t'^n_1,\ldots,t'^n_{n}$ would still yield the same geodesic $\gamma^{n}$, since $s(c^n_i)$, for $i=1,\ldots,n$ only depends on corners which lie below $\gamma^{n}$). From this observation the claim easily follows.

Now, given any positive integer $k$, consider the trees $T'(-1),\ldots, T'(-k)$ and let $f_1,\ldots,f_k$ be measurable bounded functions on $\LT^+$. We have
$$\mathbb{E}\left[\prod_{i=1}^kf_i(T'(-i))\right]=\lim_{n\to\infty}\mathbb{E}\left[\prod_{i=1}^kf_i(t^n_i)\right]=\lim_{n\to\infty}\prod_{i=1}^k\mathbb{E}\left[f_i(t^n_i)\right],$$
where the first equality holds by dominated convergence and for the second we use independence of the trees $t_1^n,\ldots,t_n^n$.
Now, since $\lim_{n\to\infty}\mathbb{E}\left[f_i(t^n_i)\right]=\mathbb{E}\left[f_i(T'(-i))\right]$ for each $i=1,\ldots,k$, one obtains that the trees $T'(-1),\ldots, T'(-k)$ are independent and distributed according to $\rho^+_1,\ldots,\rho^+_k$, as wanted.
\end{proof}

The following result is analogous to~\cite{CMMinfini}:
\begin{cor}\label{infinite cut points}
The quadrangulation $\gamma^{\mathrm{left}}|\UIHPQ|\gamma^{\mathrm{right}}$ almost surely has an infinite number of pinch points; in other words, the leftmost and rightmost geodesic rays in the UIHPQ $\UIHPQ$ almost surely have an infinite number of vertices in common (through which all geodesic rays in $\UIHPQ$ must pass).
\end{cor}

\begin{proof}
The two rays meet in a vertex labelled $k$ if the leftmost corner labelled $k$ in the treed bridge $\B^c_\infty=((X^c_i)_{i\in\mathbb{Z}},T^c)$ is actually the vertex $X^c_{-k}$ of the bridge of $\B^c_\infty$, which corresponds to the leftmost geodesic ray.

Since the bridge of $\B^c_\infty$ is geodesic, the leftmost tree whose root is labelled $k$ is $T^c(-k)$; the above event  thus corresponds to the intersection of events $\cap_{i>k}\{T^c(-i)\mbox{ has no vertices labelled $k$}\}$, which (by Proposition~\ref{law of middle section}) has probability
$$p_k=\prod_{i>k} \mathbb{P}(T^c(-i)\mbox{ has no vertices labelled $k$})=\prod_{i>k}w_{i-k}/w_i=2^{-k}\prod_{i=1}^kw_i=\frac{1}{3}\frac{k+3}{k+1}.$$

Hence the probability that the two geodesics have an infinite number of vertices in common is at least $\liminf_{k\rightarrow\infty}p_k$, which is $1/3$ (thus strictly positive). As a consequence, since the event is a tail event for the sequence of independent random variables $(T^c(-i))_{i\geq 1}$, such a probability is 1 by Kolmogorov's 0-1 law.\end{proof}

\subsection{Extension of results from~\cite{CMMinfini}}\label{Extension of results from CMMinfini}

As remarked in Section~\ref{Boltzmann quadrangulations and statement of the theorem}, the UIHPQ can be presented as $\Phi(\B^\pm_\infty)$, where $\B^\pm_\infty$ is the uniform infinite treed bridge from Definition~\ref{unconstrained infinite bridge}.

Notice that, while in the construction $\UIHPQ=\Phi(\B_\infty)$ labels in the treed bridge have a clear geometric interpretation (as distances to the root vertex in $\UIHPQ$), interpreting labels of $\B^\pm_\infty$ as a function of $\UIHPQ$ is less immediate. The following result is analogous to what one obtains in the case of the UIPQ~\cite{CMMinfini}; it is stated as an open question in~\cite{CMboundary}, and this section will be devoted to providing a proof.

\begin{thm}\label{labels as limit}
Consider the UIHPQ as $\UIHPQ=\Phi(\B^{\pm}_\infty)$ (where $\B^\pm_\infty$ is as in Definition~\ref{unconstrained infinite bridge}); then almost surely, for each pair of vertices $x,y$ of $\UIHPQ$, writing $l(x)$ and $l(y)$ for the labels they bear in $\B^\pm_\infty$, we have
\begin{equation}\label{eq:labels as limit}l(x)-l(y)=\lim_{z\to\infty}\mathrm{d_{gr}}(x,z)-\mathrm{d_{gr}}(y,z),\end{equation}
where by $z\to\infty$ we mean that $z$ leaves any finite region of $\UIHPQ$.
\end{thm}

In order to prove this result we work along the same lines as in~\cite{CMMinfini}; one of the ingredients that were missing in~\cite{CMboundary} in order to obtain Theorem~\ref{labels as limit} is the coalescence of geodesic rays implied by our new construction of the UIHPQ.

To begin with, we need to have a look at how geodesic rays appear in $\B^\pm_\infty$; we call a \emph{proper geodesic ray issued from $x$} in $\B^\pm_\infty$, where $x$ is a (real) vertex, a sequence of corners $\gamma=(c_i)_{i\geq 0}$ such that $v(c_0)=x$ and that $v(s(c_i))=v(c_{i+1})$ for all $i\geq 0$, so that $l(c_{i})=l(x)-i$. Notice that, while proper geodesic rays issued from a vertex in $\B^\pm_\infty$ do correspond to geodesics in $\UIHPQ$ (because $\mathrm{d_{gr}}(x,y)\geq |l(x)-l(y)|$ for all vertices in $\UIHPQ$), it is not clear that every geodesic ray issued from a vertex in $\UIHPQ$ corresponds to a proper geodesic ray issued from that vertex in $\B_\infty^\pm$; we know thanks to Corollary~\ref{infinite cut points}, however, that this is the case for geodesic rays issued from the root vertex:

\begin{prop}\label{all geodesic rays are proper}
Every geodesic ray (issued from the root vertex) in $\UIHPQ=\Phi(\B_\infty^\pm)$ is \emph{proper}.
\end{prop}

\begin{proof}
This follows from Corollary~\ref{infinite cut points} by the same argument as in~\cite{CMMinfini}: we know that there is (almost surely) an infinite set $I$ of vertices in $\UIHPQ$ such that every geodesic ray visits all of the vertices in the set. Notice that there exists a proper geodesic ray in $\UIHPQ$, since one is obtained by simply iterating the successor function starting from a corner adjacent to the root vertex; call such a geodesic ray $\gamma$, and call $\gamma(i)$, for $i\in\mathbb{N}$, the $i$-th vertex visited by $\gamma$. We have $I\subseteq \{\gamma(i)\mid i\geq 0\}$, so for every vertex $v$ in $I$ we have $v=\gamma(j)$ (where $j$ is the graph distance between $v$ and the root vertex) hence $l(v)=-j$. Now take any geodesic ray $\gamma'$; we know all vertices in $I$ belong to $\gamma'$, hence the set $\{i\mid l(\gamma'(i))=-i\}$ is infinite; but then, since labels vary by at most one along a geodesic ray (indeed, along any path), $l(\gamma'(i))=-i$ for all $i\geq 0$, and $\gamma'$ is proper. \end{proof}

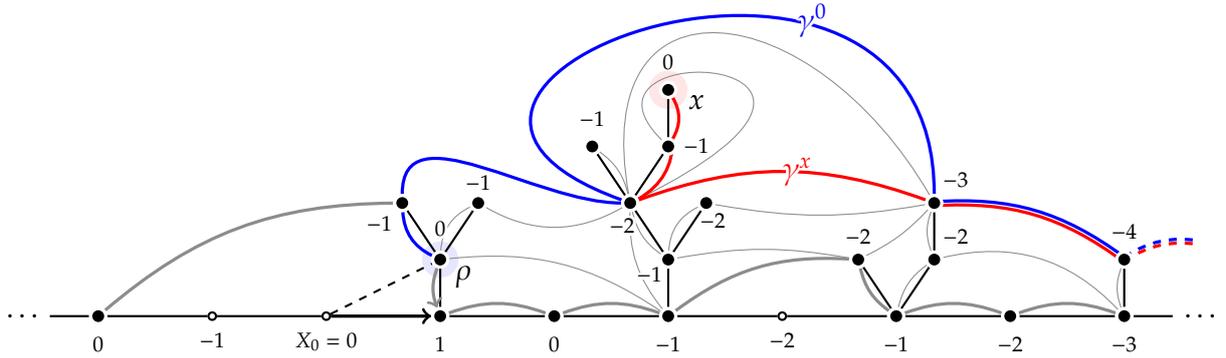
\begin{figure}
\centering
\begin{tikzpicture}
\tikzstyle{real}=[fill=black, draw=white, very thick, circle, inner sep=2pt]
\tikzstyle{boundary}=[very thick, gray!90]
\tikzstyle{map}=[thin, gray!90]
\tikzset{every path/.style = {thick}}
\tikzset{every node/.style={font=\scriptsize}}
\begin{pgfonlayer}{nodelayer}
\contourlength{1px}
		\node [style=real, label=below:$0$] (0) at (-4, 0) {};
		\node [style=phantom, label=below:$-1$] (1) at (-2.5, 0) {};
		\node [style=phantom, label={below:$X_0=0$}] (2) at (-1, 0) {};
		\node [style=real, label=below:$1$] (3) at (0.5, 0) {};
		\node [style=real, label=below:$0$] (4) at (2, 0) {};
		\node [style=real, label=below:$-1$] (5) at (3.5, 0) {};
		\node [style=phantom, label=below:$-2$] (6) at (5, 0) {};
		\node [style=real, label=below:$-1$] (7) at (6.5, 0) {};
		\node [style=real, label=below:$-2$] (8) at (8, 0) {};
		\node [style=real, label=below:$-3$] (9) at (9.5, 0) {};
		\node [style=real, label=\contour{white}{$0$}] (3-1) at (0.5, 0.75) {};
		\node [style=real, label={[xshift=3pt,yshift=2pt]below left:\contour{white}{$-1$}}] (3-11) at (0, 1.5) {};
		\node [style=real, label={[yshift=-4pt]\contour{white}{$-1$}}] (3-12) at (1, 1.5) {};
		\node [style=real, label={[xshift=-7pt,yshift=5pt]below:\contour{white}{$-1$}}] (5-1) at (3.5, 0.75) {};
		\node [style=real, label={[yshift=2pt,xshift=-3pt]below:\contour{white}{$-2$}}] (5-11) at (3, 1.5) {};
		\node [style=real, label={[xshift=-10pt,yshift=-7pt]right:\contour{white}{$-2$}}] (5-12) at (4, 1.5) {};
		\node [style=real, label=$-1$] (5-111) at (2.5, 2.25) {};
		\node [style=real, label={[xshift=-2pt]right:\contour{white}{$-1$}}] (5-112) at (3.5, 2.25) {};
		\node [style=real, label=$0$] (5-1121) at (3.5, 3) {};
		\node [style=real, label={[yshift=-2pt]$-2$}] (7-1) at (6, 0.75) {};
		\node [style=real, label={[xshift=8pt,yshift=-3pt]\contour{white}{$-2$}}] (7-2) at (7, 0.75) {};
		\node [style=real, label={[xshift=8pt,yshift=-3pt]\contour{white}{$-3$}}] (7-21) at (7, 1.5) {};
		\node [style=real, label=$-4$] (9-1) at (9.5, 0.75) {};
		\node (left) at (-5,0) {\small$\ldots$};
		\node (right) at (10.5,0) {\small$\ldots$};
		\node at (5.4, 3.9) {\contour{white}{\color{blue}\small$\gamma^{0}$}};
		\node at (5.2, 1.9) {\contour{white}{\color{red}\small$\gamma^{x}$}};
	\end{pgfonlayer}
	\begin{pgfonlayer}{edgelayer}
	\node[fill=blue!10, circle, inner sep=.1cm, label={[yshift=7pt,xshift=-4pt]below right:\small$\rho$}] (rho) at (0.5,0.75) {x};
	\node[fill=red!10, circle, inner sep=.18cm, label={[yshift=8pt,xshift=-2pt]below right:\small$x$}] (x) at (5-1121) {};
	\draw (left)--(0)--(1)--(2)--(3)--(4)--(5)--(6)--(7)--(8)--(9)--(right);
	\draw[very thick, ->] (2)--(3);
	
	\draw[dashed] (2)--(3-1);
	
	\draw (3)--(3-1)--(3-11);
	\draw (3-1)--(3-12);
	\draw (5)--(5-1)--(5-11)--(5-111);
	\draw (5-1)--(5-12);
	\draw (5-11)--(5-112)--(5-1121);
	\draw (7-1)--(7)--(7-2)--(7-21);
	\draw (9)--(9-1);
	
	\draw [map,bend left=30] (3-1) to (3-12);
	\draw [map,bend right=30] (3-12) to (5-11);
	\draw [map, bend left=20] (3-1) to (5);
	\draw [map,bend left=20] (5) to (5-11);
	\draw [map,bend left=20] (5-1) to (5-11);
	\draw [map,bend left=20] (5-111) to (5-11);
	\draw [map, out=100, in=120, looseness=2] (5-11) to (7-21);
	\draw [map] (5-112) .. controls ([xshift=-1.5cm,yshift=1.5cm]5-112)  and ([xshift=4cm,yshift=2cm]5-11) .. (5-11);
	\draw [map,bend left=30] (5-1) to (5-12);
	\draw [map,bend right=10] (5-12) to (7-21);
	\draw [map,bend left=10] (5-1) to (7-1);
	\draw [map,bend right=10] (7-1) to (7-21);
	\draw [map,bend left=20] (7) to (7-2);
	\draw [map,bend left=20] (7-2) to (7-21);
	\draw [map,bend left=30] (7-2) to (9);
	\draw [map,bend left=20] (9) to (9-1);
	
	\draw[very thick, blue, bend left=30] (3-1) to (3-11);
	\draw[very thick, blue, out=90, in=180] (3-11) to (5-11);
	\draw[very thick, blue, out=160, in=90, looseness=2.8] (5-11) to (7-21);
	\draw[very thick, blue, bend left=20] (7-21)+(0,0.5pt) to ([yshift=1pt]9-1);
	\draw[very thick, blue, bend left=20,dashed] (9-1)+(0,1pt) to (10.5,1);
	\draw[very thick, red, bend left=30] (5-1121) to (5-112) to (5-11);
	\draw[very thick, red, bend left=20] (5-11) to (7-21);
	\draw[very thick, red, bend left=20] (7-21)+(0,-1pt) to ([yshift=-1pt]9-1);
	\draw[very thick, red, bend left=20,dashed] (9-1)+(0,-0.5pt) to (10.5,0.95);
	\draw[boundary, bend left=20] (0) to (3-11);
	\draw[boundary, bend right=20, ->] (3-1) to (3);
	\draw[boundary, bend left=20] (3) to (4);
	\draw[boundary, bend left=20] (4) to (5);
	\draw[boundary, bend left=20] (5) to (7-1);
	\draw[boundary, bend right=20] (7-1) to (7);
	\draw[boundary, bend left=20] (7) to (8);
	\draw[boundary, bend left=20] (8) to (9);
	\end{pgfonlayer}
	\pgfresetboundingbox
	\path [use as bounding box] (-5,-0.5) rectangle (10.5,4);
\end{tikzpicture}
\caption{\label{proper geodesic rays}\small{The proper geodesic rays $\gamma^0$ and $\gamma^x$.}}
\end{figure}

Considering the UIHPQ as $\UIHPQ=\Phi(\B_\infty^\pm)$, we shall call $\gamma^0$ the proper geodesic ray described within the above proof, which we obtain by taking the leftmost corner around the real vertex in $\B_\infty^\pm$ which corresponds to the root vertex $\rho$ of $\UIHPQ$ and iterating the successor function, thereby generating a geodesic ray along which labels decrease strictly at each step. Notice that one may do the same for every real vertex: pick the leftmost corner adjacent to a real vertex $x$ in $\B_\infty^\pm$ and iterate the successor function, building an infinite path $\gamma^x$ in $\UIHPQ$ along which labels decrease strictly at each step, which is therefore a proper geodesic ray (issued from $x$).

Also notice that for any real vertex $x$ the two proper geodesic rays $\gamma^0=(c^0_i)_{i\geq 0}$ and $\gamma^x=(c^x_i)_{i\geq 0}$ (which is issued from $x$) must eventually meet (and, indeed, coincide from a certain point onwards, see Figure~\ref{proper geodesic rays}).

We now prove a simple lemma about a property of confluence to the root that geodesic rays display in the UIHPQ; the statement is presented in terms of $\gamma^0$, but any geodesic ray issued from the root could be used as reference.
\begin{lem}\label{confluence to the root}
For each positive integer $r$ there is a positive integer $R$ such that, if $z$ is a vertex of $\UIHPQ$ for which $\mathrm{d_{gr}}(z,\rho)>R$ (where $\rho$ is the root vertex), then there is a (finite) geodesic $\gamma^{0z}$ joining $\rho$ to $z$ that coincides with $\gamma^0$ up to the $r$-th step (hence in particular $\gamma^{0z}(i)=\gamma^0(i)$ for $i=0,\ldots,r$).
\end{lem}

\begin{proof}
Consider $k=\min\{i\geq r \mid \gamma^0(i)\in I\}$, where $I$ is the set of vertices at which the leftmost and rightmost geodesic rays in $\UIHPQ$ meet (which is infinite by Corollary~\ref{infinite cut points}, so that $k$ is well defined). Consider then the (finite) set of all vertices at distance $k$ from the root vertex, with the exception of vertex $\gamma^0(k)$; since all geodesic rays issued from $\rho$ go through $\gamma^0(k)$, no vertex in the set belongs to a geodesic ray issued from $\rho$. By local finiteness, this implies that there is $R\geq 0$ such that no finite geodesic issued from $\rho$ having length more than $R$ can go through any of the vertices in the set. Hence, if a vertex $z$ is such that $\mathrm{d_{gr}}(z,\rho)>R$, then any geodesic $\gamma$ joining $\rho$ to $z$ must go through $\gamma^0(k)$. It is therefore enough to build $\gamma^{0z}$ by following $\gamma^0$ from $\rho$ to $\gamma^0(k)$, then $\gamma$ from $\gamma^0(k)$ to $z$; since $k\geq r$, $\gamma^{0z}$ satisfies the requirement of the lemma.
\end{proof}

From here we can prove a key confluence property of all ``long'' geodesics: given any vertex $x$, geodesics issued from $x$ which reach far enough must eventually intersect $\gamma^0$; more precisely: 
\begin{lem}\label{confluence to infinity}
Given a vertex $x$ in $\UIHPQ=\Phi(\B^\pm_\infty)$ there are positive integers $r$ and $R$ such that for any vertex $z$ for which $\mathrm{d_{gr}}(z,\rho)>R$, all geodesics joining $x$ to $z$ meet $\gamma^0$ before its $r$-th step.
\end{lem}

\begin{proof}
We start by showing that all geodesic rays issued from $x$ meet $\gamma^0$. Let $\gamma$ be any such geodesic ray; since the sequence $\mathrm{d_{gr}}(\gamma(n),\rho)-n$ is decreasing in $n$ (for each $n$ one has $\mathrm{d_{gr}}(\gamma(n+1),\rho)-\mathrm{d_{gr}}(\gamma(n),\rho)\leq 1$), it must be constant from some $n_0$ onwards. Thus replacing the part of $\gamma$ that joins $x$ to $\gamma(n_0)$ with a finite geodesic joining $\rho$ to $\gamma(n_0)$ yields a geodesic ray issued from the root. By Corollary~\ref{infinite cut points}, this implies that $\gamma$ meets $\gamma^0$ an infinite number of times.

Suppose now by contradiction that for all positive integers $r$ and $n$ there is a vertex $z_{n,r}$ such that $\mathrm{d_{gr}}(z_{n,r},\rho)>n$ and that there is a geodesic $\gamma_{n,r}$ joining $x$ to $z_{n,r}$ that does not visit any vertex in the sequence $\gamma^0(0),\ldots,\gamma^0(r)$. Then for each $k>0$ we can build a geodesic ray $\gamma_k$ issued from $x$ that does not go through $\gamma^0(0),\ldots,\gamma^0(k)$ (since for every $i$ there is, by local finiteness, an infinite subsequence of $(\gamma_{n,k})_{n>0}$ such that all of its members coincide up to the $i$-th step). Similarly, for each $r$ there is an infinite subsequence of $({\gamma}_k)_{k>0}$ consisting of geodesic rays that coincide up to the $r$-th step, hence a geodesic ray issued from $x$ that never meets $\gamma^0$. The lemma is thus proven by contradiction. 
\end{proof}

The lemmas presented above are enough to establish Theorem~\ref{labels as limit}:

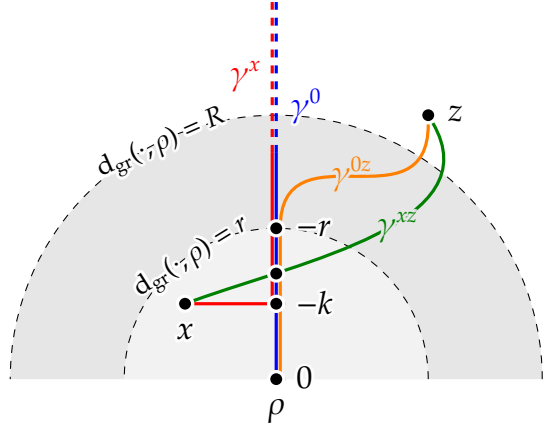
\begin{figure}
\centering
\begin{tikzpicture}
\contourlength{1px};
\tikzstyle{real}=[fill=black, draw=white, very thick, circle, inner sep=2pt]
\node[real, label=below:$\rho$, label=right:\contour{white}{$0$}] (rho) at (0,0) {};
\node[real, label=below:\contour{white}{$x$}] (x) at (-1.2,1) {};
\node[real, label=right:\contour{white}{$-k$}] (y) at (0,1) {};
\node[real, label=right:\contour{white}{$-r$}] (y') at (0,2) {};
\node[real, label=right:\contour{white}{$z$}] (z) at (2,3.5) {};
\node[real] (t) at (0,1.4) {};
\begin{pgfonlayer}{edgelayer}
\filldraw[dashed, draw=black, fill=gray!20] (3.5,0) arc (0:180:3.5);
\fill[dashed, draw=black, fill=gray!10] (2,0) arc (0:180:2);
\node[rotate=24] at (-1.5, 3.1) {\contour{white}{\footnotesize$\mathrm{d_{gr}}(\cdot,\rho)=R$}};
\node[rotate=32] at (-1.1, 1.6) {\contour{white}{\footnotesize$\mathrm{d_{gr}}(\cdot,\rho)=r$}};
\draw[very thick, orange] (rho) to ([xshift=1.5pt]rho) to ([xshift=1.5pt]y);
\draw[very thick, orange] (y) to ([xshift=1.5pt]y) to ([xshift=1.5pt]y');
\draw[very thick, orange, out=90,in=-90,looseness=1.5] (y')+(1.5pt,0) to (z);
\node at (1, 2.7) {\contour{gray!20}{\color{orange}$\gamma^{0z}$}};
\draw[very thick, blue] (rho) to (y);
\draw[very thick, blue] (y) to ([xshift=0pt]y) to ([xshift=0pt]y') to ([xshift=0pt]0,3);
\draw[very thick, blue, dashed] ([xshift=0pt]0,3) to ([xshift=0pt]0,5);
\node at (0.4, 3.6) {\contour{white}{\color{blue}$\gamma^{0}$}};
\draw[very thick,red] (x) to ([xshift=-1.5pt]y) to ([xshift=-1.5pt]y') to ([xshift=-1.5pt]0,3);
\draw[very thick, red, dashed] [xshift=-1.5pt](0,3) to (0,5);
\node at (-0.4, 4) {\contour{white}{\color{red}$\gamma^{x}$}};
\draw[very thick,green!50!black, out=20, in=-60, looseness=1] (x) to (z);
\node at (1.6, 2) {\contour{gray!20}{\color{green!50!black}$\gamma^{xz}$}};
\end{pgfonlayer}
\end{tikzpicture}
\caption{\label{proof of theorem labels as limit}\small{The geodesics and geodesic rays from the proof of Theorem~\ref{labels as limit}: note that $\gamma^x$ and $\gamma^0$ are proper, and that $\gamma^{xz}$ is a generic geodesic joining $x$ to $z$.}}
\end{figure}

\begin{proof}[Proof of Theorem~\ref{labels as limit}]
Without loss of generality, we can prove the required equality for pairs involving the root vertex $\rho$ of the UIHPQ; the statement for a pair of vertices $(x,y)$ trivially follows by applying \eqref{eq:labels as limit} to $(x,\rho)$ and $(\rho, y)$.

Consider a vertex $x$ of $\UIHPQ$; build the geodesic rays $\gamma^0$ and $\gamma^x$, and suppose they coincide from $\gamma^0(k)$ onwards for some $k$. By Lemma~\ref{confluence to infinity} there are positive integers $r$ and $R_1$ such that, if $\mathrm{d_{gr}}(z,\rho)>R_1$, then any geodesic joining $x$ to $z$ must meet $\gamma^0$ before its $r$-th step (and, clearly, we may assume $r\geq k$). On the other hand, by Lemma~\ref{confluence to the root} one can find $R\geq R_1$ such that, if $\mathrm{d_{gr}}(z,\rho)>R$, then there is a geodesic $\gamma^{0z}$ joining $\rho$ to $z$ which coincides with $\gamma^0$ up to $\gamma^0(r)$.

Take $z$ such that $\mathrm{d_{gr}}(z,\rho)>R$ and consider a geodesic $\gamma^{xz}$ joining $x$ to $z$. Notice that there is a geodesic joining $x$ to $z$ that goes through $\gamma^0(r)$: suppose that $\gamma^{xz}$ meets $\gamma^0$ at $\gamma^0(i)$, with $i<r$; then we may follow the geodesic $\gamma^{xz}$ from $x$ up to $\gamma^0(i)$, and then follow the geodesic $\gamma^{0z}$ between $\gamma^0(i)$ and $z$, thus obtaining a geodesic path from $x$ to $z$ that goes through $\gamma^0(r)$ (see Figure~\ref{proof of theorem labels as limit}).

As a consequence, we can compute the distance between $x$ and $z$ as $\mathrm{d_{gr}}(z,x)=\mathrm{d_{gr}}(z,\gamma^0(r))+\mathrm{d_{gr}}(\gamma^0(r),x)$, while we have $\mathrm{d_{gr}}(z,\rho)=\mathrm{d_{gr}}(z,\gamma^0(r))+r$ (since the geodesic $\gamma^{0z}$ joins $\rho$ to $z$ and goes through $\gamma^0(r)$). Subtracting the two expressions yields
$$\mathrm{d_{gr}}(z,x)-\mathrm{d_{gr}}(z,\rho)=\mathrm{d_{gr}}(\gamma^0(r),x)-r.$$
Now, since $\gamma^0$ and $\gamma^x$ are proper and they both go through $\gamma^0(r)$, we must have $l(\gamma^0(r))=-r=l(x)-\mathrm{d_{gr}}(\gamma^0(r),x)$, hence 
$$\mathrm{d_{gr}}(z,x)-\mathrm{d_{gr}}(z,0)=l(x)$$
as wanted.
\end{proof}

\section{Scaling limits}\label{scaling}
In this section we compute the scaling limit of the infinite positive treed bridge $ \mathcal{B}_{\infty}$ introduced in Proposition~\ref{limitinp}, which encodes the UIHPQ. The definitions, results and proofs in this section are very close to those of~\cite{LGM10}; we thus use the same notation and presentation so that the reader may easily compare the relevant sections in the two papers. From the scaling limit of $\B_\infty$ we derive the limiting law for the the volume of balls, suitably rescaled, in the UIHPQ. We then extend such results to the UIHPQ with a simple boundary, using the pruning operation from~\cite{CMboundary}; they shall be used in a companion paper to study the model of a self-avoiding walk on the UIPQ. Finally, as a byproduct of the computation of these scaling limits we construct a new random locally compact metric space, the \emph{Brownian half-plane}, which -- we believe -- describes the scaling limit of the UIHPQ in the local Gromov--Hausdorff topology.

\subsection{Scaling limits for $ \mathcal{B}_{\infty}$}\label{scaling limits}

Let $ \B_{\infty}=((X_{i})_{i \in \mathbb{Z}};T)$ be the positive infinite treed bridge from Proposition~\ref{limitinp}. Recall that the bridge $X$ is given by the concatenation of two independent walks with step distribution $ \mathbf{p}$, issued from 0, and that in the construction of $ \UIHPQ$ as $\Phi( \mathcal{B}_{\infty})$ the process $X$ encodes distances between the root vertex and vertices found along the boundary of $\UIHPQ$.

\begin{proposition}[Scaling limit for distances along the boundary] \label{prop:scalingX} We have the following convergence in distribution, uniformly on every compact subset of $ \mathbb{R}$:
$$ \left( \frac{X_{[nt]}}{ \sqrt{n}} \right)_{t \in \mathbb{R}}  \xrightarrow[n\to\infty]{(d)} \left(Z_{t}\right)_{t \in \mathbb{R}},$$ where $(Z_{t})_{ t \geq 0}$ and $(Z_{-t})_{t \geq 0}$ are two independent Bessel processes of dimension $5$.
\end{proposition}
\proof  This is a direct consequence of a well known result of Lamperti~\cite{Lam61} (see also~\cite[Section~5.2]{BK14} and~\cite{thesis} for details), if one uses the explicit expression for the transition probabilities $ \mathbf{p}$ given by~\eqref{eq:p}.  \endproof 

In order to compute the scaling limit of the full infinite positive treed bridge $ \B_{\infty}$ we encode it by a pair of processes $(C,V)$ defined on $\mathbb{Z}$, built by reading labels along the contour. Heuristically, imagine that a particle follows the contour of the treed bridge $\B_{\infty}$ from left to right, meeting all corners $(c_{i})_{i \in \mathbb{Z}}$ of the treed bridge (adjacent to both real and phantom vertices, in left-to-right order, so that $c_0$ is the corner adjacent to the root vertex of the treed bridge); we define $C(i)$ to be the distance in the treed bridge between $v(c_{i})$ and the root vertex, and $V(i)$ to be the label of $v(c_i)$.

\begin{figure}
\centering\begin{tikzpicture}

\tikzstyle{real}=[inner sep=1.5pt, draw=white,fill=black, thick, circle]
\tikzstyle{phantom}=[inner sep=1.5pt, draw=black,fill=white, thick, circle]

	\begin{pgfonlayer}{nodelayer}
	\draw[dashed] (0,-2.1) to (0,2.5);
		\node [style=real, label=below:\contour{white}{0}] (0) at (0, 0) {};
		\node [style=phantom, label=below:1] (1) at (1, 0) {};
		\node [style=real, label=below:2] (2) at (2, 0) {};
		\node [style=phantom, label=below:1] (3) at (3, 0) {};
		\node [style=real,, label=below:2] (4) at (4, 0) {};
		\node [style=phantom, label=below:1] (5) at (5, 0) {};
		\node [style=real, label=below:1] (6) at (-1, 0) {};
		\node [style=real, label=below:2] (7) at (-2, 0) {};
		\node [style=real, label=below:3] (8) at (-3, 0) {};
		\node [style=phantom, label=below:2] (9) at (-4, 0) {};
		\node [style=real, label=below:3] (10) at (-5, 0) {};
		\node (left) at (-6,0) {$\ldots$};
		\node (right) at (6,0) {$\ldots$};
		\node [style=real, label=right:\contour{white}{\footnotesize 2}] (11) at (2, 0.75) {};
		\node [style=real, label=left:\contour{white}{\footnotesize 3}] (12) at (1.5, 1.5) {};
		\node [style=real, label=above:\contour{white}{\footnotesize 1}] (13) at (2.5, 1.5) {};
		\node [style=real, label=above:\contour{white}{\footnotesize 2}] (14) at (1.25, 2.25) {};
		\node [style=real, label=above:\contour{white}{\footnotesize 4}] (15) at (1.75, 2.25) {};
		\node [style=real, label=above:\contour{white}{\footnotesize 2}] (16) at (2, 1.5) {};
		\node [style=real, label=above:\contour{white}{\footnotesize 2}] (17) at (3.75, 0.75) {};
		\node [style=real, label=above:\contour{white}{\footnotesize 3}] (18) at (4.25, 0.75) {};
		\node [style=real, label=left:\contour{white}{\footnotesize 2}] (19) at (-2, 0.75) {};
		\node [style=real, label=above:\contour{white}{\footnotesize 2}] (20) at (-2.5, 1.5) {};
		\node [style=real, label=above:\contour{white}{\footnotesize 3}] (21) at (-1.5, 1.5) {};
		\node [style=real, label=right:\contour{white}{\footnotesize 4}] (22) at (-5, 0.75) {};
		\node [style=real,, label=right:\contour{white}{\footnotesize 4}] (23) at (-5, 1.5) {};
		\node [style=real,, label=right:\contour{white}{\footnotesize 5}] (24) at (-5, 2.25) {};
	\end{pgfonlayer}
	\begin{pgfonlayer}{edgelayer}
	\draw[red!90,line cap=round,line width=10pt] (0.center) to (1.center) to (2.center) to (11.center)--(12.center)--(14.center);
	\draw[red!90,line cap=round,line width=10pt] (12.center) to (15.center);
	\draw[red!90,line cap=round,line width=10pt] (11.center) to (13.center);
	\draw[red!90,line cap=round,line width=10pt] (11.center) to (16.center);
	\draw[red!90,line width=1pt, ->] ([yshift=4.5pt]2.center) to ([yshift=4.5pt, xshift=5pt]3.center);
	\draw[orange,line cap=round,line width=10pt] (7.center) to (19.center);
	\draw[orange,line cap=round,line width=10pt] (20.center) to (19.center) to (21.center);
	\draw[orange,line width=1pt, ->] ([yshift=4.5pt]0.center) to ([yshift=4.5pt, xshift=5pt]10.center);
	\draw[white,line cap=round,line width=8pt] (0.center) to (1.center) to (2.center);
	\draw[white,line cap=round,line width=8pt] (0.center) to (9.center);
	\draw[red!10,line cap=round,line width=8pt] (11.center) to (13.center);
	\draw[red!10,line cap=round,line width=8pt] (11.center) to (16.center);
	\draw[white,line cap=round,line width=8pt] (2.center) to (3.west);
	\draw[white,line cap=round,line width=8pt] (-3,-0.1) to (5,-0.1);
	\draw[red!10,line cap=round,line width=8pt] (2.center) to (11.center)--(12.center)--(14.center);
	\draw[red!10,line cap=round,line width=8pt] (12.center) to (15.center);
	\draw[orange!10,line cap=round,line width=8pt] (7.center) to (19.center);
	\draw[orange!10,line cap=round,line width=8pt] (20.center) to (19.center) to (21.center);
	\fill[white] (0.2,-0,1) rectangle (0.1,0.4);
		\draw   (7) to (19);
		\draw   (19) to (21);
		\draw   (19) to (20);
		\draw   (2) to (11);
		\draw   (11) to (13);
		\draw   (11) to (16);
		\draw   (12) to (11);
		\draw   (14) to (12);
		\draw   (12) to (15);
		\draw   (10) to (22);
		\draw   (22) to (23);
		\draw   (23) to (24);
		\draw   (17) to (4);
		\draw   (4) to (18);
		\draw   (0) to (1);
		\draw   (1) to (2);
		\draw   (2) to (3);
		\draw   (3) to (4);
		\draw   (4) to (5) to (right);
		\draw   (0) to (6);
		\draw   (6) to (7);
		\draw   (7) to (8);
		\draw   (8) to (9);
		\draw   (9) to (10) to (left);
	\end{pgfonlayer}
\begin{scope}[shift={(0,-2.1)},scale=0.3]
\draw[help lines,step=1,nearly transparent, path fading=north] (-15.5,0) grid (15.5,6.5);
\draw[black,thick,->] (-20,0)--(20,0);
\node at (15,4) {\color{red}{$C(n)$}};
\node at (17,1) {\color{blue}{$V(n)$}};
\fill[red!10, nearly transparent] (2,0)--(2,2)--(3,3)--(4,4)--(5,5)--(6,4)--(7,5)--(8,4)--(9,3)--(10,4)--(11,3)--(12,4)--(13,3)--(14,2)--(14,0);
\fill[orange!10, nearly transparent] (-2,0)--(-2,2)--(-3,3)--(-4,4)--(-5,3)--(-6,4)--(-7,3)--(-8,2)--(-8,0);
\draw[very thick, black] (5,5)--(5,4) (7,5)--(7,4) (6,4)--(6,3) (10,4)--(10,3) (12,4)--(12,3) (9,3)--(9,2);
\draw[thick,red,->] (0,0)--(1,1)--(2,2)--(3,3)--(4,4)--(5,5)--(6,4)--(7,5)--(8,4)--(9,3)--(10,4)--(11,3)--(12,4)--(13,3)--(14,2)--(15,3);
\draw[thick,orange,->] (0,0)--(-1,1)--(-2,2)--(-3,3)--(-4,4)--(-5,3)--(-6,4)--(-7,3)--(-8,2)--(-9,3)--(-10,4)--(-11,5);
\draw[thick,blue,->] (0,0)--(1,1)--(2,2)--(3,2)--(4,3)--(5,2)--(6,3)--(7,4)--(8,3)--(9,2)--(10,2)--(11,2)--(12,1)--(13,2)--(14,2)--(15,1);
\draw[thick,blue,->] (0,0)--(-1,1)--(-2,2)--(-3,2)--(-4,3)--(-5,2)--(-6,2)--(-7,2)--(-8,2)--(-9,3)--(-10,2)--(-11,3);
\draw[thick, black, dashed] (5,4)--(7,4) (6,3)--(12,3);
\draw[very thick,red] (2,0)--(14,0);
\draw[very thick,orange] (-2,0)--(-8,0);
\draw node at (8,-1) {\footnotesize{$T(2)$}};
\draw node at (-5,-1) {\footnotesize{$T(-2)$}};
\end{scope}
\end{tikzpicture}
\caption{\small\label{fig:contour and label processes}The processes $C$ and $V$ as produced by a particle visiting the treed bridge along its contour. Notice that $V$ may remain constant on intervals corresponding to trees $T(i)$, and may only increase or decrease when the particle is moving along the bridge; intervals corresponding to trees may be retrieved from the values of $C$, since the last time the $i$-th bridge vertex is visited corresponds to the point where $C$ takes the value $i$ for the last time.}
\end{figure}
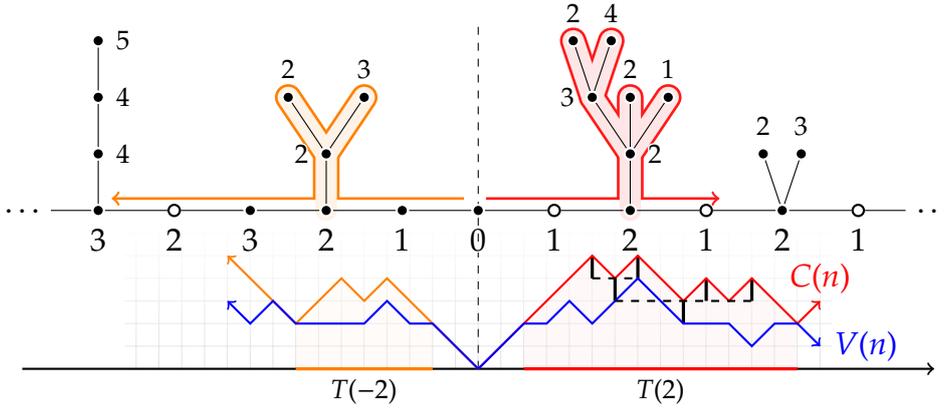

We now proceed to describe the scaling limit for the pair $(C,V)$. In order to introduce the limiting processes we shall need some notation: we refer the reader to~\cite[Section 2.4]{LGM10} for details. In particular we need $ \mathbb{N}_{x}$, the excursion measure of the Brownian snake started at $x$ and driven by an excursion of the Brownian motion under the standard It\^o excursion measure. For our purposes, we can see this object as a sigma-finite measure on the space $\Omega$ of all pairs of continuous paths $$\omega=( \mathbf{e}_{t}, {\widehat{ \mathbf{w}}}_{t})_{0 \leq t \leq \sigma}$$ where $\sigma= \sigma(\omega)$ is the lifetime of the path $\omega$ and is ``distributed'' according to the measure $ \mathrm{d\sigma}/(2\sqrt{2 \pi \sigma^3})$. Roughly speaking, conditionally on $\sigma$, the process $( \mathbf{e}_{t})_{0 \leq t \leq \sigma}$ is a Brownian excursion of length $\sigma$ and ${\widehat{ \mathbf{w}}}$ is the head of the Brownian snake started at $x$ and driven by $ \mathbf{e}$. We denote by $ \mathcal{R}(\omega)$ the range of the path $ {\widehat{ \mathbf{w}}}$.

Recall from Proposition~\ref{prop:scalingX} that $Z$ is a two-sided Bessel process of dimension $5$.   Conditionally given $Z$, let 
$  \mathcal{N} = \sum_{i \in I} \delta_{(r_{i}, \omega_{i}) } $ be a Poisson process on $ \mathbb{R} \times  \Omega$ with intensity 
 \begin{eqnarray} \label{eq:intensity}  \mathrm{d}t \,\mathbb{N}_{0}( \mathrm{d}\omega)\mathbf{1}_{\min  \mathcal{R}(\omega) > - \sqrt{ \frac{3}{2}}Z_{t}}.  \end{eqnarray}
Note that, in contrast with~\cite[Eq. (7)]{LGM10}, there is no multiplicative factor of $2$ in the above display, and that the Bessel process has been multiplied by a factor of $ \sqrt{3/2}$. From this Poisson process we build two continuous functions $\zeta$ and $\widehat{W}$ as follows (see~\cite[Section 3.1]{LGM10} for details). For each $i \in I$ we set 
$$ \sigma_{i} = \sigma( \omega_{i}), \qquad \omega_{i} =(  \mathbf{e}^{i}_{s}, \widehat{ \mathbf{w}}_{s}^i)_{0 \leq s \leq \sigma_{i}}.$$ Then the two functions $\zeta$ and $\widehat{W}$ needed to describe the scaling limit of $(C,V)$ are obtained by concatenating (in the order given by the $r_{i}$) the functions 
$$\left( r_{i} +  \mathbf{e}_{s}^i, \sqrt{ \frac{3}{2}} Z_{r_{i}} +  \widehat{ \mathbf{w}}_{s}^i\right)_{0 \leq s \leq \sigma_{i}}.$$
More formally,  for $u \in \mathbb{R}$ we set 
$$ \tau_{u} = \sum_{i \in I} \mathbf{1}_{\{ r_{i} \in [0,u]\}} \cdot \sigma_{i};$$
then clearly $u \mapsto \tau_{u}$ is right-continuous and increasing on $ \mathbb{R}_{+}$, left-continuous and decreasing on $ \mathbb{R}_{-}$. For all $s \in \mathbb{R}$ there is a unique $u$ with the same sign as $s$ such that $|s|$ is between $\tau_{u-}$ and $\tau_{u+}$, and we have exactly one of the following:
\begin{itemize}
\item there is a unique $i \in I$ such that $u=r_{i}$; in this case, we set 
 \begin{eqnarray*} 
 \zeta_{s} &=& |u| +  \mathbf{e}_{||s|-\tau_{u-}|}^i,\\
  \widehat{W}_{s} &=& \sqrt{ \frac{3}{2}} Z_{u} +  \widehat{ \mathbf{w}}^i_{||s|-\tau_{u-}|}; \end{eqnarray*}
  \item there is no such $i$ (this happens if $\tau_{u-}=s=\tau_{u+}$); in this case, we set
$$ \zeta_{s}=|u| \qquad \mbox{ and } \qquad \widehat{W}_{s} = \sqrt{ \frac{3}{2}} Z_{u}.$$
\end{itemize}
It is easy to see that both $\zeta$ and $\widehat{W}$ are continuous processes over $\mathbb{R}$; an easy adaptation of the argument yielding~\cite[Eq. (8)]{LGM10} (replacing the Bessel process of dimension $9$ by $ \sqrt{3/2}$ times a Bessel process of dimension $5$) shows that $\widehat{W}$ is transient in the sense that $\lim_{s \to \pm \infty} \widehat{W}_{s} = + \infty$ almost surely.

 \begin{theorem}[Scaling limit for the contour functions]  \label{thm:scalinglimit} We have the following convergence in distribution uniformly on every compact subset of $\mathbb{R}$:
 
 $$ \left( \frac{1}{n} C(n^2s),  \sqrt{\frac{3}{2n}} V(n^2s) \right)_{s \in \mathbb{R}} \xrightarrow[n\to\infty]{(d)} \big(  \zeta_{s}, \widehat{W}_{s}\big)_{s \in \mathbb{R}}.$$
 \end{theorem}
 
 \proof A proof of this result can be obtained with the methods developed in~\cite{LGM10} to prove Theorem~5; since all that is needed is a series of marginal adjustments, we shall not repeat here the rather long and technical argument in full, but simply highlight the differences between our model and that of~\cite{LGM10}. In~\cite{LGM10}, finite trees are grafted on both sides of a semi-infinite spine; they are, exactly as in our case, conditionally independent given the labels $(Y_i)_{i\geq 0}$, and distributed according to $\rho^+_{Y_i}$. The fact that in our case trees are grafted on only one side of the infinite bridge simplifies the situation slightly (we only have to account for one process $\widehat{W}$ coding the labels instead of two correlated processes). As in~\cite{LGM10}, the fact that the labels in the contour of random positive labeled trees evolve roughly like random walks with increments uniformly chosen in $\{-1,0,+1\}$ justifies the rescaling of the label process by $\sqrt{2n/3}$ in space and $n$ in time, see~\cite[Proposition~3]{LGM10}. Contrary to what happens in~\cite{LGM10}, however, labels along the bridge itself can vary only by $+1$ or $-1$; because of this, rescaling (by $\sqrt{\frac{2n}{3}}$ in space and $n$ in time) and taking the limit results in $\sqrt{3/2}$ times a standard Bessel process of dimension $5$. This modification has no bearing at all on the arguments of the proof, since the rough estimates needed are also valid for multiples of Bessel $5$ processes.  Finally, finite trees in~\cite{LGM10} are grafted at every vertex of the spine; here, trees are only grafted on down-steps of our infinite bridge; this has been taken into account by having the intensity measure in \eqref{eq:intensity} lose a factor of 2 with respect to the analogue~\cite[Eq. (7)]{LGM10}. More precisely, if one considers the counting measure on down-steps then it is easy to see that
 $$ \frac{1}{n}\sum_{i \in \DS(X)} \delta_{i/n}  \xrightarrow[n\to\infty]{a.s.} \frac{1}{2} \cdot \mathrm{Leb},$$ where $ \mathrm{Leb}$ is the Lebesgue measure. With these remarks at hand, it is easy to adapt the proof of~\cite[Theorem 5]{LGM10} to obtain the desired result. \endproof

\subsection{Volume estimates}
We can now use Theorem~\ref{thm:scalinglimit} to deduce the limiting law of the profile of distances in the UIHPQ exactly as in~\cite[Section 4]{LGM10}. Here we shall focus on a simpler quantity than the full distance profile, namely the volume of the ball of radius $r$ in the UIHPQ. 
\begin{proposition} \label{prop:limitlawball} Let $\#[ \UIHPQ]_{n}$ denote the volume of the ball of radius $n$ in the UIHPQ (that is, the number of its vertices). Then we have the following convergence in distribution:
 \begin{eqnarray*} \frac{1}{n^4} \# [ \UIHPQ]_{n} & \xrightarrow[n\to\infty]{(d)} &  \frac{9}{8}\int_{ \mathbb{R}} \mathrm{d}s\, \mathbf{1}_{ \widehat{W}_{s} \leq 1}.  \end{eqnarray*}
\end{proposition}
Again we do not provide a full proof of the proposition, since all that is needed is a simple adaptation of~\cite[Proof of Theorem~6]{LGM10} using Theorem~\ref{thm:scalinglimit}. We shall, however, sketch the chain of approximate equalities in order to highlight the role of the scaling constants. 

To begin with, note that for large $r \geq 0$ we have 
  \begin{eqnarray*}  \frac{1}{r^4} \#[ \UIHPQ]_{r}  \approx  \frac{1}{2r^4}\int_{ \mathbb{R}} \mathbf{1}_{ V(s) \leq r} \mathrm{d}s.  \end{eqnarray*}
The reason for the factor $ \frac{1}{2}$ is that every edge of each tree is visited twice by the contour process, and that there are roughly as many real vertices in $\B_\infty$ as edges visited by the contour process on large scales. Next we use Theorem~\ref{thm:scalinglimit} to argue that
 \begin{eqnarray*}\frac{1}{2r^4}\int_{ \mathbb{R}} \mathbf{1}_{ V(s) \leq r} \mathrm{d}s & =& \frac{1}{2} \int_{ \mathbb{R}} \mathbf{1}_{V(r^4s) \leq r} \mathrm{ds}\\
 &=& \frac{1}{2} \int_{ \mathbb{R}} \mathbf{1}_{  \sqrt{\frac{3}{2r^2}}V(r^4s) \leq   \sqrt{3/2}} \mathrm{ds}\\
 & \overset{(d)}{\underset{Thm.~\ref{thm:scalinglimit}}{\approx}}& \frac{1}{2} \int_{ \mathbb{R}} \mathbf{1}_{ \widehat{W}_{s} \leq   \sqrt{3/2}} \mathrm{ds}\\
  & \overset{(d)}{\underset{ \mathrm{scaling}}{=}}& \frac{1}{2}  \left(  \sqrt{\frac{3}{2}} \right)^4	 \int_{ \mathbb{R}} \mathbf{1}_{ \widehat{W}_{s} \leq  1} \mathrm{ds}
.  \end{eqnarray*}
For later use, let us adapt the calculation of~\cite[Proposition 5]{LGM10} (see~\cite{LGM10erratum} for an amended version) and compute the mean of the limiting law that appears in the proposition. By the Poissonian construction of $( \zeta, \widehat{W})$ and the fact that $Z_{t}$ is a five-dimensional Bessel process we have 
 \begin{eqnarray*} \mathbb{E}\left[ \int_{ \mathbb{R}} \mathrm{d}s \mathbf{1}_{ \widehat{W}_{s} \leq 1} \right]  &=& 2\cdot   \mathbb{E}\left[ \int_{0}^\infty  \mathrm{d}t\  \mathbb{N}_{ \sqrt{3/2}Z_{t}}\left[ \mathbf{1}_{ \mathcal{R} \subset (0,\infty)} \int_{0}^\sigma  \mathrm{d}s \mathbf{1}_{ \widehat{ \mathbf{w}}_{s} \leq  1} \right]\right].  \end{eqnarray*}
By a standard calculation, the expected value of the total local time accumulated at level $x\geq 0$ for a $d$-dimensional Bessel process is $ x\frac{2}{d-2}$. Consequently, the expected total local time at level $x\geq 0$ for $ (\sqrt{3/2} Z_{t})_{t \geq 0} \overset{(d)}{=} (Z_{ 3t/2})_{t \geq 0}$ is equal to $  \frac{4 x}{9}$. The function $x \mapsto  \mathbb{N}_{x}\left[ \mathbf{1}_{ \mathcal{R} \subset (0,\infty)} \int_{0}^\sigma  \mathrm{d}s \mathbf{1}_{ \widehat{ \mathbf{w}}_{s} \leq 1} \right]$ has also been computed in~\cite{LGM10,LGM10erratum}, and  is equal to $\frac{2}{35x^3} \mathbf{1}_{x \geq 1} + ( \frac{x^2}{5}- \frac{x^4}{7}) \mathbf{1}_{x <1}$.  Hence we have 
 \begin{eqnarray} \label{eq:mean1/12} \frac{9}{8}\mathbb{E}\left[ \int_{ \mathbb{R}} \mathrm{d}s \mathbf{1}_{ \widehat{W}_{s} \leq 1} \right] =   \int_{0}^\infty  \mathrm{d}x \ x \left( \frac{2}{35x^3} \mathbf{1}_{x \geq 1} + \left( \frac{x^2}{5}- \frac{x^4}{7}\right) \mathbf{1}_{x <1}\right)	 = \frac{1}{12}.  \end{eqnarray}

 \subsection{The Brownian half-plane}
To conclude this section, we use the previous scaling limit result for the contour functions to define the  \emph{Brownian half-plane}, which should be the scaling limit of the UIHPQ in the local Gromov--Hausdorff sense. The construction is similar to the construction of the Brownian plane~\cite{CLGplane,CLGHull}.

Recall the notation of Section~\ref{scaling limits}: we first build an infinite continuous labelled tree from the functions $\zeta$ and $\widehat{W}$ by grafting trees onto the real line according to the Poisson point measure $ \mathcal{N}$. More precisely, we introduce a pseudo-distance on $\mathbb{R}$, denoted $d_{\zeta}$, defined  for $s,t \in \mathbb{R}$ by $ d_{\zeta}(s,t) = \zeta(s)+ \zeta(t) - 2 \min_{u \in [s \wedge t, s \vee t]} \zeta(u).$ The space $ \mathcal{T}$ obtained by taking the quotient of $ \mathbb{R}$ by the equivalence relation $ d_{\zeta}=0$, endowed with the quotient metric $ d_{\zeta}$, is a locally compact real tree (see~\cite{CLGplane} for a similar construction). Notice that $ \mathcal{T}$ has a unique doubly-infinite geodesic: equivalently, $\mathcal{T}$ has two ends. It follows from the construction of the Brownian snake excursion measure that the labelling $\widehat{W}$ is compatible with the canonical projection $ \pi : \mathbb{R} \to \mathcal{T}$: we keep the notation $\widehat{W}$ for this labelling on $\mathcal{T}$. We then mimic Schaeffer's construction in the continuous setting: for $s,t \in \mathbb{R}$, we first define
$$ D^\circ(s,t) = \widehat{W}_{s}+ \widehat{W}_{t} - 2 \max \left\{ \min_{u \in [s\wedge t, s \vee t]} \widehat{W}_{u} , \min_{u \in [s \vee t, \infty) \cup (-\infty, s\wedge t]} \widehat{W}_{u} \right\}$$ 
and then extend $D^\circ(s,t)$ to $ \mathcal{T} \times \mathcal{T}$ by
setting (for $a,b\in \mathcal{T}$)
$$D^\circ (a,b)= \min\{ D^\circ (s,t): s,t\in \mathbb{R},\; \pi(s)=a, \pi(t)=b\}.$$
Finally, for all $a,b\in\mathcal{T}$, consider
$$D^*(a,b) = \inf_{a_0=a,a_1,\ldots,a_p=b} \sum_{i=1}^p D^\circ(a_{i-1},a_i)$$
where the infimum is taken over all finite sequences $a_0,a_1,\ldots,a_p$ in $\mathcal{T}$ such that $a_0=a$ and
$a_p=b$. It is not hard to check that $D^*$ is a pseudo-distance on $\mathcal{T}$ (see~\cite{CLGplane,LeGallICM,MieStFlour} for similar constructions). We once more take the quotient of the space $ \mathcal{T}$ by the equivalence relation $ D^*=0$ to obtain a metric space, endowed with the quotient metric and with a distinguished point (the equivalence class of $\pi(0)$), which we call the \emph{Brownian half-plane}.

\section{The UIHPQ with a simple boundary}\label{simple boundary}
In this section we extend some of the scaling limit results on the geometry of the UIHPQ to the UIHPQ with a simple boundary, from now on denoted $\UIHPQS$, which was first introduced by Angel~\cite{Ang05} as the local limit of quadrangulations with a simple boundary whose size and perimeter both tend to infinity in a suitable way. An interesting feature of maps with a simple boundary is that they can be used to perform surgery operations by glueing their boundaries in various ways; for example, one can fold the boundary of $\UIHPQS$ upon itself (see Figure~\ref{folding}) to obtain a random infinite quadrangulation of the plane endowed with an infinite (one-ended) self-avoiding walk. One may also glue together two independent copies of $\UIHPQS$ and obtain an infinite quadrangulation of the plane, this time endowed with a doubly infinite self-avoiding walk. These models of quadrangulations with self-avoiding walks will be further discussed in a companion paper~\cite{companion}, which will heavily rely on results from this section.

 \begin{figure}[!h]
  \begin{center}
  \includegraphics[width=0.8\textwidth]{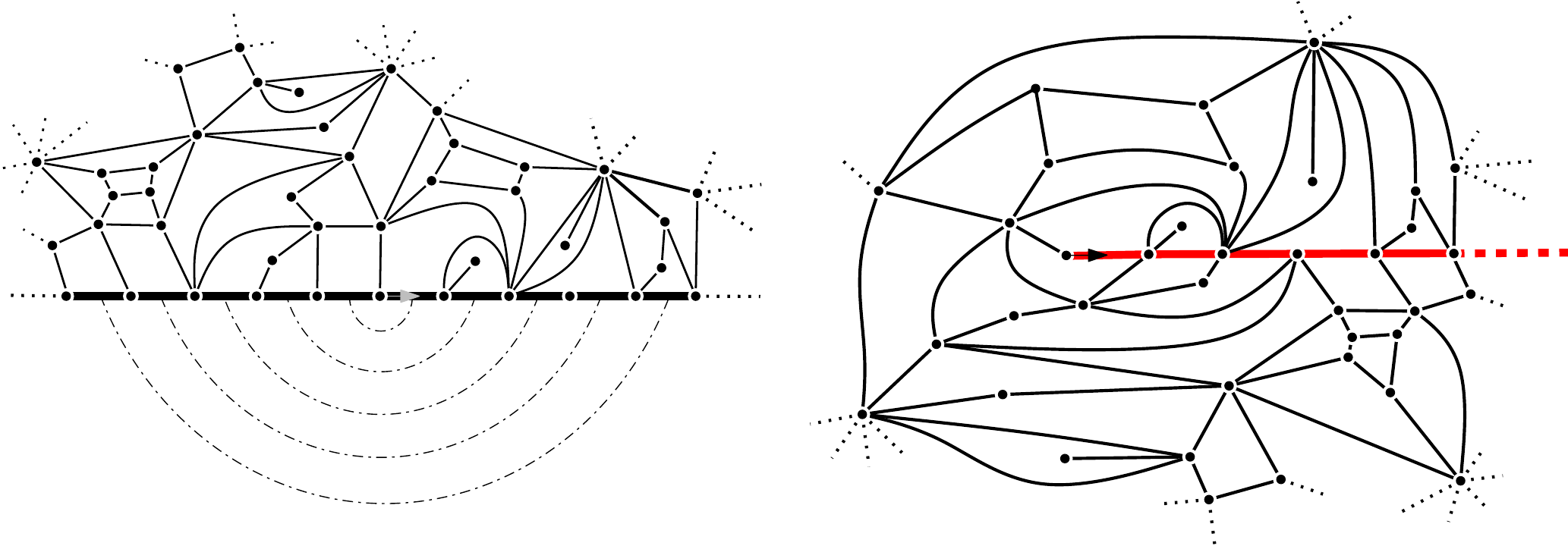}
  \caption{\label{folding}\small{How to fold the UIHPQ with a simple boundary upon itself by glueing together the two halves of its boundary.}}
  \end{center}
  \end{figure}
  
As stated in~\cite[Section 5]{CMboundary}, $\UIHPQS$ can also be obtained via a pruning procedure from the UIHPQ $\UIHPQ$. The pruning consists in eliminating all of the finite quadrangulations hanging from pinch points of the boundary of $\UIHPQ$ (see Figure~\ref{fig:pruning}). Starting with the root vertex, one can follow the contour of the boundary of $\UIHPQ$, thus obtaining a sequence of vertices $(x_i)_{i\in\mathbb{Z}}$, where $x_0$ is the root vertex and vertices may appear more than once; for each vertex $y$ on the boundary of $\UIHPQ$, consider $\mathrm{left}(y)=\min\{i\mid x_i=y\}$ and $\mathrm{right}(y)=\max\{i\mid x_i=y\}$; we can build a subsequence $(v_i)_{i\in\mathbb{Z}}$ of $(x_i)_{i\in\mathbb{Z}}$ by erasing all elements $x_i$ such that there exists a vertex $y$ for which $\mathrm{left}(y)\leq i<\mathrm{right}(y)$ and setting $v_0$ to correspond to the first vertex with non-negative index in the original sequence not to be erased. The sequence $(v_i)_{i\in\mathbb{Z}}$ is in fact a simple doubly infinite path; the infinite (random) quadrangulation lying above it, rooted in the edge $v_0v_1$ (oriented away from $v_0$), is $\UIHPQS$. The complement of $\UIHPQS$ in $\UIHPQ$ is a sequence of (almost surely) finite quadrangulations with a generic boundary hanging from the vertices $(v_i)_{i\in\mathbb{Z}}$, some of which may consist of a single vertex; we name such (random) quadrangulations $(\q_i)_{i\in\mathbb{Z}}$, where $\q_i$ includes vertex $v_i$ (so that $\q_0$ includes the root vertex of $\UIHPQ$, and possibly the entire root edge). By~\cite[Proposition 6]{CMboundary}, the $\q_i$'s are independent and distributed according to the so-called free Boltzmann distribution on quadrangulations with a general boundary, with the exception of $\q_0$, which has an additional bias (see~\cite{CMboundary} for details). 

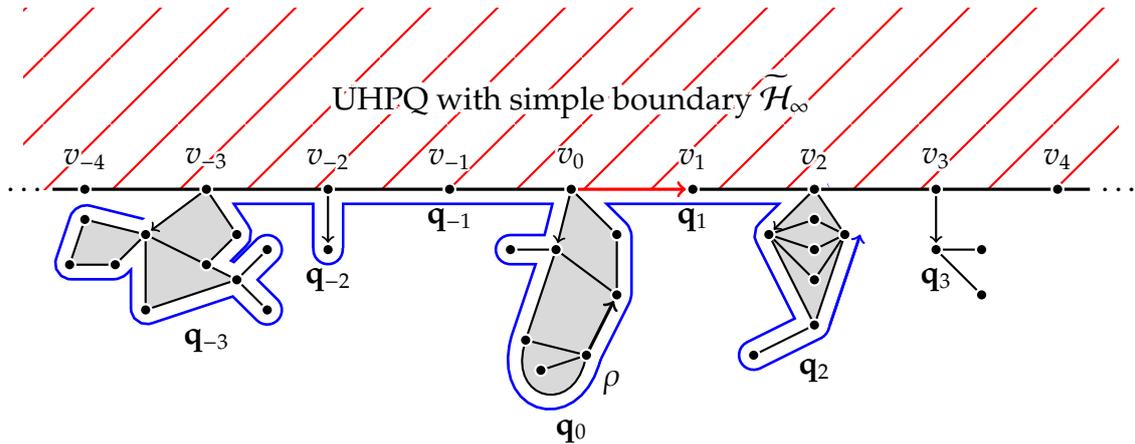
\begin{figure}[width=0.8\textwidth]
\centering\begin{tikzpicture}[scale=0.8]
\tikzstyle{real}=[inner sep=1.5pt, draw=white,fill=black, thick, circle]
\tikzstyle{bla}=[thick]
\tikzstyle{geoleft}=[very thick]
\tikzstyle{root}=[->, very thick]
    
	\begin{pgfonlayer}{nodelayer}
	\contourlength{1pt}
	\node at (-6,-2.5) {$\q_{-3}$};
	\node at (-4,-1.5) {$\q_{-2}$};
	\node at (-2,-0.5) {$\q_{-1}$};
	\node at (0,-4) {$\q_{0}$};
	\node at (2,-0.5) {$\q_{1}$};
	\node at (4,-3) {$\q_{2}$};
	\node at (6,-1.5) {$\q_{3}$};
		\node [style=real, label=above:\contour{white}{$v_{0}$}] (0) at (0, 0) {};
		\node [style=real,label=above:\contour{white}{$v_{1}$}] (1) at (2, 0) {};
		\node [style=real,label=above:\contour{white}{$v_{2}$}] (2) at (4, 0) {};
		\node [style=real,label=above:\contour{white}{$v_{-1}$}] (3) at (-2, 0) {};
		\node [style=real,label=above:\contour{white}{$v_{-2}$}] (4) at (-4, 0) {};
		\node [style=real,label=above:\contour{white}{$v_{-3}$}] (5) at (-6, 0) {};
		\node [style=real, label=above:\contour{white}{$v_{-4}$}] (6) at (-8, 0) {};
		\node (left) at (-9,0) {$\ldots$};
		\node (right) at (9,0) {$\ldots$};
		\node [style=real,label=above:\contour{white}{$v_{3}$}] (7) at (6, 0) {};
		\node [style=real, label=above:\contour{white}{$v_4$}] (8) at (8, 0) {};
		\node [style=real] (9) at (0.75, -0.75) {};
		\node [fill=blue,style=real] (10) at (-0.25, -1) {};
		\node [fill=blue,style=real] (11) at (0.75, -1.75) {};
		\node [fill=blue,style=real] (12) at (-0.75, -2.5) {};
		\node [fill=blue,style=real, label=below right:$\rho$] (13) at (0.25, -2.75) {};
		\node [style=real] (14) at (-1, -1) {};
		\node [style=real] (15) at (-0.5, -3) {};
		\node [style=real] (16) at (3.25, -0.75) {};
		\node [style=real] (17) at (4.5, -0.75) {};
		\node [style=real] (18) at (4, -0.5) {};
		\node [style=real] (19) at (4, -1) {};
		\node [style=real] (20) at (4, -1.5) {};
		\node [style=real] (21) at (4, -2.25) {};
		\node [style=real] (22) at (3, -2.75) {};
		\node [style=real] (23) at (6, -1) {};
		\node [style=real] (24) at (6.75, -1) {};
		\node [style=real] (25) at (6.75, -1.75) {};
		\node [style=real] (26) at (6, -1) {};
		\node [style=real] (27) at (-4, -1) {};
		\node [style=real] (28) at (-7, -0.75) {};
		\node [style=real] (29) at (-6, -1.25) {};
		\node [style=real] (30) at (-5.5, -0.75) {};
		\node [style=real] (31) at (-5.5, -1.5) {};
		\node [style=real] (32) at (-5, -1) {};
		\node [style=real] (33) at (-5, -2) {};
		\node [style=real] (34) at (-7, -2) {};
		\node [style=real] (35) at (-8, -0.5) {};
		\node [style=real] (36) at (-8.25, -1.25) {};
		\node [style=real] (37) at (-7.5, -1.25) {};
	\end{pgfonlayer}
	\begin{pgfonlayer}{edgelayer}
	\draw[line width=12pt, blue, rounded corners=1pt] (28.center) to (35.center) to (36.center) to (37.center) to (28.center) to (34.center) 
	to (31.center) to (29.center) to (30.center) to (5.center) to (4.center) to (27.center) to (4.center) to 
	(2.center) to (16.center) to (21.center) to (17.center);
	\draw[line width=12pt, blue, rounded corners=1pt, cap=round] (13.center) to (11.center) to (9.center) to (0.center) to (10.center) to (12.center) [bend right=90, looseness=2.6] to (13.center);
	\draw[line width=12pt, blue, rounded corners=1pt, cap=round] (33.center)--(31.center)--(32.center);
	\draw[line width=12pt, blue, rounded corners=1pt, cap=round] (4.center)--(27.center);
	\draw[line width=12pt, blue, rounded corners=1pt, cap=round] (14.center)--(10.center);
	\draw[line width=12pt, blue, rounded corners=1pt, cap=round] (21.center)--(22.center);
	
	\draw[line width=10pt, white, rounded corners=1pt, cap=round] (5.center) to (28.center) to (35.center) to (36.center) to (37.center) to (28.center) to (34.center) 
	
	to (31.center) to (29.center) to (30.center) to (5.center) to (4.center) to (27.center) to (4.center) to 
	(2.center) to (17.center) to (21.center) to (16.center) to (2.center);
	\draw[line width=10pt, white, rounded corners=1pt, cap=round] (13.center) to (11.center) to (9.center) to (0.center) to (10.center) to (12.center) [bend right=90, looseness=2.6] to (13.center);
	\draw[line width=10pt, white, rounded corners=1pt, cap=round] (33.center)--(31.center)--(32.center);
	\draw[line width=10pt, white, rounded corners=1pt, cap=round] (4.center)--(27.center);
	\draw[line width=10pt, white, rounded corners=1pt, cap=round] (14.center)--(10.center);
	\draw[line width=10pt, white, rounded corners=1pt, cap=round] (21.center)--(22.center);
	\draw[line width=10pt, white, rounded corners=1pt, cap=round] ([yshift=3pt]5.center) to ([yshift=3pt]2.center);
	\draw[blue, ->, line width=1pt] ([xshift=7.2pt]21.center) to ([xshift=7.2pt]17.center);
	
		\fill[gray!30] (13.center) to (11.center) to (9.center) to (0.center) to (10.center) to (12.center) [bend right=90, looseness=2.6] to (13.center);
		\fill[gray!30] (2.center) to (17.center) to (21.center) to (16.center) to (2.center);
		\fill[gray!30] (5.center) to (28.center) to (35.center) to (36.center) to (37.center) to (28.center) to (34.center) to (31.center) to (29.center) to (30.center) to (5.center);
		\contourlength{2pt}
		\path[pattern=flexible hatch, pattern color=red, hatch distance=25pt, hatch thickness=0.8pt] (-9,0) rectangle (9,3) node[pos=.5] {\contour{white}{UHPQ with simple boundary $\UIHPQS$}};
		\draw [style=bla, <-] (10) to (0);
		\draw [style=bla] (10) to (11);
		\draw [style=bla] (11) to (9);
		\draw [style=bla] (9) to (0);
		\draw [style=bla] (12) to (10);
		\draw [style=bla, bend right=90, looseness=2.25] (12) to (13);
		\draw [style=root, draw=black] (13) to (11);
		\draw [style=bla] (14) to (10);
		\draw [style=bla] (13) to (15);
		\draw [style=bla] (12) to (13);
		\draw [style=bla, <-] (16) to (2);
		\draw [style=bla] (2) to (17);
		\draw [style=bla] (18) to (16);
		\draw [style=bla] (18) to (17);
		\draw [style=bla] (17) to (19);
		\draw [style=bla] (19) to (16);
		\draw [style=bla] (16) to (20);
		\draw [style=bla] (20) to (17);
		\draw [style=bla] (21) to (17);
		\draw [style=bla] (21) to (16);
		\draw [style=bla] (22) to (21);
		\draw [style=bla, <-] (23) to (7);
		\draw [style=bla] (23) to (24);
		\draw [style=bla] (23) to (25);
		\draw [style=bla, <-] (27) to (4);
		\draw [style=bla, ->] (5) to (28);
		\draw [style=bla] (28) to (29);
		\draw [style=bla] (29) to (30);
		\draw [style=bla] (30) to (5);
		\draw [style=bla] (29) to (31);
		\draw [style=bla] (31) to (32);
		\draw [style=bla] (31) to (33);
		\draw [style=bla] (31) to (34);
		\draw [style=bla] (34) to (28);
		\draw [style=bla] (36) to (35);
		\draw [style=bla] (37) to (36);
		\draw [style=bla] (37) to (28);
		\draw [style=bla] (28) to (35);
		\draw [style=geoleft] (6) to (5);
		\draw [style=geoleft] (5) to (4);
		\draw [style=geoleft] (4) to (3);
		\draw [style=geoleft] (3) to (0);
		\draw [style=root, draw=red] (0) to (1);
		\draw [style=geoleft] (1) to (2);
		\draw [style=geoleft] (2) to (7);
		\draw [style=geoleft] (7) to (8);
		\draw [style=geoleft] (6) to (left);
		\draw [style=geoleft] (8) to (right);
	\end{pgfonlayer}
\end{tikzpicture}
\caption{\label{fig:pruning}\small{The UIHPQ with a simple boundary obtained by a pruning procedure from the UIHPQ with a generic boundary. The finite quadrangulations $\q_i$ are rooted in their first edge to be met when following the left-to-right contour of the boundary of the UIPQ (represented by the blue arrow) so that the root vertex of $\q_i$ is the vertex $v_i$ on the simple boundary of $\UIHPQS$.}}
\end{figure}
 
We denote by $\Q_f$ a free Boltzmann quadrangulation with a general boundary; the probability that $Q_{f}=q$ for a given (rooted) quadrangulation $q\in\sQ_{n,p}$ (with $n$ internal faces and perimeter $2p$) is
 $$ \mathbb{P}(\Q_{f} =q) = \frac{1}{W(1/12,1/8)} \left(\frac{1}{12}\right)^{n} \left(\frac{1}{8}\right)^p,$$ where $W(g,z)$ is the generating function defined in Section~\ref{quadrangulations}, counting cardinalities of the sets $\sQ_{n,p}$ with a weight $g$ per inner face and $ \sqrt{z}$ per edge on the boundary; using the explicit formula~\eqref{W} one can deduce that
 \begin{eqnarray} \label{eq:meanfinite} \mathbb{E}[ \mathrm{Perimeter}(\Q_{f})] = \left.\frac{ 2z\partial_{z}W(g,z)}{W(g,z)}\right |_{\begin{subarray}{c} g=1/12\\z=1/8 \end{subarray}} = 2.\end{eqnarray}
The area (i.e.~the number of inner faces) of $\Q_{f}$ has no first moment since $\partial_{g}W(g,z)$ explodes when $z=1/8$ as $g \to 1/12$. Singularity analysis, however, shows that 
 \begin{eqnarray} \label{eq:sizetail} \mathbb{P}( \mathrm{Area}(\Q_{f}) = n) \sim \frac{[g^n]W(g,1/8)}{W(1/12,1/8)} = C \cdot n^{-7/4}  \end{eqnarray} for some constant $C>0$ (whose value is not relevant for what follows).

\bigskip 
Using the pruning procedure described above we can now deduce a result analogous to Proposition~\ref{prop:scalingX} for the UIHPQ with a simple boundary:
\begin{proposition}
Let $(\widetilde{X}_{k})_{k \in \mathbb{Z}}$ be the process of distances to the root vertex read along the boundary of the UIHPQ with a simple boundary (which is a doubly infinite path, implicitly identified with $ \mathbb{Z}$); then, using the same notation as in Proposition~\ref{prop:scalingX}, we have
$$ \left( \frac{\widetilde{X}_{[nt]}}{ \sqrt{n}} \right)_{t \in \mathbb{R}}  \xrightarrow[n\to\infty]{(d)} \sqrt{3}\left(Z_{t}\right)_{t \in \mathbb{R}}.$$
\end{proposition}
\proof  We use the construction of the UIHPQ with a simple boundary $\UIHPQS$ as the result of the pruning operation on the standard UIHPQ, as described at the beginning of this section; recall that the boundary $(v_i)_{i\in\mathbb{Z}}$ of $\UIHPQS$ was constructed as a subsequence of the contour $(x_i)_{i\in\mathbb{Z}}$ of the full boundary of $\UIHPQ$, so that each $v_i$ is identified with some $x_{n_i}$, its leftmost occurrence in the sequence $(x_i)_{i\in\mathbb{Z}}$. It is now quite clear that $n_{i}-n_{i-1}-1$ is equal to the perimeter of the quadrangulation $\q_i$, so that by the law of large numbers we get
$$ \frac{n_{i}}{i} \xrightarrow[i \to \pm\infty]{a.s.} 1 + \mathbb{E}[ \mathrm{Perimeter}(Q_{f})]\underset{ \eqref{eq:meanfinite}}{=}3.$$
This fact can be combined with Proposition~\ref{prop:scalingX} to obtain the desired result.\endproof

Similarly, Proposition~\ref{prop:limitlawball} can be restated for the UIHPQ with a simple boundary:
\begin{proposition}\label{prop:limitlawball for UIHPQS} Proposition~\ref{prop:limitlawball} still holds if $\UIHPQ$ is replaced by $\UIHPQS$.
\end{proposition}

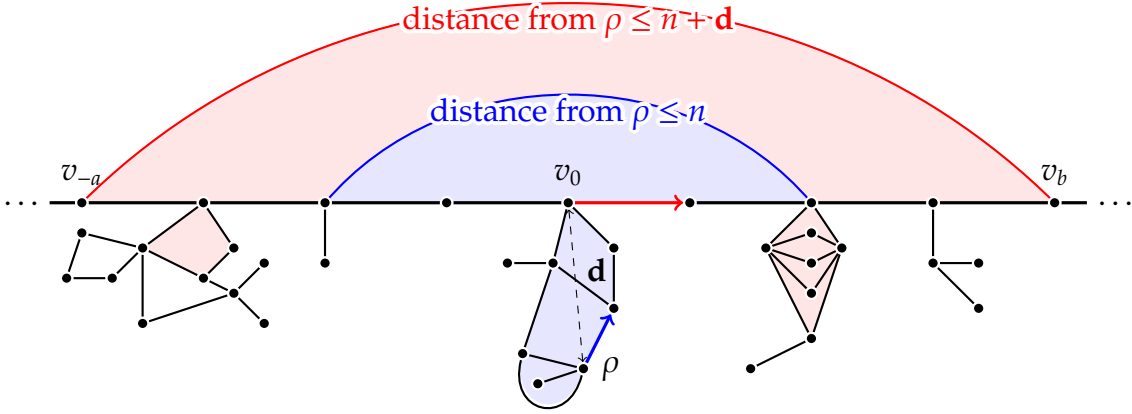
\begin{figure}[width=0.8\textwidth]
\centering\begin{tikzpicture}[scale=0.8]
\tikzstyle{real}=[inner sep=1.5pt, draw=white,fill=black, thick, circle]
\tikzstyle{bla}=[thick]
\tikzstyle{geoleft}=[very thick]
\tikzstyle{root}=[->, very thick]
	\begin{pgfonlayer}{nodelayer}
		\node [style=real, label=above:$v_0$] (0) at (0, 0) {};
		\node [style=real] (1) at (2, 0) {};
		\node [style=real] (2) at (4, 0) {};
		\node [style=real] (3) at (-2, 0) {};
		\node [style=real] (4) at (-4, 0) {};
		\node [style=real] (5) at (-6, 0) {};
		\node [style=real, label=above:$v_{-a}$] (6) at (-8, 0) {};
		\node (left) at (-9,0) {$\ldots$};
		\node (right) at (9,0) {$\ldots$};
		\node [style=real] (7) at (6, 0) {};
		\node [style=real, label=above:$v_b$] (8) at (8, 0) {};
		\node [style=real] (9) at (0.75, -0.75) {};
		\node [fill=blue,style=real] (10) at (-0.25, -1) {};
		\node [fill=blue,style=real] (11) at (0.75, -1.75) {};
		\node [fill=blue,style=real] (12) at (-0.75, -2.5) {};
		\node [fill=blue,style=real, label=right:$\rho$] (13) at (0.25, -2.75) {};
		\node [style=real] (14) at (-1, -1) {};
		\node [style=real] (15) at (-0.5, -3) {};
		\node [style=real] (16) at (3.25, -0.75) {};
		\node [style=real] (17) at (4.5, -0.75) {};
		\node [style=real] (18) at (4, -0.5) {};
		\node [style=real] (19) at (4, -1) {};
		\node [style=real] (20) at (4, -1.5) {};
		\node [style=real] (21) at (4, -2.25) {};
		\node [style=real] (22) at (3, -2.75) {};
		\node [style=real] (23) at (6, -1) {};
		\node [style=real] (24) at (6.75, -1) {};
		\node [style=real] (25) at (6.75, -1.75) {};
		\node [style=real] (26) at (6, -1) {};
		\node [style=real] (27) at (-4, -1) {};
		\node [style=real] (28) at (-7, -0.75) {};
		\node [style=real] (29) at (-6, -1.25) {};
		\node [style=real] (30) at (-5.5, -0.75) {};
		\node [style=real] (31) at (-5.5, -1.5) {};
		\node [style=real] (32) at (-5, -1) {};
		\node [style=real] (33) at (-5, -2) {};
		\node [style=real] (34) at (-7, -2) {};
		\node [style=real] (35) at (-8, -0.5) {};
		\node [style=real] (36) at (-8.25, -1.25) {};
		\node [style=real] (37) at (-7.5, -1.25) {};
	\end{pgfonlayer}
	\begin{pgfonlayer}{edgelayer}
	\contourlength{1.5pt}
	\draw[draw=red, bend right=45, thick, fill=red!10] (8.center) to node[midway, yshift=-8pt, inner sep=3pt]{\contour{white}{\textcolor{red}{distance from $\rho$ $\leq n+\mathbf{d}$}}} (6.center);
		\draw[draw=blue, thick, fill=blue!10, bend right=50] (2.center) to node[midway, yshift=-8pt, inner sep=3pt]{\contour{white}{\textcolor{blue}{distance from $\rho$ $\leq n$}}} (4.center);
		\fill[blue!10] (13.center) to (11.center) to (9.center) to (0.center) to (10.center) to (12.center) [bend right=90, looseness=2.4] to (13.center);
		\fill[red!10] (2.center) to (17.center) to (21.center) to (16.center) to (2.center);
		\fill[red!10] (5.center) to (28.center) to (29.center) to (30.center) to (5.center);
		\draw [style=bla] (10) to (0);
		\draw [style=bla] (10) to (11);
		\draw [style=bla] (11) to (9);
		\draw [style=bla] (9) to (0);
		\draw [style=bla] (12) to (10);
		\draw [style=bla, bend right=90, looseness=2.25] (12) to (13);
		\draw [style=root, draw=blue] (13) to (11);
		\draw [style=bla] (14) to (10);
		\draw [style=bla] (13) to (15);
		\draw [style=bla] (12) to (13);
		\draw [style=bla] (16) to (2);
		\draw [style=bla] (2) to (17);
		\draw [style=bla] (18) to (16);
		\draw [style=bla] (18) to (17);
		\draw [style=bla] (17) to (19);
		\draw [style=bla] (19) to (16);
		\draw [style=bla] (16) to (20);
		\draw [style=bla] (20) to (17);
		\draw [style=bla] (21) to (17);
		\draw [style=bla] (21) to (16);
		\draw [style=bla] (22) to (21);
		\draw [style=bla] (23) to (7);
		\draw [style=bla] (23) to (24);
		\draw [style=bla] (23) to (25);
		\draw [style=bla] (27) to (4);
		\draw [style=bla] (5) to (28);
		\draw [style=bla] (28) to (29);
		\draw [style=bla] (29) to (30);
		\draw [style=bla] (30) to (5);
		\draw [style=bla] (29) to (31);
		\draw [style=bla] (31) to (32);
		\draw [style=bla] (31) to (33);
		\draw [style=bla] (31) to (34);
		\draw [style=bla] (34) to (28);
		\draw [style=bla] (36) to (35);
		\draw [style=bla] (37) to (36);
		\draw [style=bla] (37) to (28);
		\draw [style=bla] (28) to (35);
		\draw [style=geoleft] (6) to (5);
		\draw [style=geoleft] (5) to (4);
		\draw [style=geoleft] (4) to (3);
		\draw [style=geoleft] (3) to (0);
		\draw [style=root, draw=red] (0) to (1);
		\draw [style=geoleft] (1) to (2);
		\draw [style=geoleft] (2) to (7);
		\draw [style=geoleft] (7) to (8);
		\draw [style=geoleft] (6) to (left);
		\draw [style=geoleft] (8) to (right);
		\draw [dashed, <->] (0) -- (13) node[midway, right,yshift=6pt] {$\mathbf{d}$};
	\end{pgfonlayer}
\end{tikzpicture}
\caption{\label{fig:volume difference}\small{A representation of the balls $[\UIHPQ]_n$ and $[\UIHPQ]_{n+\mathbf{d}}$; notice that $[\UIHPQS]_n\subseteq [\UIHPQ]_{n+\mathbf{d}}$ (in fact, it is obtained from $[\UIHPQ]_{n+\mathbf{d}}$ by eliminating any vertices belonging to $\cup_{i\in\mathbb{Z}}\q_i$, with the exception of the $v_i$'s).}}
\end{figure}

\proof Again, we see $\UIHPQS$ as being constructed from the UIHPQ $\UIHPQ$ via the pruning procedure of~\cite{CMboundary}. We will show that
$$\lim_{n\to\infty}\P(|\#[\UIHPQ]_n-\#[\UIHPQS]_n|>\epsilon n^4)=0$$
for any positive $\epsilon$, which entails the proposition.

Thanks to the pruning construction, we have $|\#[\UIHPQ]_n-\#[\UIHPQS]_n|\leq \sum_{i=-a}^b|\q_i|+|\#[\UIHPQ]_{n+\mathbf{d}}-\#[\UIHPQ]_n|$, where $\mathbf{d}$ is the distance between the root vertex of $\UIHPQ$ and the infinite component, while $-a$ and $b$ are the minimum and maximum $i$ such that vertex $v_i$ is at distance $n$ from vertex $v_0$ (see Figure~\ref{fig:volume difference}).

Without loss of generality we may then restrict ourselves to the event $\mathbf{d}< \sqrt{n}$, which is asymptotically almost sure; on the other hand, we have $\P\left(|\#[\UIHPQ]_{n+\mathbf{d}}-\#[\UIHPQ]_{n}|>\frac{\epsilon}{2} n^4\bigm\vert \mathbf{d}<\sqrt{n}\right)\to 0$ as $n\to\infty$ by Proposition~\ref{prop:limitlawball}.

Also, we have shown in Proposition~\ref{prop:limitlawball for UIHPQS} that the probability of $\max\{a,b\}$ being at least $n^{2+\alpha}$ is infinitesimal in $n$ for all positive $\alpha$. Excluding the single (almost surely) finite quadrangulation $\q_0$ in order not to be forced to take the bias of its law into account, we are reduced to computing
$$\lim_{n\to\infty}\P\left(\sum_{0<|i|\leq n^{2+\alpha}}|\q_i|>\frac{\epsilon}{4}n^4\right),$$
where the $\q_i$'s are i.i.d.~free Boltzmann quadrangulations. Since the size of $\q_i$ can always be bounded from above by four times its area, we may use \eqref{eq:sizetail} to deduce that $\q_i$ is stochastically dominated by a random variable in the domain of attraction of the totally asymmetric 3/4-stable random variable. In particular it follows by standard estimates \cite{Sk57} that there is a constant $C$ such that
$$\lim_{n\to\infty}\P\left(\sum_{0<|i|\leq n^{2+2   \alpha}}|\q_i|> Cn^{4(2+\alpha)/3}\right)=0,$$
and we can choose $\alpha$ so that asymptotically $Cn^{4(2+2\alpha)/3}<\epsilon n^4/4$.
\endproof

\bibliographystyle{siam}
\bibliography{bibli}

\end{document}